\documentclass[a4paper, 10pt, twoside, notitlepage]{amsart}

\usepackage[utf8]{inputenc}
\usepackage[dvipsnames]{xcolor}
\usepackage{amsmath} 
\usepackage{amssymb} 
\usepackage{amsthm}
\usepackage{geometry}
\usepackage{graphicx}
\usepackage{mathtools}
\usepackage{esint}
\usepackage[shortlabels]{enumitem}
\usepackage{comment}
\usepackage[colorlinks=true,linkcolor=blue]{hyperref}

\usepackage{tikz}
\usetikzlibrary{decorations.pathreplacing,hobby,calc,math}
\usepackage{ifthen}

\theoremstyle{plain}
\newtheorem{thm}{Theorem}[section]
\newtheorem{cor}[thm]{Corollary}
\newtheorem{prop}[thm]{Proposition}
\newtheorem{lem}[thm]{Lemma}

\newtheorem*{rmk*}{Remark}

\newcommand {\R} {\mathbb{R}} 
 \newcommand {\N} {\mathbb{N}}
 
\newcommand {\p} {\partial}

\newcommand {\D} {\Delta}

\newcommand\restr[2]{{% we make the whole thing an ordinary symbol
  \left.\kern-\nulldelimiterspace % automatically resize the bar with \right
  #1 % the function
  \littletaller % pretend it's a little taller at normal size
  \right|_{#2} % this is the delimiter
  }}
\newcommand{\littletaller}{\mathchoice{\vphantom{\big|}}{}{}{}}

\numberwithin{equation}{section}
\DeclareMathOperator {\dist} {dist}

\DeclareMathOperator{\Id} {Id}
\DeclareMathOperator{\F} {\mathcal{F}}

\DeclareMathOperator{\supp}{supp}

\title[Stability Local Reduction]{On the transfer of stability from the local to the fractional anisotropic Calder\'on problem with exterior measurements}

\author[H. Baers]{Hendrik Baers}
\address{Institute for Applied Mathematics, University of Bonn, Endenicher Allee 60, 53115 Bonn, Germany}
\email{hendrik.baers@uni-bonn.de}

\author[A. R\"uland]{Angkana R\"uland}
\address{Institute for Applied Mathematics and Hausdorff Center for Mathematics, University of Bonn, Endenicher Allee 60, 53115 Bonn, Germany}
\email{rueland@uni-bonn.de}

\begin{document}

\begin{abstract}
We study the quantitative transfer of uniqueness from the classical to the fractional Calder\'on problem with exterior data. This allows us to deduce the first stability estimates for the principal part of the isotropic fractional Calder\'on problem with exterior data in the absence of Liouville transforms. Our argument relies on careful quantitative unique continuation and Runge approximation estimates. Due to the unbounded geometry and the mismatch of the dimensionalities of the measurement domains (exterior data on an open set vs boundary data on a co-dimension one manifold) novel challenges arise compared to the setting of source-to-solution measurements on closed manifolds.
\end{abstract}

\maketitle
%\tableofcontents

\section{Introduction}\label{sec:intro}

We consider the relation between the fractional and the classical Calderón problem with (exterior generalized) Dirichlet-to-Neumann data. More precisely, we reduce the fractional to the classical Calderón problem in a quantitative way, proving stability of this reduction. As a consequence, any stability result on the level of the classical Calderón problem can be transferred to a stability result for the fractional formulation. With this strategy we are able to derive novel stability results for fractional Calderón type problems. In particular, we infer the first stability results for the principal part for a fractional Calderón problem for which no Liouville-type reduction is known. 

This work builds on the qualitative relation of the classical and fractional Calder\'on problems established and analysed in \cite{CGRU23}. A related quantitative result on the reduction of a nonlocal to a local Calderón type problem was derived in \cite{BR25} in the context of closed manifolds and with source-to-solution measurements. In the present whole space setting with exterior (generalized) Dirichlet-to-Neumann measurements, novel challenges arise which we address in what follows below.

\subsection{Set-up}

Let us begin by recalling the problem set-up. 

In the \emph{classical Calderón problem} one seeks to determine the (possibly anisotropic) electrical conductivity of a medium by making voltage and current measurements on its boundary. In mathematical terms, let $\Omega \subset \R^n$ be open, bounded and sufficiently smooth and let $a \in L^\infty(\Omega, \R^{n \times n}_{\text{sym}})$ be uniformly elliptic and symmetric. The goal is to determine the unknown metric $a$ based on the knowledge of the Dirichlet-to-Neumann map $\Lambda_{1,\Omega}^a$,
\begin{align}\label{eq:Local_DN_map}
\Lambda_{1,\Omega}^a: H^{\frac{1}{2}}(\partial\Omega) \to H^{-\frac{1}{2}}(\partial\Omega), \qquad g \mapsto \restr{\partial_\nu^a v_a^g}{\partial\Omega} := \restr{\nu \cdot a \nabla' v_a^g}{\partial\Omega},
\end{align}
where $\nu$ is the outward unit normal on $\partial\Omega$, $\nabla'$ denotes the gradient in $\R^n$ and $v_a^g \in H^1(\Omega)$ is the unique weak solution to
\begin{equation}\label{eq:Local_Calderon_equation}
\begin{cases}
\begin{alignedat}{2}
-\nabla'\cdot a \nabla' v_a^g &= 0 \quad &&\text{in } \Omega,\\
v_a^g &= g \quad &&\text{on } \partial\Omega.
\end{alignedat}
\end{cases}
\end{equation}
This problem has been studied intensively. Landmark results for the isotropic setting were deduced in the works \cite{SU87,HT13,CR16,Hab15} on uniqueness, in \cite{A88} on stability and in \cite{N88, No88} on reconstruction. Moreover, partial data results were derived in \cite{KS13}. Let us emphasize that in the anisotropic setting, even questions about uniqueness and stability are widely open, with only some results known under extremely strong structural conditions such as analyticity \cite{LU89, LU01, LTU03}. We refer to the survey article \cite{Uhlmann09} and the recent book \cite{FSU25} for further references on the large literature on the classical Calder\'on problem.

Next, let us turn to the \emph{fractional Calderón problem}. For this, let $\Omega \subset \R^n$ be open, bounded and Lipschitz, such that $\R^n \setminus \Omega$ is connected and let $W \subset \Omega_e := \R^n \setminus \overline{\Omega}$ be open, bounded and Lipschitz such that $\overline{\Omega} \cap \overline{W} = \emptyset$. Let $a \in L^\infty(\R^n, \R^{n \times n}_{\text{sym}})$ be uniformly elliptic and symmetric with $a = \Id$ in $\Omega_e$ and let $s\in (0,1)$. In this set-up the objective is to recover the metric $a$ based on the knowledge of the fractional Dirichlet-to-Neumann map $\Lambda_s^a$,
\begin{align}\label{eq:Fractional_DN_map}
\Lambda_s^a: \widetilde{H}^{s}(W) \rightarrow H^{-s}(W), \qquad f \mapsto \restr{(-\nabla'\cdot a \nabla')^s u}{W},
\end{align}
where $u \in H^s(\R^n)$ is a solution to the equation
\begin{equation}\label{eq:Fractional_Calderon_equation}
\begin{cases}
\begin{alignedat}{2}
(-\nabla'\cdot a \nabla')^s u &= 0 \quad &&\text{in } \Omega,\\
u &= f \quad &&\text{in } \Omega_e.
\end{alignedat}
\end{cases}
\end{equation}
Here, the fractional (or nonlocal) operator $(-\nabla'\cdot a \nabla')^s$ can be interpreted in various equivalent ways, including spectral or variable coefficient Caffarelli-Silvestre type extension perspectives (see \cite{CS07} for the constant and \cite{ST10} for the variable coefficient case). In \cite{FGKRSU25}, building on \cite{FGKU21}, the uniqueness question for this fractional Calderón problem was solved even in the anisotropic setting.

In this work we will mostly take the perspective of the variable coefficient Caffarelli-Silvestre type extension. More precisely, by virtue of the variable coefficient Caffarelli-Silvestre extension perspective, the nonlocal operator $(-\nabla'\cdot a \nabla')^s$ can be realized as
\begin{align*}
(-\nabla'\cdot a \nabla')^s u = -c_{s} \lim_{x_{n+1}\to0} x_{n+1}^{1-2s} \partial_{n+1} \tilde{u} (\cdot,x_{n+1}),
\end{align*}
where $\tilde{u} \in \dot{H}^1(\R^{n+1}_+,x_{n+1}^{1-2s})$ is the weak solution to
\begin{equation}
\label{eq:CS_intro}
\begin{cases}
\begin{alignedat}{2}
-\nabla \cdot x_{n+1}^{1-2s} \tilde{a} \nabla \tilde{u} &= 0 \quad &&\text{in } \R^{n+1}_+,\\
\tilde{u} &= u \quad &&\text{on } \R^n \times \{0\}.
\end{alignedat}
\end{cases}
\end{equation}
Here, $\tilde{a} \in L^\infty(\R^{n+1}_+,\R^{(n+1)\times(n+1)}_{\text{sym}})$ is given as $\tilde{a} = \begin{pmatrix} a & 0\\ 0 & 1 \end{pmatrix}$. In what follows, we will refer to $\tilde{u}$ as the Caffarelli-Silvestre type extension of $u$. See Section \ref{sec:prel_CSExtension} for more details on this.

In \cite{CGRU23}, using the Caffarelli-Silvestre type interpretation, it is shown that the knowledge of the fractional Dirichlet-to-Neumann map $\Lambda_s^a$ allows one to uniquely recover the local Dirichlet-to-Neumann map $\Lambda_{1,\Omega}^a$. More precisely, it is proven that the set
\begin{align*}
\{ (\restr{v^f}{\partial\Omega}, \restr{(\partial_\nu^a v^f)}{\partial\Omega}) \in H^{\frac{1}{2}}(\partial\Omega) \times H^{-\frac{1}{2}}(\partial\Omega): \ f \in C_c^\infty(W), \ v^f \text{ as in } \eqref{eq:reduction_formula} \}
\end{align*}
is densely contained in the set
\begin{align*}
\{ (g, \restr{(\partial_\nu^a v_a^g)}{\partial\Omega}) \in H^{\frac{1}{2}}(\partial\Omega) \times H^{-\frac{1}{2}}(\partial\Omega): \ v_a^g \text{ solves \eqref{eq:Local_Calderon_equation}} \}.
\end{align*}
Here, $v^f$ is defined by
\begin{align}\label{eq:reduction_formula}
v^f := \int_0^\infty t^{1-2s} \tilde{u}^f(\cdot,t) dt,
\end{align}
where $\tilde{u}^f \in \dot{H}^1(\R^{n+1}_+,x_{n+1}^{1-2s})$ is the Caffarelli-Silvestre type extension of $u$ (i.e.~a solution of \eqref{eq:CS_intro}) and $u \in H^s(\R^n)$ is a solution to \eqref{eq:Fractional_Calderon_equation} with Dirichlet data $f \in C_c^\infty(W)$.

It is our objective to quantify this reduction argument under suitable a priori bounds on the metric $a$, and, building on this quantification, to derive novel stability estimates for the principal part of the operator.

\subsection{Main results}

In what follows we discuss our main results on the quantitative transfer of uniqueness from the local to the nonlocal context. For technical reasons, we will always assume that $n\geq3$ and we impose the following a-priori assumptions on the class of admissible metrics. We will work with slightly different assumptions in the anisotropic and isotropic settings, with the latter allowing for a modest relaxation of the conditions.
\\

\textbf{Assumption (A1):} Firstly, we assume that the admissible metrics $a \in L^\infty(\R^n, \R^{n \times n}_{\text{sym}})$ satisfy the following uniform ellipticity condition. Let $\theta_1 \in (0,1)$ be given and assume that for all admissible metrics $a$ (as bilinear forms) it holds that
\begin{align*}
\theta_1 \leq a \leq \theta_1^{-1}.
\end{align*}

This is a very mild standard condition assumed in essentially all Calder\'on type problems, encoding the uniform ellipticity of the problem.\\

\textbf{Assumption (A2):} Secondly, we impose a uniform lower bound on the Ricci-curvature of $(\R^n, a)$. Let $\theta_2 > 0$ and assume that for all admissible metrics $a$ it holds that
\begin{align*}
\operatorname{Ric}((\R^n,a)) \geq - \theta_2.
\end{align*}

We will make use of this condition in deducing suitable heat kernel estimates which in turn imply estimates for the associated Poisson kernel (see Section \ref{sec:heat}). As we will assume that our metrics agree with the identity matrix outside of a compact set, under $C^2$ regularity conditions on the metric, by compactness, the Ricci curvature is always bounded from below. By assuming condition (A2), we here restrict to the class of metrics with a uniform lower bound of this type.\\

\textbf{Assumption (A3):} Thirdly, we impose a flatness condition on the metrics $a$ outside of $\Omega$. More precisely, we assume that $a \in L^\infty(\R^n, \R^{n \times n}_{\text{sym}})$ with $a = \Id$ in $\Omega_e$.
\\

The condition ensures that outside of $\Omega$ the metric is known. Similar conditions had earlier been considered in the qualitative results \cite{CGRU23} and \cite{FGKRSU25}. It is expected that the condition here can be relaxed to be valid on the complement of a compact set (instead of in the complement of $\Omega$).\\

\textbf{Assumption (A4):} We assume that the metric $a$ is uniformly bounded in the $C^1$-norm near the boundary of $\Omega$. More precisely, let $\theta_3>0$ and let $\Omega'\Subset\Omega$. We assume that for all admissible metrics $a$ it holds that
\begin{align*}
\Vert a \Vert_{C^1(\Omega\setminus\Omega')} \leq \theta_3.
\end{align*}

This technical assumption will be sufficient to prove one  of our essential ingredients, a quantitative unique continuation argument, Proposition \ref{prop:QUCP_1}. However, in order to derive our main results  on stability we need to strengthen assumption (A4). In particular, when applying this proposition and a quantitative Runge approximation, Proposition \ref{prop:QRA}, we will make use of these stronger assumptions. We distinguish between the anisotropic and isotropic settings.\\

\textbf{Assumption (A4'):} In the case of anisotropic coefficients, we assume that the metric $a$ is known near the boundary of $\Omega$, i.e. there exists $\Omega'\Subset\Omega$ such that for two metrics $a_1$, $a_2$ from the class of all admissible metrics it holds that $a_1 = a_2$ in $\Omega \setminus \Omega'$.
\\

As we will discuss below, this condition represents a technical assumption allowing us to make use of Caccioppoli's inequality and to rely on an $L^2$-based quantitative Runge approximation result.\\

In the setting of isotropic coefficients it is possible to drop the condition that $a_1 = a_2$ in a neighbourhood of $\partial \Omega$. It suffices to require a certain amount of additional regularity instead.\\

\textbf{Assumption (A4''):} In the case of isotropic coefficients, we assume that the metric $a$ is uniformly bounded in the $C^2(\Omega)$-norm. Let $\theta_3>0$ and assume that for all admissible metrics $a$ it holds that
\begin{align*}
\Vert a \Vert_{C^2(\Omega)} \leq \theta_3.
\end{align*}

The reason why we can drop the assumption (A4') is that in the isotropic case it is possible to rely on a Liouville reduction. This reduction implies that $L^2$-Runge approximation results suffice without having to rely on Caccioppoli type estimates. However, since the Liouville reduction involves second order derivatives and since we will need a uniform $L^\infty$-bound of these, in our argument below, we require the $C^2$-boundedness assumption from (A4'').\\

If for given parameters $\theta_1, \theta_2, \theta_3$ and a domain $\Omega' \Subset \Omega$ the conditions (A1)-(A4) are satisfied for a metric $a :\R^n \rightarrow \R^{n\times n}_{\text{sym}}$, we write that $a \in \mathcal{A}(\theta_1, \theta_2, \theta_3, \Omega')$. If instead, for given parameters $\theta_1, \theta_2, \theta_3$ and a domain $\Omega'\Subset\Omega$ the metric $a :\R^n \rightarrow \R^{n \times n}_{\text{sym}}$ satisfies the conditions (A1)-(A3) and (A4') or (A4''), we write that $a \in \mathcal{A}'(\theta_1, \theta_2, \Omega')$ or $a \in \mathcal{A}''(\theta_1, \theta_2, \theta_3)$, respectively.

Under these assumptions, we will prove the following quantitative reduction results. Let us first turn to the anisotropic setting.

\begin{thm}\label{thm:quantitative_reduction_anisotropic}
Let $\Omega \subset \R^n$ be open, non-empty, bounded and Lipschitz such that $\R^n\setminus\Omega$ is connected and let $W \subset \Omega_e$ be open, bounded, non-empty and Lipschitz such that $\overline{\Omega} \cap \overline{W} = \emptyset$. Let $\Omega_1 \subset \R^n$ be open, bounded, Lipschitz such that $\Omega \Subset \Omega_1$. Let $\theta_1\in(0,1)$, $\theta_2>0$ and let $\Omega'\Subset\Omega$ be open, bounded, Lipschitz. Let $a_1, a_2 \in C^2(\R^n, \R^{n \times n}_{\text{sym}}) \cap \mathcal{A}'(\theta_1, \theta_2, \Omega')$. Let $\Lambda_{1,\Omega_1}^{a_j}$ and $\Lambda_s^{a_j}$, $j\in\{1,2\}$, be the local and the fractional Dirichlet-to-Neumann maps as in \eqref{eq:Local_DN_map} and \eqref{eq:Fractional_DN_map}, respectively.

There exist constants $\mu>0$ and $C>0$ such that if $\Vert \Lambda_s^{a_1} - \Lambda_s^{a_2} \Vert_{\widetilde{H}^s(W) \to H^{-s}(W)} \leq \frac{1}{2}$, then
\begin{align*}
\Vert \Lambda_{1,\Omega_1}^{a_1} - \Lambda_{1,\Omega_1}^{a_2} \Vert_{H^{\frac{1}{2}}(\partial\Omega_1) \to H^{-\frac{1}{2}}(\partial\Omega_1)} \leq C \log( \vert \log( \Vert \Lambda_s^{a_1} - \Lambda_s^{a_2} \Vert_{\widetilde{H}^s(W) \to H^{-s}(W)} ) \vert)^{-\frac{1}{\mu}}.
\end{align*}
The constants $\mu$ and $C$ only depend on $n$, $s$, $\Omega$, $\Omega_1$, $\Omega'$, $W$, $\theta_1$, $\theta_2$ and $\Vert a_1 \Vert_{C^1(\Omega\setminus\Omega')}$.
\end{thm}

Note that, for technical reasons and in order to formulate our main results in the anisotropic and isotropic settings in parallel, in Theorem \ref{thm:quantitative_reduction_anisotropic} we relate the fractional Dirichlet-to-Neumann map (with the sets $\Omega$ and $W$ as above) to the local Dirichlet-to-Neumann map for the larger set $\Omega_1$ instead of the set $\Omega$. As the metrics $a_j$ satisfy $a_j = \Id$ in $\Omega_e$, this does not require additional information. Indeed, by virtue of the assumption that $a_j = \Id$ in $\Omega_e$, we may also consider the classical Calder\'on problem on the domain $\Omega_1$ without requiring additional information compared to the problem on $\Omega$. In view of our main purpose -- the transfer of stability from the local Calder\'on problem to the nonlocal one (see Corollary \ref{cor:stability_fract_Calderon_isotropic} for an example of this for isotropic metrics) -- by the assumption that $a_1 = a_2$ in $\Omega_e$, there is no major difference whether Theorem \ref{thm:quantitative_reduction_anisotropic} is formulated with the local Dirichlet-to-Neumann map defined on $\Omega$ or defined on $\Omega_1$. 

We emphasize that the assumption (A4') which is included in the definition of our set of admissible anisotropic metrics would also allow us to prove a version of Theorem \ref{thm:quantitative_reduction_anisotropic} in the domain $\Omega$. Since in the isotropic setting we only require the regularity condition (A4'') (which however in turn forces us to work in the larger domain $\Omega_1$  in the isotropic setting when applying the Runge approximation result) instead of the structural condition that $a_1 = a_2$ in a neighbourhood of $\partial \Omega$, we have opted to formulate already the anisotropic theorem in the setting of the larger domain $\Omega_1$.

In the isotropic setting, we can relax the structural assumption (A4') that asserts that the metric in $\Omega$ is known near the boundary. Here, it suffices to assume the regularity condition (A4''), that we have uniform $C^2(\Omega)$-bounds on the metrics.

\begin{thm}\label{thm:quantitative_reduction_isotropic}
Let $\Omega \subset \R^n$ be open, non-empty, bounded and Lipschitz such that $\R^n\setminus\Omega$ is connected and let $W \subset \Omega_e$ be open, bounded, non-empty and Lipschitz such that $\overline{\Omega} \cap \overline{W} = \emptyset$. Let $\Omega_1 \subset \R^n$ be open, bounded, Lipschitz such that $\Omega \Subset \Omega_1$. Let $\theta_1\in(0,1)$, $\theta_2>0$, $\theta_3>0$. Let $a_1 = \gamma_1 \Id$ and $a_2 = \gamma_2 \Id$ be isotropic coefficients (i.e. $\gamma_1, \gamma_2: \R^n \to \R$) satisfying $a_1, a_2 \in C^2(\R^n, \R^{n \times n}_{\text{sym}}) \cap \mathcal{A}''(\theta_1, \theta_2, \theta_3)$. Let $\Lambda_{1,\Omega_1}^{a_j}$ and $\Lambda_s^{a_j}$, $j\in\{1,2\}$, be the local and the fractional Dirichlet-to-Neumann maps as in \eqref{eq:Local_DN_map} and \eqref{eq:Fractional_DN_map}, respectively.

There exist constants $\mu>0$ and $C>0$ such that if $\Vert \Lambda_s^{a_1} - \Lambda_s^{a_2} \Vert_{\widetilde{H}^s(W) \to H^{-s}(W)} \leq \frac{1}{2}$, then
\begin{align*}
\Vert \Lambda_{1,\Omega_1}^{a_1} - \Lambda_{1,\Omega_1}^{a_2} \Vert_{H^{\frac{1}{2}}(\partial\Omega_1) \to H^{-\frac{1}{2}}(\partial\Omega_1)} \leq C \log( \vert \log( \Vert \Lambda_s^{a_1} - \Lambda_s^{a_2} \Vert_{\widetilde{H}^s(W) \to H^{-s}(W)} ) \vert)^{-\frac{1}{\mu}}.
\end{align*}
The constants $\mu$ and $C$ only depend on $n$, $s$, $\Omega$, $\Omega_1$, $W$, $\theta_1$, $\theta_2$ and $\theta_3$.
\end{thm}

Compared to the anisotropic setting, we can relax the assumption (A4') in the isotropic setting as we are able to formulate an $L^2$-based Alessandrini-type identity in the isotropic setting by relying on a Liouville reduction on the level of the local equation (see the beginning of Section \ref{sec:proof_main_results_isotropic}). In contrast, in the anisotropic setting, we only obtain an $H^1$-based Alessandrini-type identity (see the beginning of Section \ref{sec:proof_main_results_anisotropic}). Since the quantitative Runge approximation (see Proposition \ref{prop:QRA}), one of the main ingredients in the proof, only yields an $L^2$-based approximation, this requires us to apply a Caccioppoli-type inequality, which results in the additional assumption that the metrics are known near the boundary for anisotropic coefficients. In the isotropic setting this is not necessary. However, the Liouville reduction comes at the cost of imposing a uniform $C^2(\Omega)$-bound on the coefficients $a_j$.

As a consequence of the reduction theorems formulated in Theorems \ref{thm:quantitative_reduction_anisotropic} and \ref{thm:quantitative_reduction_isotropic}, any stability result on the local level implies a stability estimate for the corresponding fractional problem. We derive the following triple-$\log$ stability estimate for the fractional \textit{isotropic} Calderón problem as a consequence of our quantitative reduction argument and the well-known stability estimate for the local Calderón problem from \cite{A88}. The triple-$\log$ is then a consequence of the double-$\log$ stability of the reduction and the logarithmic stability of the local Calderón problem.

\begin{cor}\label{cor:stability_fract_Calderon_isotropic}
Let $\Omega \subset \R^n$ be open, non-empty, bounded and Lipschitz such that $\R^n\setminus\Omega$ is connected and let $W \subset \Omega_e$ be open, bounded, non-empty and Lipschitz such that $\overline{\Omega} \cap \overline{W} = \emptyset$. Let $\theta_1\in(0,1)$, $\theta_2>0$, $\theta_3>0$. Let $a_1 = \gamma_1 \Id, a_2 = \gamma_2 \Id$ satisfy $a_1, a_2 \in C^2(\R^n, \R^{n \times n}_{\text{sym}}) \cap \mathcal{A}''(\theta_1, \theta_2, \theta_3)$. Assume that $\gamma_1, \gamma_2 \in C^2(\R^n) \cap H^{t+2}(\Omega)$ for some $t > \frac{n}{2}$ with $\Vert \gamma_j \Vert_{H^{t+2}(\Omega)} \leq M$ for some $M>0$, $j \in \{1,2\}$. Let $\Lambda_s^{a_j}$, $j\in\{1,2\}$, be the fractional Dirichlet-to-Neumann maps as in \eqref{eq:Fractional_DN_map}.

There exist $C>0$ and $\sigma>0$ such that if $\Vert \Lambda_s^{a_1} - \Lambda_s^{a_2} \Vert_{\widetilde{H}^s(W) \to H^{-s}(W)} \leq \frac{1}{2}$, then
\begin{align*}
\Vert a_1 - a_2 \Vert_{L^\infty(\R^n)} \leq C \left\vert \log\left( \log\left( \big\vert \log\big( \Vert \Lambda_s^{a_1} - \Lambda_s^{a_2} \Vert_{\widetilde{H}^s(W) \to H^{-s}(W)} \big) \big\vert \right) \right) \right\vert^{-\sigma}.
\end{align*}
Here, the constant $\sigma$ only depends on $n$ and $t$, and the constant $C$ only depends on $n$, $s$, $\Omega$, $W$, $\theta_1$, $\theta_2$,  $\theta_3$, $M$ and $\sigma$.
\end{cor}

This corollary provides the first quantitative stability result for a nonlocal equation in which for the nonlocal equation no Liouville reduction (as in \cite{C20}) is known.

Let us highlight that while we do not know whether the triple logarithmic estimate and the (double) logarithmic losses in passing from the local to the nonlocal Calder\'on problems are necessary, logarithmic losses also occur in the stability estimates for the passage from the full to the partial data Calder\'on problems. For instance, while the seminal results of Alessandrini \cite{A88} have an optimal logarithmic stability modulus \cite{M01,DCR03,KRS21}, the partial data stability results from \cite{CDSFR16} are only known with double logarithmic moduli. To the best of the authors' knowledge, it is not known, whether this additional logarithmic loss is necessary. Hence, as the fractional Calder\'on problem has a natural interpretation (for $s=1/2$) as a partial data problem for the classical Calder\'on problem \cite{LNZ24, R23}, it may be that some logarithmic loss cannot be avoided in case that an additional logarithmic loss in the partial data problem is necessary.

\subsection{Main ideas and outline of the argument}

We briefly sketch the ideas of the proof of Theorem \ref{thm:quantitative_reduction_anisotropic}.

Adopting the Caffarelli-Silvestre type extension perspective, the fractional Dirichlet-to-Neu\-mann maps $\Lambda_s^{a_j}: \widetilde{H}^s(W) \to H^{-s}(W)$ are given as (see also Section \ref{sec:prel_CSExtension} for more details on this)
\begin{align*}
\Lambda_s^{a_j}(f) = \restr{-c_s \lim_{x_{n+1}\to0} x_{n+1}^{1-2s} \partial_{n+1} \tilde{u}_j^f(\cdot,x_{n+1})}{W},
\end{align*}
where $\tilde{u}_j^f \in \dot{H}^1(\R^{n+1}_+,x_{n+1}^{1-2s})$ is the solution to the following extension problem
\begin{equation}\label{eq:CS}
\begin{cases}
\begin{alignedat}{2}
-\nabla\cdot x_{n+1}^{1-2s} \tilde{a}_j \nabla \tilde{u}_j^f & = 0 \quad &&\text{in } \R^{n+1}_+,\\
-c_s\lim\limits_{x_{n+1} \rightarrow 0} x_{n+1}^{1-2s} \partial_{n+1} \tilde{u}_j^f & = 0 \quad &&\text{on } \Omega \times \{0\},\\
\tilde{u}_j^f &= f \quad &&\text{on } \Omega_e \times \{0\},
\end{alignedat}
\end{cases}
\end{equation}
and $c_s \neq 0$.

Using this representation, we recall the qualitative result from \cite{CGRU23}: In \cite{CGRU23} it is proven that the fractional Dirichlet-to-Neumann map $\Lambda_s^{a_j}$ uniquely determines the local Dirichlet-to-Neumann map $\Lambda_{1,\Omega}^{a_j}: H^{\frac{1}{2}}(\partial\Omega) \to H^{-\frac{1}{2}}(\partial\Omega)$ with the reduction being given by the following procedure. Let $g_j := \restr{\int_0^\infty t^{1-2s} \tilde{u}_j^f(\cdot,t) dt}{\partial\Omega} \in H^{\frac{1}{2}}(\partial\Omega)$, then
\begin{align*}
\Lambda_{1,\Omega}^{a_j} (g_j) = \partial_\nu^{a_j} \restr{\left( \int_0^\infty t^{1-2s} \tilde{u}_j^f(\cdot,t) dt \right)}{\partial\Omega} \in H^{-\frac{1}{2}}(\partial\Omega),
\end{align*}
where $\partial_\nu^{a_j} v := \nu \cdot {a_j} \nabla v$ and $\nu$ is the outward unit normal on $\partial\Omega$. The functions $\tilde{u}_j^f$, in turn, can be constructed from the exterior boundary data $(f,\Lambda_{s}^{a_j}(f))$ and thus implies an associated density result.

Building on this qualitative result, the quantitative proof of Theorem \ref{thm:quantitative_reduction_anisotropic} then consists of two main ingredients. On the one hand, we  rely on a quantitative unique continuation argument (see Proposition \ref{prop:QUCP_1}), and on the other hand, we prove a quantitative Runge approximation result (see Proposition \ref{prop:QRA}). Assume for the following outline of arguments that $\Vert \Lambda_s^{a_1} - \Lambda_s^{a_2} \Vert_{\widetilde{H}^s(W) \to H^{-s}(W)}$ is small and let $\Omega_1\Supset\Omega$ be open, bounded and Lipschitz.

Let us turn to our first main ingredient, the \emph{quantitative unique continuation result}. Its result essentially asserts that for $f \in \widetilde{H}^s(W)$ and $g_j := \restr{\int_0^\infty t^{1-2s} \tilde{u}^f_j (\cdot,t) dt}{\partial\Omega}$, with $\tilde{u}_j^f$ solutions to \eqref{eq:CS} as above, it holds that for some $C>0$ and $\beta>0$
\begin{align*}
\Vert \restr{(g_1 - g_2)}{\partial\Omega} \Vert_{H^{\frac{1}{2}}(\partial\Omega)} &= \Big\Vert \restr{ \left( \int_0^\infty t^{1-2s} (\tilde{u}_1^f(\cdot,t) -\tilde{u}_2^f(\cdot,t)) dt \right)}{\partial\Omega} \Big\Vert_{H^{\frac{1}{2}}(\partial\Omega)}\\
&\leq C \vert \log( \Vert \Lambda_s^{a_1} - \Lambda_s^{a_2} \Vert_{\widetilde{H}^s(W) \to H^{-s}(W)}) \vert^{-\beta} \Vert f \Vert_{\widetilde{H}^s(W)}
\end{align*}
and that
\begin{align*}
\Vert \Lambda_{1,\Omega}^{a_1}(g_1) - \Lambda_{1,\Omega}^{a_2}(g_2) \Vert_{H^{-\frac{1}{2}}(\partial\Omega)} &= \Big\Vert \restr{ \partial_\nu^a \left( \int_0^\infty t^{1-2s} (\tilde{u}_1^f(\cdot,t) -\tilde{u}_2^f(\cdot,t)) dt \right)}{\partial\Omega} \Big\Vert_{H^{-\frac{1}{2}}(\partial\Omega)}\\
&\leq C \vert \log( \Vert \Lambda_s^{a_1} - \Lambda_s^{a_2} \Vert_{\widetilde{H}^s(W) \to H^{-s}(W)}) \vert^{-\beta} \Vert f \Vert_{\widetilde{H}^s(W)}.
\end{align*}
In other words, if the difference of the two fractional Dirichlet-to-Neumann maps is small then also the difference of the Cauchy data for the local problem is small (accompanied with a logarithmic loss). We prove this by first reducing the integral defining $g_j$ to a finite height integral and then using quantitative unique continuation arguments. More precisely, we choose $L>0$ large enough and $h>0$ small enough, depending on the modulus of continuity that we seek to prove, such that it is guaranteed that
\begin{align*}
&\Vert \int_0^h t^{1-2s} (\tilde{u}_1^f(\cdot,t) - \tilde{u}_2^f(\cdot,t)) dt \Vert_{H^1(\Omega_1 \setminus \Omega)} + \Vert \int_L^\infty t^{1-2s} (\tilde{u}_1^f(\cdot,t) - \tilde{u}_2^f(\cdot,t)) dt \Vert_{H^1(\Omega_1 \setminus \Omega)}\\
&\hspace{25mm} \leq C \vert \log( \Vert \Lambda_s^{a_1} - \Lambda_s^{a_2} \Vert_{\widetilde{H}^s(W) \to H^{-s}(W)} ) \vert^{-\beta} \Vert f \Vert_{\widetilde{H}^s(W)}.
\end{align*}
In order to estimate the finite height integral $\Vert \int_h^L t^{1-2s} (\tilde{u}_1^f -\tilde{u}_2^f) dt \Vert_{H^1(\Omega_1 \setminus \Omega)}$, we first note that the smallness condition for the difference in the nonlocal problem corresponds to smallness in the boundary data for the Caffarelli-Silvestre type extension $\tilde{u}_1^f - \tilde{u}_2^f$ (on $W \times \{0\}$). Using a quantitative boundary-bulk-unique continuation argument (see Proposition \ref{prop:bbucp}) we transfer the smallness on $W \times \{0\}$ into the bulk $W \times \R_+$. Then, by a three-balls-inequality argument (see Proposition \ref{prop:3ballsinequ2}) we are able to propagate the smallness upwards in space and towards $(\Omega_1 \setminus \Omega) \times (h,L)$. With this approach we also derive that
\begin{align*}
\Big\Vert \int_h^L t^{1-2s} (\tilde{u}_1^f(\cdot,t) - \tilde{u}_2^f(\cdot,t)) dt \Big\Vert_{H^1(\Omega_1 \setminus \Omega)} \leq C \vert \log( \Vert \Lambda_s^{a_1} - \Lambda_s^{a_2} \Vert_{\widetilde{H}^s(W) \to H^{-s}(W)} ) \vert^{-\beta} \Vert f \Vert_{\widetilde{H}^s(W)}.
\end{align*}
In addition to estimating the different parts of the integral, we have to estimate a related inhomogeneous term. Such a contribution comes into play since we consider $\partial\Omega$ as the boundary of $\Omega_1 \setminus \Omega \subset \Omega_e$ and as outside of $\Omega$, the integrated quantity $\int_0^\infty t^{1-2s} \tilde{u}_j^f(\cdot,t) dt$ only satisfies an inhomogeneous elliptic equation.

The second main ingredient, the \emph{quantitative Runge approximation} (see Proposition \ref{prop:QRA}), then allows us to infer smallness for the difference of the full local Dirichlet-to-Neumann maps, not only for the specific boundary data of the form $\restr{\int_0^\infty t^{1-2s} \tilde{u}_j^f(\cdot,t) dt}{\partial\Omega}$. More precisely, it formally states that for any $g \in H^{\frac{1}{2}}(\partial\Omega_1)$, $v_{a_j}^g \in H^1(\Omega_1)$ solving \eqref{eq:Local_Calderon_equation} in $\Omega_1$, and for any $\varepsilon > 0$, there exists $f\in\widetilde{H}^s(W)$ such that for some $C>0$ and $\mu>0$
\begin{align*}
\Big\Vert v_{a_j}^g - \int_0^\infty t^{1-2s} \tilde{u}_j^f(\cdot,t) dt \Big\Vert_{H^1(\Omega)} \leq \varepsilon \Vert g \Vert_{H^{\frac{1}{2}}(\partial\Omega_1)}, \qquad \Vert f \Vert_{\widetilde{H}^s(W)} \leq C e^{C\varepsilon^{-\mu}} \Vert g \Vert_{H^{\frac{1}{2}}(\partial\Omega_1)}.
\end{align*}
In other words, the quantitative Runge approximation provides a quantitative version of a density result. Also here, in the ``cost" of $\Vert f \Vert_{\widetilde{H}^s(W)}$ (i.e. the second bound in the above estimate) a logarithmic loss occurs. The proof of this result, in particular, combines arguments from \cite{RS18a} and \cite{CGRU23}.

We infer the main result of this work, Theorem \ref{thm:quantitative_reduction_anisotropic}, as a consequence of the above described two main ingredients. Note that both ingredients are accompanied with a logarithmic loss, leading eventually to the double logarithmic loss in our quantitative reduction result.

\subsection{Comparison to the source-to-solution setting}

Let us briefly comment on the similarities and differences between the quantitative reduction in this work which is set in the whole space with exterior Dirichlet-to-Neumann data and the quantitative reduction argument in the source-to-solution Calderón problem analysed in \cite{BR25} which is set on a closed manifold. 

While both arguments strongly rely on quantitative unique continuation, the unboundedness of the domain and the fact that we consider exterior rather than source-to-solution data creates novel challenges which need to be overcome.
This is most directly visible in our second main ingredient, the quantitative Runge approximation. Indeed, this result is an essentially completely novel ingredient in the reduction process. In \cite{BR25} this step was not needed, the quantitative unique continuation argument was the only essential ingredient. The reason for this is that the set-up of the fractional and the local source-to-solution Calderón problem are essentially ``the same'' and the data are taken on the same measurement set. In our setting of the (exterior) Dirichlet-to-Neumann map Calderón problem, however, the measurement set and data are different (even of different dimensionalities). This raises the need of an additional density argument to guarantee the smallness for the full measurement set. Here careful approximation arguments become necessary.
While this certainly is the most prominent difficulty, also additional challenges arise in the quantitative unique continuation argument. Although the quantitative unique continuation property in our present article is also based on propagation of smallness arguments as in \cite{BR25}, the passage from source-to-solution to exterior data also here requires additional care, making the argument somewhat more involved. 
In particular, the combination of the two ingredients involve substantial additional difficulties compared to \cite{BR25} which eventually also result in an additional logarithmic loss compared to the compact manifold setting. To be more precise, the second logarithmic loss is due to the additional quantitative density argument, which was not needed in the compact manifold setting.

\subsection{Relation to the literature}

Building on the seminal article \cite{GSU20}, the study of fractional Calder\'on type problems has developed into an extremely active field in the past years. The by now available results are characterized by the presence of genuinely nonlocal properties. For instance, these are reflected in partial data uniqueness results at critical regularity \cite{GSU20, RS20}, optimal partial data stability results \cite{RS20, RS18, BCR25}, uniqueness of potentials in the presence of known background metrics \cite{GLX17} or single measurement results \cite{GRSU20, R21} for fractional Calder\'on problems in the fractional Schrödinger formulation. All of these results represent properties which are in this generality either not known or not true in the local setting of the classical Calder\'on problem. In particular, many of these results rely on the global unique continuation properties of the fractional Laplacian \cite{GSU20}, see also \cite{FF14, R15, Y17} for frequency and Carleman approaches to these. For further striking results on fractional Calder\'on problems we point to \cite{HL19,HL20} for monotonicity methods, to \cite{LL22} for results on nonlinear equations, to \cite{C20} for the presence of a Liouville reduction and to \cite{S17, R18, C24, LL25} as well as the references therein for surveys on these and further results.

In addition to the study of uniqueness, stability and reconstruction of lower order contributions, a particularly striking property of fractional Calder\'on type problems is the unique identifiability of the principal order contribution up to natural gauges. This study had been initiated in the seminal works \cite{F21,FGKU21} in the setting of closed manifolds with source-to-solution measurements. Contrary to the classical Calder\'on problem in which the unique recovery of the metric remains an outstanding problem, in the nonlocal setting, using nonlocality, it is possible to prove uniqueness up to diffeomorphisms fixing the measurement domain. This strategy has by now been extended to various related geometric nonlocal problems, including Dirac operators and connections \cite{GU21,C23}, settings involving potentials \cite{FKU24}, unbounded manifolds with source-to-solution measurements \cite{CO24}, unbounded manifolds with exterior measurements \cite{FGKRSU25} as well as models for parabolic operators \cite{LLU23}. Moreover, also matching ``boundary reconstruction results'' \cite{CR25} and Caffarelli-Silvestre type perspectives \cite{R23,LNZ24} have been obtained.

Due to the strikingly strong results for the fractional Calder\'on problem compared to the classical Calder\'on problem, from the very beginning of the study of these nonlocal inverse problems, it had been a fundamental question to relate these two nonlinear, elliptic inverse problems. First qualitative results on this relation were deduced in \cite{GU21,CGRU23} and \cite{R23}. In these results it is proved that the nonlocal measurement data for the fractional Calder\'on problem (either in the form of the Dirichlet-to-Neumann map or the source-to-solution map) uniquely determine the corresponding local measurement data for the local, classical Calder\'on problem. As a consequence, uniqueness results from the local setting could be transferred to the nonlocal context, e.g., in recovering conformal factors. In order to investigate this relation in more detail, in recent work \cite{BR25} corresponding quantitative properties had been analyzed in the setting of closed manifolds with source-to-solution data. These allowed one to not only transfer uniqueness but also stability results from the local to the nonlocal context and thus to obtain the first stability results of principal terms in the absence of Liouville transforms. In the present article, it is our objective to continue this quantitative investigation in the setting of unbounded manifolds with exterior data. Due to the mismatch of the dimensionalities and geometric settings between the local and nonlocal Calder\'on problems (one posed with data in the exterior domain, one with data on the boundary of an open, sufficiently regular subset) this poses additional challenges which we address in our analysis below.

\subsection{Outline of the article}

The remainder of this article is structured as follows. In Section \ref{sec:preliminaries} we collect some preliminary results and introduce relevant notation for our analysis. Section \ref{sec:quant_UCP} lays out the arguments of the first main ingredient, the quantitative unique continuation argument, in detail. The second main ingredient, the quantitative Runge approximation, is then discussed in Section \ref{sec:QRA}. In Section \ref{sec:proof_main_results}, we then finally combine the results of the previous two sections to deduce the main results of this work. For the convenience of the reader, the main body of our text is complemented by an appendix, in which we recall and shortly prove some rather well-known results related to the fractional Calderón problem.

\section{Preliminaries}\label{sec:preliminaries}

In this section we collect some preliminary results and notation needed for our analysis.

\subsection{Function spaces}\label{sec:prel_FunctionSpaces}

We start by introducing the relevant function spaces. Let $s\in\R$. We define the fractional Sobolev spaces $H^s(\R^n)$ by
\begin{align*}
H^s(\R^n) := \{ u \in \mathcal{S}'(\R^n): \ \Vert u \Vert_{H^s(\R^n)} := \Vert (1+\vert\cdot\vert^2)^{s/2} \mathcal{F}u (\cdot) \Vert_{L^2(\R^n)} < \infty \},
\end{align*}
where $\mathcal{S}'(\R^n)$ denotes the space of tempered distributions and $\mathcal{F}u$ denotes the Fourier-transform applied to $u$. The homogeneous fractional Sobolev space $\dot{H}^s(\R^n)$ we define by
\begin{align*}
\dot{H}^s(\R^n) := \mbox{completion of } C_{c}^{\infty}(\R^n) \mbox{ with respect to }  [u]_{\dot{H}^s(\R^n)} := \Vert \vert \cdot \vert^s \mathcal{F}u(\cdot) \Vert_{L^2(\R^n)}. 
\end{align*}
For an open, bounded, Lipschitz set $W$ we define the fractional Sobolev spaces $H^s(W)$ as the quotient space
\begin{align*}
H^s(W) := \{ \restr{u}{W}: \ u \in H^s(\R^n) \},
\end{align*}
to which we associate the quotient norm $\Vert u \Vert_{H^s(W)} := \inf \{ \Vert U \Vert_{H^s(\R^n)}: \ U \in H^s(\R^n), \  \restr{U}{W} = u \}$. Additionally, we define the spaces
\begin{align*}
\widetilde{H}^s(W) := \overline{C_c^\infty(W)}^{H^s(\R^n)}, \qquad H^s_{\overline{W}} := \{ u \in H^s(\R^n):\ \supp(u) \subset \overline{W} \}
\end{align*}
and we recall that $\widetilde{H}^s(W) = H^s_{\overline{W}}$ for Lipschitz sets $W$ and all $s\in\R$ (see \cite[Theorem 3.29]{McLean}) and the following duality relations
\begin{align*}
(\widetilde{H}^s(W))^* = H^{-s}(W) \qquad \text{and} \qquad (H^s(W))^* = \widetilde{H}^{-s}(W).
\end{align*}
While $\widetilde{H}^s(W)$ is, by definition, naturally endowed with the $H^s(\R^n)$ norm, in order to highlight the function space, we also write $\|\cdot\|_{\widetilde{H}^s(W)}$.

Furthermore, for $s\in(0,1)$ and $\widetilde{\Omega} \subset \R^{n+1}_+$ open, Lipschitz, we define the following weighted Lebesgue and Sobolev-spaces
\begin{align*}
L^2(\widetilde{\Omega}, x_{n+1}^{1-2s}) &:= \{ \tilde{u}: \widetilde{\Omega} \to \R : \ \Vert x_{n+1}^{\frac{1-2s}{2}} \tilde{u} \Vert_{L^2(\widetilde{\Omega})} < \infty \},\\
H^1(\widetilde{\Omega}, x_{n+1}^{1-2s}) &:= \{ \tilde{u}: \widetilde{\Omega} \to \R : \ \Vert x_{n+1}^{\frac{1-2s}{2}} \tilde{u} \Vert_{L^2(\widetilde{\Omega})} + \Vert x_{n+1}^{\frac{1-2s}{2}} \nabla \tilde{u} \Vert_{L^2(\widetilde{\Omega})} < \infty \},\\
\dot{H}^1(\widetilde{\Omega}, x_{n+1}^{1-2s}) &:= \{ \tilde{u}: \widetilde{\Omega} \to \R : \ \Vert x_{n+1}^{\frac{1-2s}{2}} \nabla \tilde{u} \Vert_{L^2(\widetilde{\Omega})} < \infty \}.
\end{align*}
We remark that the trace operator $T: H^1(\R^{n+1}_+, x_{n+1}^{1-2s}) \to H^s(\R^n)$ such that $\tilde{u}(\cdot, x_{n+1}) \to T\tilde{u}$ in $L^2(\R^n)$ as $x_{n+1} \to 0$ is well-defined (see for example Lemma 4.4 in \cite{RS20}).

For the proof of the quantitative Runge approximation we also need to introduce the following compact support spaces. For $\Omega \subset \R^n$ open, bounded, Lipschitz and $\Omega_e:= \R^n \setminus \overline{\Omega}$,
\begin{align*}
H^1_c(\R^{n+1}_+, x_{n+1}^{1-2s}) &:= \{ \tilde{v} \in H^1(\R^{n+1}_+, x_{n+1}^{1-2s}): \tilde{v} \text{ has compact support in } \overline{\R^{n+1}_+}\},\\
H^1_{c,0}(\R^{n+1}_+, x_{n+1}^{1-2s}) &:= \{ \tilde{v} \in H^1(\R^{n+1}_+,x_{n+1}^{1-2s}): \ \tilde{v} \text{ has compact support in } \overline{\R^{n+1}_+}, \ \restr{v}{\Omega_e \times \{0\}} = 0 \}.
\end{align*}

In the context of these spaces, let us recall two essential inequalities which we will apply at different points throughout our analysis. The first one is a weighted Sobolev embedding theorem and the second one is a weighted Gehring Lemma. We recall that our weight $w(x) := x_{n+1}^{1-2s}$ is an $A_2$-Muckenhoupt weight (see Chapter 7 in \cite{Grafakos2014}). In particular, $w$ is a doubling weight. For $E \subset \R^n$ we will use the notation $w(E) := \int_E w \ dx$.

With this notation, the following Sobolev embedding holds.

\begin{prop}[Theorem 1.2 in \cite{FKS82}]
Let $1<p<\infty$ and assume that $w$ is an $A_p$-Muckenhoupt weight. There exist constants $C>0$ and $\delta>0$ such that for all balls $B_R$, all $u \in C_0^\infty(B_R)$ and all $k \in (1,\frac{n}{n-1}+\delta)$, it holds
\begin{align*}
\left( \frac{1}{w(B_R)} \int_{B_R} \vert u \vert^{kp} w \ dx \right)^{\frac{1}{kp}} \leq CR \left( \frac{1}{w(B_R)} \int_{B_R} \vert \nabla u \vert^p w \ dx \right)^{\frac{1}{p}}.
\end{align*}
\end{prop}

Moreover, also Gehring's lemma is valid in this weighted context.

\begin{lem}[Theorem 1.5 in \cite{Kin94}]\label{lem:Gehring}
Suppose that $w$ is a doubling weight and let $f \in L^p_{\text{loc}}(\Omega,w)$ for some $1<p<\infty$. Assume that $f$ satisfies for each cube $Q$ with $2Q \subset \Omega$
\begin{align*}
\left( \frac{1}{w(Q)} \int_Q \vert f \vert^p w \ dx \right)^{\frac{1}{p}} \leq C_1 \frac{1}{w(2Q)} \int_{2Q} \vert f \vert w \ dx
\end{align*}
for some $C_1>1$ independent of the cube $Q$. Then there exists $q>p$ such that
\begin{align*}
\left( \frac{1}{w(Q)} \int_Q \vert f \vert^q w \ dx \right)^{\frac{1}{q}} \leq C_2 \left( \frac{1}{w(2Q)} \int_{2Q} \vert f \vert^p w \ dx \right)^{\frac{1}{p}},
\end{align*}
where the constant $C_2$ is independent of the cube $Q$. In particular, $f \in L^q_{\text{loc}}(\Omega,w)$.
\end{lem}

\subsection{Caffarelli-Silvestre type extension}\label{sec:prel_CSExtension}

It was proved in \cite{CS07} (for the constant coefficient case) and in \cite{ST10} (for the variable coefficient case) that the nonlocal elliptic operator $(-\nabla'\cdot a \nabla')^s$ can be interpreted as a local operator at the cost of adding one dimension. More precisely, for $u \in C_c^{\infty}(\R^n)$ (and by density and a-priori estimates also for $u \in H^s(\R^n)$), the operator $(-\nabla'\cdot a \nabla')^s$ can be realized as
\begin{align*}
(-\nabla'\cdot a \nabla')^s u = -c_{s} \lim_{x_{n+1}\to0} x_{n+1}^{1-2s} \partial_{n+1} \tilde{u} (\cdot,x_{n+1}),
\end{align*}
where $\tilde{u} \in \dot{H}^1(\R^{n+1}_+,x_{n+1}^{1-2s})$ is the unique weak solution to
\begin{equation*}
\begin{cases}
\begin{alignedat}{2}
-\nabla \cdot x_{n+1}^{1-2s} \tilde{a} \nabla \tilde{u} &= 0 \quad &&\text{in } \R^{n+1}_+,\\
\tilde{u} &= u \quad &&\text{on } \R^n \times \{0\}.
\end{alignedat}
\end{cases}
\end{equation*}
Here, $\tilde{a} : \R^{n+1}_+ \rightarrow \R^{(n+1) \times (n+1)}_{sym}$ is given as $\tilde{a} = \begin{pmatrix} a & 0\\ 0 & 1 \end{pmatrix}$. We call $\tilde{u}$ the Caffarelli-Silvestre type extension of $u$. If now $u \in H^s(\R^n)$ is the unique solution to
\begin{equation*}
\begin{cases}
\begin{alignedat}{2}
(-\nabla'\cdot a \nabla')^s u &= 0 \quad &&\text{in } \Omega,\\
u &= f \quad &&\text{in } \Omega_e,
\end{alignedat}
\end{cases}
\end{equation*}
for $f\in\widetilde{H}^s(W)$, then using the Caffarelli-Silvestre type extension from above, $u$ is given as $u = \tilde{u}(\cdot,0)$, where $\tilde{u} \in \dot{H}^1(\R^{n+1}_+,x_{n+1}^{1-2s})$ is a solution to the mixed Dirichlet-Neumann problem
\begin{equation*}
\begin{cases}
\begin{alignedat}{2}
-\nabla \cdot x_{n+1}^{1-2s} \tilde{a} \nabla \tilde{u} &= 0 \quad &&\text{in } \R^{n+1}_+,\\
-c_{s} \lim_{x_{n+1}\to0} x_{n+1}^{1-2s} \partial_{n+1} \tilde{u} &= 0 \quad &&\text{on } \Omega\times\{0\},\\
\tilde{u} &= f \quad &&\text{on } \Omega_e \times \{0\}.
\end{alignedat}
\end{cases}
\end{equation*}

\subsection{Additional notation}

Lastly, we collect further notation used in this article. We denote by $\Omega_e$ the complement of $\Omega$, i.e. $\Omega_e := \R^n \setminus \overline{\Omega}$. We define $\R^{n+1}_+ := \{x=(x',x_{n+1}) \in \R^{n+1}: \ x_{n+1} >0\}$. The notation $x \in \R^{n+1}$ vs. $x' \in \R^n$ will be used throughout our article. In particular, we will write $\nabla$ to denote the gradient with respect to all $n+1$ variables, and $\nabla'$ to denote the tangential gradient, i.e. the gradient with respect to the first $n$ variables. Similarly, writing $B_r(x)$ we mean the $(n+1)$-dimensional ball with radius $r$ centered at $x$ and by writing $B_r'(x')$ we refer to the $n$-dimensional ball with radius $r$ centered at $x'$. For $x = (x',0)$ we denote by $B_r^+(x)$ the $(n+1)$-dimensional half ball centered at $x$, i.e. $B_r^+(x) := B_r(x) \cap \R^{n+1}_+$. If $x=0$ or $x'=0$, we occasionally also drop the center of the ball and write $B_r = B_r(0)$ and $B_r' = B_{r}'(0)$, respectively.

The normal derivative  with respect to the metric $a$ of a function $v$, $\partial_\nu^a v$, is defined by $\partial_\nu^a v := \nu \cdot a \nabla' v$, where $\nu$ is the outward unit normal at $\partial\Omega$. By $\tilde{a} \in L^\infty(\R^{n+1}_+, \R^{(n+1) \times (n+1)}_{\text{sym}})$ we denote the extension of $a \in L^\infty(\R^n, \R^{n \times n}_{\text{sym}})$ given by $\tilde{a}(x) = \begin{pmatrix} a(x') & 0\\ 0 & 1 \end{pmatrix}$.

Moreover, given a normed vector space $Y$, we denote by $Y^*$ the continuous dual space of $Y$. We write $\langle \cdot, \cdot \rangle_{Y^*,Y}$ for the dual pairing of $Y^*$ with $Y$. 
Additionally, for a given set $E$ we denote its characteristic function by $\chi_E(x) := \begin{cases} 1 \ \text{ for } x \in E,\\ 0 \ \text{ for } x \not\in E. \end{cases}$

For $A,B \geq 0$, we will write $A \lesssim B$ if there exists a constant $C>0$ such that $A \leq CB$, and we will write $A \sim B$, if $A \lesssim B$ and $B \lesssim A$.

In what follows, we will often use $C$ and $p_h$, $p_L$, $p_\delta$, $p_\omega$, $p$ as generic constants and exponents, respectively, meaning that they only depend on stated quantities but their precise value may change from appearance to appearance.

\section{Quantitative unique continuation}\label{sec:quant_UCP}

In this section we will prove the first important ingredient for deriving our main result. It is a quantitative unique continuation result for the Caffarelli-Silvestre type perspective of the fractional Calderón problem.

\begin{prop}\label{prop:QUCP_1}
Let $\Omega \subset \R^n$ be open, non-empty, bounded and Lipschitz such that $\R^n \setminus \Omega$ is connected and let $W \subset \Omega_e$ be open, bounded, non-empty and Lipschitz such that $\overline{\Omega} \cap \overline{W} = \emptyset$. Let $\theta_1\in(0,1)$, $\theta_2>0$, $\theta_3>0$ and let $\Omega'\Subset\Omega$ be open, Lipschitz. Let $a_1, a_2 \in L^\infty(\R^n, \R^{n \times n}_{\text{sym}}) \cap \mathcal{A}(\theta_1, \theta_2, \theta_3, \Omega')$. Let $\Lambda_s^{a_j}$, $j\in\{1,2\}$, be the fractional Dirichlet-to-Neumann maps as in \eqref{eq:Fractional_DN_map}.

Let $f \in \widetilde{H}^s(W)$ and let $\tilde{u}_j^f \in \dot{H}^1(\R^{n+1}_+, x_{n+1}^{1-2s})$, $j\in\{1,2\}$, satisfy
\begin{equation*}
\begin{cases}
\begin{alignedat}{2}
-\nabla\cdot x_{n+1}^{1-2s} \tilde{a}_j \nabla \tilde{u}_j^f &= 0 \quad &&\text{in } \R^{n+1}_+,\\
-c_s \lim_{x_{n+1}\to0} x_{n+1}^{1-2s} \partial_{n+1} \tilde{u}_j^f &= 0 \quad &&\text{on } \Omega \times \{0\},\\
\tilde{u}_j^{f} &= f \quad &&\text{on } \Omega_e \times \{0\}.
\end{alignedat}
\end{cases}
\end{equation*}
Then there exist constants $C>0$ and $\beta>0$ such that, if for some $\varepsilon \in (0,\frac{1}{2})$
\begin{align*}
\Vert \Lambda_s^{a_1} - \Lambda_s^{a_2} \Vert_{\widetilde{H}^s(W) \to H^{-s}(W)} \leq \varepsilon,
\end{align*}
then it holds that
\begin{align}\label{eq:UCP_final}
\begin{split}
&\left\Vert \int_0^\infty t^{1-2s} \left( \tilde{u}_1^{f} (\cdot,t) - \tilde{u}_2^{f} (\cdot,t) \right) dt \right\Vert_{H^{\frac{1}{2}}(\partial\Omega)} + \left\Vert \partial_\nu^a \left( \int_0^\infty t^{1-2s} \left( \tilde{u}_1^{f} (\cdot,t) - \tilde{u}_2^{f} (\cdot,t) \right) dt \right) \right\Vert_{H^{-\frac{1}{2}}(\partial\Omega)}\\
&\hspace{3cm} \leq C \vert\log(\varepsilon)\vert^{-\beta} \Vert f \Vert_{\widetilde{H}^s(W)}.
\end{split}
\end{align}
Here, $C$ and $\beta$ only depend on $n$, $s$, $\Omega$, $\Omega'$, $W$, $\theta_1$, $\theta_2$, and $\theta_3$.
\end{prop}

To prove this result, we will argue in several steps. First, we will split the integrals on the left hand side of \eqref{eq:UCP_final} into an upper, a lower and an intermediate integral. For each of these parts we will infer a corresponding stability estimate. We note that the mechanism in the proof for the upper and lower integrals is essentially different from the mechanism in the proof of the intermediate integral. The first one relies on the solution representation for the Caffarelli-Silvestre type extension equation via the Poisson-kernel, whereas the latter one works on the level of the equation and uses rather robust propagation of smallness arguments. It is the latter part of the argument which requires the most careful arguments.

For abbreviation, let us now write $v_{a_j}^{f} := \int_0^\infty t^{1-2s} \tilde{u}_j^f (\cdot,t) dt$ and let $\Omega_1 \Supset \Omega$. In addition to the above splitting of the integral, we will also need to estimate an inhomogeneous term which arises in $\Omega_1 \setminus \Omega$. Indeed, since by Theorem 3 in \cite{CGRU23} the function $v_{a_1}^f - v_{a_2}^f$ solves in $\Omega_1 \setminus \Omega$ (recall that $a_1 = a_2 = a$ in $\Omega_e$)
\begin{align*}
\nabla'\cdot a \nabla' (v_{a_1}^f - v_{a_2}^f) = (-\nabla'\cdot a_1 \nabla')^s u_1^f - (-\nabla'\cdot a_2 \nabla')^s u_2^f = -c_s \lim_{x_{n+1}\to0} x_{n+1}^{1-2s} \partial_{n+1} (\tilde{u}_1^f - \tilde{u}_2^f),
\end{align*}
in order to estimate $\Vert \partial_\nu^a (v_{a_1}^f - v_{a_2}^f) \Vert_{H^{-\frac{1}{2}}(\partial(\Omega_1 \setminus \Omega))}$ we also need to estimate the inhomogeneous term $\Vert \lim_{x_{n+1}\to0} x_{n+1}^{1-2s} \partial_{n+1} (\tilde{u}_1^f - \tilde{u}_2^f) \Vert_{H^{-1}(\Omega_1\setminus\Omega)}$.

\begin{rmk*}
Note that for Proposition \ref{prop:QUCP_1} we do not assume that the metrics $a_j$ are isotropic. In particular, the quantitative unique continuation result is the same for isotropic and anisotropic conductivities. The stronger assumption (A4') for the anisotropic than the assumption (A4'') for the isotropic setting is due to the application of the second main ingredient, the quantitative Runge approximation, Proposition \ref{prop:QRA}.

Let us also remark that, as one can see from the proof, the result of Proposition \ref{prop:QUCP_1} can also be stated as a single measurement result. More precisely, it would suffice to assume that for some non-trivial $f \in \widetilde{H}^s(W)$
\begin{align*}
\Vert (\Lambda_s^{a_1} - \Lambda_s^{a_2}) f \Vert_{H^{-s}(W)} \leq \varepsilon \Vert f \Vert_{\widetilde{H}^s(W)}
\end{align*}
to arrive at an analogous conclusion. In order to more easily combine this result with the second main ingredient, the quantitative Runge approximation, we opted to state it in the above infinite measurement form.

Moreover, in Proposition \ref{prop:QUCP_1} it would suffice to assume that $n\geq2$, instead of $n\geq3$.
\end{rmk*}

In the following subsections, we now begin to estimate the individual ingredients which eventually lead to the proof of Proposition \ref{prop:QUCP_1}.

\subsection{Estimates for the upper and lower integrals}\label{sec:heat}

We start by estimating the upper and lower integrals in the decomposition of the left hand side of \eqref{eq:UCP_final}. Here, for $f \in \widetilde{H}^s(W)$ let $u \in H^s(\R^n)$ be a solution to the fractional conductivity equation
\begin{equation}\label{eq:fract_cond}
\begin{cases}
\begin{alignedat}{2}
(- \nabla' \cdot a \nabla')^s u &= 0 \quad &&\text{in } \Omega,\\
u &= f \quad &&\text{in } \Omega_e := \R^n \setminus \overline{\Omega},
\end{alignedat}
\end{cases}
\end{equation}
and let $\tilde{u}^f \in \dot{H}^1(\R^{n+1}_+, x_{n+1}^{1-2s})$ be the solution to the Caffarellli-Silvestre extension formulation of \eqref{eq:fract_cond}
\begin{equation}\label{eq:CS_fract_cond}
\begin{cases}
\begin{alignedat}{2}
-\nabla \cdot x_{n+1}^{1-2s} \tilde{a} \nabla\tilde{u}^{f} &= 0 \quad &&\text{in } \R^{n+1}_+,\\
-c_s \lim_{x_{n+1} \to 0} x_{n+1}^{1-2s} \partial_{n+1} \tilde{u}^{f} &= 0 \quad &&\text{on } \Omega\times\{0\},\\
\tilde{u}^{f} &= f \quad &&\text{on } \Omega_e \times \{0\}.
\end{alignedat}
\end{cases}
\end{equation}
Under these conditions the following decay properties in the vertical direction hold.

\begin{lem}\label{lem:apriori_decay}
Let $\Omega \subset \R^n$ be open, non-empty, bounded and Lipschitz and let $W \subset \Omega_e$ be open, bounded, non-empty and Lipschitz. Let $\theta_1\in(0,1)$, $\theta_2>0$. Let $a \in L^\infty(\R^n, \R^{n \times n}_{\text{sym}})$ satisfy the assumptions (A1) and (A2) with the given $\theta_1$ and $\theta_2$, respectively. Let $f \in \widetilde{H}^{s}(W)$, let $u \in H^s(\R^n)$ be the solution to \eqref{eq:fract_cond} with Dirichlet data $f$, and let $\tilde{u}^f \in \dot{H}^1(\R^{n+1}_+, x_{n+1}^{1-2s})$ be the Caffarelli-Silvestre type extension of $u$. In particular, $\tilde{u}^f$ is a solution to \eqref{eq:CS_fract_cond} with boundary data $f$.

Then, there exist constants $C_0,C_1 > 0$ such that for all $x_{n+1} \geq 1$ and $x' \in \R^n$ it holds that
\begin{align*}
\vert \tilde{u}^f(x',x_{n+1}) \vert \leq C_0 x_{n+1}^{-n} \Vert u \Vert_{L^1(\R^n)} \quad \text{and} \quad \vert \nabla' \tilde{u}^f(x',x_{n+1}) \vert \leq C_1 x_{n+1}^{-n} \Vert u \Vert_{L^1(\R^n)},
\end{align*}
Moreover, if $1 \leq p,q,r \leq \infty$ are such that $1+\frac{1}{r} = \frac{1}{p}+\frac{1}{q}$, then $\tilde{u}^f$ also satisfies
\begin{align*}
\Vert \partial_{n+1} \tilde{u}^f(\cdot,x_{n+1}) \Vert_{L^r(\R^n)} \leq C_2 x_{n+1}^{\frac{n}{p}-n-1} \Vert u \Vert_{L^q(\R^n)}.
\end{align*}
Here the constants $C_0,C_1$ and $C_2$ only depend on $n$, $s$, $\theta_1$ and $\theta_2$.
\end{lem}

\begin{rmk*}
We remark that for the sharp gradient estimate, we expect a decay of the rate $x_{n+1}^{-(n+1)}$ (c.f. Lemma 6.2 in \cite{CGRU23}). As the slightly weaker decay estimate from above suffices for our analysis and as it actually does not affect the final result up to possibly a constant factor, we have opted to work with the above (slightly suboptimal) estimate.

Moreover, we note that the defining equation for $u$, \eqref{eq:fract_cond}, is not used in the proof. We only use that $\tilde{u}^f$ is the Caffarelli-Silvestre type extension of $u$. Thus, the statement holds for any $u$ and $\tilde{u}$, with the latter one being the Caffarelli-Silvestre type extension of $u$ (see equation \eqref{eq:CS_intro}).
\end{rmk*}

\begin{proof}
The proof follows along the lines of the proof of Lemma 6.2 in \cite{CGRU23}. By Corollary 3.2.8 in \cite{DaviesBook1989}, for any $\delta>0$ and $x',z' \in \R^n$ it holds that
\begin{align}\label{eq:estimate_heat_kernel}
0 \leq K_t(x',z') \leq c_{\delta,\theta_1} t^{-\frac{n}{2}} \exp\left( - \frac{\vert x'-z' \vert^2}{4(1+\delta) \theta_1^{-1} t} \right)
\end{align}
and for the gradient, by Theorem 6, Case 1, in \cite{D89},
\begin{align*}
\vert \nabla_x K_t(x',z') \vert \leq c_{\delta,n} t^{-\frac{n}{2}} (t^{-1} + \theta_2)^{\frac{1}{2}} \exp\left( - \frac{d_g(x',z')^2}{4(1+\delta)t} \right).
\end{align*}
Here, $d_g(\cdot,\cdot)$ denotes the Riemannian distance with respect to the manifold $(M,g) = (\R^n, a)$. We easily verify that $\sqrt{\theta_1} \vert x'-z' \vert \leq d_g(x',z') \leq \sqrt{\theta_1^{-1}} \vert x'-z' \vert$, and thus the latter estimate turns into
\begin{align}\label{eq:estimate_heat_kernel_gradient}
\vert \nabla_x K_t(x',z') \vert \leq c_{\delta,n} t^{-\frac{n}{2}} (t^{-1} + \theta_2)^{\frac{1}{2}} \exp\left( - \bar{c}_{\delta,\theta_1} \frac{\vert x'-z' \vert^2}{t} \right).
\end{align}
By Theorem 2.1 in \cite{ST10}, for $x' \in \R^n, y \in \R_+$, a solution to the Caffarelli-Silvestre extension is given by
\begin{align*}
\tilde{u}(x',y) = c_s y^{2s} \int_{\R^n} \int_0^\infty K_t(x',z') e^{-\frac{y^2}{4t}} \frac{dt}{t^{1+s}} u(z') dz'.
\end{align*}
We will apply \eqref{eq:estimate_heat_kernel} and \eqref{eq:estimate_heat_kernel_gradient} and use that
\begin{align*}
\int_0^\infty \frac{e^{-\frac{A}{4t}}}{t^{\frac{n}{2}+\sigma+1}} dt \sim A^{-\frac{n}{2}-\sigma}.
\end{align*}
On the one hand, we then infer that for $x' \in \R^n, y \in \R_+$
\begin{align*}
\vert \tilde{u}^f(x',y) \vert &\leq C y^{2s} \int_{\R^n} \int_0^\infty \vert K_t(x',z') \vert e^{-\frac{y^2}{4t}} \frac{dt}{t^{1+s}} \vert u(z') \vert dz'\\
&\leq C y^{2s} \int_{\R^n} \vert u(z') \vert \int_0^\infty t^{-(\frac{n}{2}+1+s)} e^{-(c\frac{\vert x'-z' \vert^2}{t} + \frac{y^2}{4t})} dt dz'\\
&\leq C y^{2s} \int_{\R^n} \frac{\vert u(z') \vert}{(\vert x'-z' \vert^2 + y^2)^{\frac{n}{2}+s}} dz'\\
&=C (\vert u \vert \ast T_{0,y}) (x'),
\end{align*}
where, for $\tau \in \R, x' \in \R^n, y \in \R_+$, $T_{\tau,y}(x') := \frac{y^{2s}}{(\vert x' \vert^2 + y^2)^{\frac{n}{2}+s+\tau}}$. By Young's convolution inequality we then obtain
\begin{align*}
\vert \tilde{u}^f(x',y) \vert \leq Cy^{-n} \Vert u \Vert_{L^1(\R^n)} \quad \text{for } x' \in \R^n, y \in \R_+,
\end{align*}
where $C$ depends on $n$, $s$, $\theta_1$. On the other hand, we derive similarly
\begin{align*}
\vert \nabla' \tilde{u}^f(x',y) \vert &\leq C y^{2s} \int_{\R^n} \int_0^\infty \vert \nabla_{x'} K_t(x',z') \vert e^{-\frac{y^2}{4t}} \frac{dt}{t^{1+s}} \vert u(z') \vert dz'\\
&\leq C y^{2s} \int_{\R^n} \vert u(z') \vert \int_0^\infty (t^{-(\frac{n+1}{2}+1+s)} + t^{-(\frac{n}{2}+1+s)}) e^{-(c\frac{\vert x'-z' \vert^2}{t} + \frac{y^2}{4t})} dt dz'\\
&\leq C y^{2s} \int_{\R^n} \frac{\vert u(z') \vert}{(\vert x'-z' \vert^2 + y^2)^{\frac{n+1}{2}+s}} dz + C y^{2s} \int_{\R^n} \frac{\vert u(z') \vert}{(\vert x'-z' \vert^2 + y^2)^{\frac{n}{2}+s}} dz'\\
&=C (\vert u \vert \ast T_{\frac{1}{2},y}) (x') + C (\vert u \vert \ast T_{0,y}) (x')\\
&\leq C y^{-(n+1)} \Vert u \Vert_{L^1(\R^n)} + C y^{-n} \Vert u \Vert_{L^1(\R^n)}\\
&\leq C y^{-n} \Vert u \Vert_{L^1(\R^n)},
\end{align*}
where for the last inequality we have used that $y \geq 1$. Here, the constant $C$ now depends on $n$, $s$, $\theta_1$ and $\theta_2$.

For the normal derivative we observe similarly
\begin{align*}
\vert \partial_y \tilde{u}(x',y) \vert &\leq C y^{2s-1} \int_{\R^n} \int_0^\infty \vert K_t (x',z') \vert e^{-\frac{y^2}{4t}} \frac{dt}{t^{1+s}} \vert u(z') \vert dz'\\
&\hspace{15mm} + y^{2s+1} \int_{\R^n} \int_0^\infty \vert K_t(x',z') \vert e^{-\frac{y^2}{4t}} \frac{dt}{t^{2+s}} \vert u(z') \vert dz'\\
&\leq C y^{-1} (\vert u \vert \ast T_{0,y})(x') + y (\vert u \vert \ast T_{1,y})(x').
\end{align*}
We apply Young's convolution inequality with $1+\frac{1}{r} = \frac{1}{p}+\frac{1}{q}$, $1\leq r,p,q \leq \infty$, to get
\begin{align*}
\Vert \partial_{y} \tilde{u}(\cdot,y) \Vert_{L^r(\R^n)} \leq \Big( y^{-1} \Vert T_{0,y}(\cdot) \Vert_{L^p(\R^n)} + y \Vert T_{1,y}(\cdot) \Vert_{L^p(\R^n)} \Big) \Vert u \Vert_{L^q(\R^n)},
\end{align*}
and by a change of variables $\bar{r}= \frac{r}{y}$ we find that
\begin{align*}
\Vert T_{\tau,y}(\cdot) \Vert_{L^p(\R^n)}^p \sim \int_0^\infty \frac{y^{2sp} r^{n-1}}{(r^2 + y^2)^{p(\frac{n}{2}+s+\tau)}} dr = \int_0^\infty \frac{y^{n-pn-2p\tau} \bar{r}^{n-1}}{(\bar{r}^2 + 1)^{p(\frac{n}{2}+s+\tau)}} d\bar{r} \leq C y^{n-pn-2p\tau}.
\end{align*}
Consequently, it holds that
\begin{align*}
\Vert \partial_y \tilde{u}(\cdot,y) \Vert_{L^r(\R^n)} \leq C y^{\frac{n}{p}-n-1} \Vert u \Vert_{L^q(\R^n)},
\end{align*}
which finishes the proof.
\end{proof}

Relying on the above bounds, we estimate two parts of the integrals which will be studied in what follows below. Using the estimates from above, we show that for suitably chosen parameters $L, h$ both will be well controlled. Thus, in the remainder of our argument, it suffices to study the intermediate integral between these two parameters.

\begin{lem}\label{lem:apriori_fract_cond}
Let $\Omega, W \subset \R^n$, $\theta_1 \in (0,1)$, $\theta_2>0$, $a \in L^\infty(\R^n,\R^{n \times n}_{\text{sym}})$, $f \in \widetilde{H}^{s}(W)$, $u \in H^s(\R^n)$ and $\tilde{u}^f \in \dot{H}^1(\R^{n+1}_+, x_{n+1}^{1-2s})$ be as in Lemma \ref{lem:apriori_decay}. Let $D \subset \R^n$ be some bounded, smooth domain.

Then, there exists a constant $C>0$ such that for $L \geq 1$ it holds that
\begin{align*}
\left\| \int_{L}^{\infty} t^{1-2s} \tilde{u}^f(\cdot,t) dt \right\|_{H^1(D)} \leq C L^{2-n-2s} \|f\|_{\widetilde{H}^s(W)}.
\end{align*}
Moreover, there exists $C>0$ such that for $h \leq 1$
\begin{align*}
\left\| \int_{0}^{h} t^{1-2s} \tilde{u}^f(\cdot,t) dt \right\|_{H^1(D)} \leq C h^{1-s} \|f\|_{\widetilde{H}^s(W)}.
\end{align*}
Here, the constants $C$ only depend on $n$, $s$, $\Omega$, $W$, $\theta_1$, $\theta_2$ and $D$.
\end{lem}

\begin{proof}
We begin with the upper integral. Here we will apply the first two estimates from Lemma \ref{lem:apriori_decay} asserting that for $x' \in \R^n, y \in \R_+$
\begin{align*}
\vert \tilde{u}^f (x',y) \vert \leq C y^{-n} \Vert u \Vert_{L^1(\R^n)}, \qquad \vert \nabla'\tilde{u}^f(x',y) \vert \leq C y^{-n} \Vert u \Vert_{L^1(\R^n)}.
\end{align*}
With this at hand we estimate
\begin{align*}
&\left\Vert \int_L^\infty t^{1-2s} \tilde{u}^f (\cdot,t) dt \right\Vert_{L^2(D)}^2 = \int_D \left\vert \int_L^\infty t^{1-2s} \tilde{u}^{f}(x',t) dt \right\vert^2 dx' \leq \int_D \left\vert \int_L^\infty t^{1-2s-n} \Vert u \Vert_{L^1(\R^n)} dt \right\vert^2 dx'\\
&\qquad \leq c_{n,s,D} L^{2(2-2s-n)} \Vert u \Vert_{L^2(\R^n)}^2 \leq c_{n,s,D} L^{2(2-2s-n)} \Vert f \Vert_{\widetilde{H}^s(W)}^2,
\end{align*}
where for the second to last inequality we have used that $\supp(u) \subset \overline{\Omega} \cup \overline{W} \subset B_R(0)$ for $R>0$ sufficiently large and thus $\Vert u \Vert_{L^1(\R^n)} \leq C \Vert u \Vert_{L^2(\R^n)}$, where the constant $C$ depends on $\Omega$ and $W$. Secondly, we estimate in the same way
\begin{align*}
&\left\Vert \nabla' \left( \int_L^\infty t^{1-2s} \tilde{u}^f (\cdot,t) dt \right) \right\Vert_{L^2(D)}^2 = \int_D \left\vert \int_L^\infty t^{1-2s} \nabla' \tilde{u}^{f}(x',t) dt \right\vert^2 dx' \\
& \qquad  \leq \int_D \left\vert \int_L^\infty t^{1-2s-n} \Vert u \Vert_{L^1(\R^n)} dt \right\vert^2 dx'
\leq C L^{2(2-2s-n)} \Vert u \Vert_{L^2(\R^n)}^2\\
& \qquad  \leq C L^{2(2-2s-n)} \Vert f \Vert_{\widetilde{H}^s(W)}^2.
\end{align*}
Combining the two previous inequalities yields the estimate for the upper integral.

For the lower integral we use Hölder's inequality and get
\begin{align*}
&\left\Vert \int_0^h t^{1-2s} \tilde{u}^f (\cdot,t) dt \right\Vert_{L^2(D)}^2 = \int_D \left\vert \int_0^h t^{1-2s} \tilde{u}^f(x',t) dt \right\vert^2 dx'\\
& \qquad \leq \left\Vert t^{\frac{1-2s}{2}} \right\Vert_{L^2((0,h))}^2 \int_D \int_0^h t^{1-2s} \vert \tilde{u}^f(x',t) \vert^2 dt dx' \leq C h^{2-2s} \left\Vert x_{n+1}^{\frac{1-2s}{2}} \tilde{u}^f \right\Vert_{L^2(D\times(0,h))}^2\\
&\qquad \leq C h^{2-2s} h^{2(2-2s)(\frac{1}{2}-\frac{1}{p})} \Vert x_{n+1}^{\frac{1-2s}{2}} \nabla \tilde{u}^f \Vert_{L^2(\R^{n+1}_+)}^2\\
&\qquad \leq C h^{(2-2s)(2-\frac{2}{p})} \Vert f \Vert_{\widetilde{H}^s(W)}^2.
\end{align*}
For the second to last inequality we have used that for $p\in(2,\infty)$ being the Sobolev exponent associated to the Muckenhoupt weight $x_{n+1}^{\frac{1-2s}{2}}$ (c.f. Theorem 1.3 in \cite{FKS82}) and for $q$ satisfying $\frac{1}{p} + \frac{1}{q} = \frac{1}{2}$ we have
\begin{align*}
\Vert x_{n+1}^{\frac{1-2s}{2}} \tilde{u}^f \Vert_{L^2(D\times(0,h))} &\leq \Vert x_{n+1}^{\frac{1-2s}{q}} \Vert_{L^q(D\times(0,h))} \Vert x_{n+1}^{\frac{1-2s}{p}} \tilde{u}^f \Vert_{L^p(D\times(0,h))}\\
&\leq C h^{(2-2s)(\frac{1}{2}-\frac{1}{p})} \Vert x_{n+1}^{\frac{1-2s}{2}} \nabla \tilde{u} \Vert_{L^2(\R^{n+1}_+)}.
\end{align*}
Moreover, for the gradient, we argue similarly
\begin{align*}
\left\Vert \nabla'\left( \int_0^h t^{1-2s} \tilde{u}^f dt \right) \right\Vert_{L^2(D)}^2 &= \int_D \left\vert \int_0^h t^{1-2s} \nabla'\tilde{u}^f(x',t) dt \right\vert^2 dx' \leq C h^{2-2s} \Vert x_{n+1}^{\frac{1-2s}{2}} \nabla'\tilde{u}^f \Vert_{L^2(D\times(0,h))}^2\\
&\leq C h^{2-2s} \Vert f \Vert_{\widetilde{H}^s(W)}^2.
\end{align*}
Combining the estimates for the $L^2$- and the $\dot{H}^1$-norms yields the result for the lower integral.
\end{proof}

\subsection{Estimate for intermediate integral}

In this subsection we will derive an estimate for the intermediate integral on $(\Omega_1 \setminus \Omega) \times (h,L)$, where $\Omega_1$ is some open, bounded, Lipschitz set satisfying $\Omega\Subset\Omega_1$ and $\overline{\Omega_1} \cap \overline{W} = \emptyset$. The main ingredients for the proof will be a quantitative boundary-bulk unique continuation result and a three-balls-inequality. However, to create an offset from $\Omega \times (h,L)$ we will first derive an estimate for the $H^1$-norm of a thin set surrounding $\Omega \times (h,L)$.

\subsubsection{Estimate of the weighted $H^1$-norm on a thin annulus around $\Omega \times (h,L)$}

We aim to prove an estimate for the weighted $H^1$-norm on a thin annulus around $\Omega \times (h,L)$, i.e. on $(\Omega_{+\delta} \setminus \Omega) \times (h,L)$ for some set $\Omega_{+\delta} \Supset \Omega$ with $\dist(\Omega, \partial\Omega_{+\delta}) \sim \delta$. For this we mainly need to observe that $\tilde{u}_j^f$ solves a uniformly elliptic partial differential equation in $\Omega_{1} \times (h,L)$ (with coefficients which depend on $h,L$) and thus $\tilde{u}_j^f$ enjoys higher regularity properties in $\Omega_{1} \times (h,L)$. We start by proving the following lemma.

\begin{lem}\label{lem:estimate_second_order_grad}
Let $\Omega \subset \R^n$ be open, non-empty, bounded and Lipschitz and let $W \subset \Omega_e$ be open, bounded, non-empty and Lipschitz. Let $\theta_1\in(0,1)$. Let $a \in L^\infty(\R^n, \R^{n \times n}_{\text{sym}})$ satisfy the assumption (A1) with the given $\theta_1$. Let $f \in \widetilde{H}^{s}(W)$ and let $\tilde{u}^f \in \dot{H}^1(\R^{n+1}_+, x_{n+1}^{1-2s})$ be the solution to the Caffarelli-Silvestre type extension equation \eqref{eq:CS_fract_cond} with boundary data $f$. Let $\Omega'$, $\Omega_1$, $\Omega_2$ be open, bounded, Lipschitz sets such that $\Omega' \Subset \Omega \Subset \Omega_1 \Subset \Omega_2$ and let $0 < h \leq 1$, $L \geq 1$.

Then there exist constants $C>0$ and $p_h, p_L > 0$ such that
\begin{align*}
\Vert x_{n+1}^{\frac{1-2s}{2}} \nabla^2 \tilde{u}^f \Vert_{L^2((\Omega_1 \setminus \Omega) \times (h,L))} \leq C (L^{p_L} + h^{-p_h}) \Vert f \Vert_{\widetilde{H}^s(W)}.
\end{align*}
Here, the exponents $p_h$ and $p_L$ depend on $n$ and $s$, and the constant $C$ depends on $n$, $s$, $\theta_1$, $\Omega$, $\Omega'$, $\Omega_1$, $\Omega_2$ and $\Vert a \Vert_{C^1(\Omega_2\setminus\Omega')}$.
\end{lem}

We remark that at this point, we have not sought to provide an optimal regularity estimate using the equation's full structure. By viewing the equation as a uniformly elliptic problem (with coefficients depending on $h,L$) substantial (algebraic) losses occur. Yet, for our purposes, these algebraic losses do not change the overall logarithmic type stability result which will be deduced below. 

We further remark that it is in the following proof that we make use of a uniform bound on the $C^1$-norm of the metric $a$ on $\Omega_1\setminus\Omega'$, i.e. Assumption (A4).

\begin{proof}
We note that on $\Omega_2 \times (h,L)$ the function $\tilde{u}^f$ solves a uniformly elliptic partial differential equation (with ellipticity constants depending on $h$ and $L$). Thus, we can apply the standard Caccioppoli inequality.

First, we observe that
\begin{align}\label{eq:proof_lemma_second_grad_1}
\Vert x_{n+1}^{\frac{1-2s}{2}} \nabla^2 \tilde{u}^f \Vert_{L^2((\Omega_1 \setminus \Omega) \times (h,L))} \leq (h^{ \frac{1-2s}{2}} + L^{\frac{1-2s}{2}}) \Vert \nabla^2 \tilde{u}^f \Vert_{L^2((\Omega_1\setminus\Omega) \times (h,L))}.
\end{align}
Writing $\tilde{b} := x_{n+1}^{1-2s} \tilde{a}$ and differentiating the equation for $\tilde{u}^f$ with respect to $\partial_i$, $i \in \{1,\dots,n+1\}$, we find that $\partial_i \tilde{u}^f$ satisfies
\begin{align*}
- \nabla \cdot \tilde{b} \nabla (\partial_i \tilde{u}^f) = \nabla \cdot (\partial_i \tilde{b}) \nabla \tilde{u}^f.
\end{align*}
Next, we apply the standard Caccioppoli inequality (cf. Lemma \ref{lem:Caccioppoli_inequality} with $s=\frac{1}{2}$) to obtain
\begin{equation}\label{eq:proof_lemma_second_grad_2}
\begin{aligned}
&\Vert \nabla^2 \tilde{u}^f \Vert_{L^2((\Omega_1 \setminus \Omega) \times (h,L))}\\
&\leq C \max \Big\{ \Big(\frac{L}{h}\Big)^{1-2s}, \Big(\frac{h}{L}\Big)^{1-2s} \Big\} \Big( h^{-1} \Vert \nabla \tilde{u}^f \Vert_{L^2(\Omega_2 \setminus \Omega') \times (\frac{h}{2},L+1))} + \Vert (\nabla \tilde{b}) \nabla \tilde{u}^f \Vert_{L^2((\Omega_2 \setminus\Omega') \times (\frac{h}{2},L+1))} \Big)\\
&\leq C \max \Big\{ \Big(\frac{L}{h}\Big)^{1-2s}, \Big(\frac{h}{L}\Big)^{1-2s} \Big\} \Big( h^{-1} (h^{-\frac{1-2s}{2}} + L^{-\frac{1-2s}{2}}) \Vert x_{n+1}^{\frac{1-2s}{2}} \nabla \tilde{u}^f \Vert_{L^2((\Omega_2 \setminus \Omega') \times (\frac{h}{2},L+1))}\\
&\hspace{25mm} + \Vert \tilde{b} \Vert_{C^1((\Omega_2 \setminus \Omega') \times (\frac{h}{2},L+1))} (h^{-\frac{1-2s}{2}} + L^{-\frac{1-2s}{2}}) \Vert x_{n+1}^{\frac{1-2s}{2}} \nabla \tilde{u}^f \Vert_{L^2((\Omega_2 \setminus \Omega') \times (\frac{h}{2},L+1))} \Big).
\end{aligned}
\end{equation}
Here, the factor $\max\{(L/h)^{1-2s},(h/L)^{1-2s}\}$ represents the dependency of the constant in Caccioppoli's inequality on the ellipticity constants $\lambda,\Lambda>0$ of $b$ which satisfy $\lambda \leq \tilde{b} \leq \Lambda$. Also note, that $\Vert \tilde{b} \Vert_{C^1((\Omega_2 \setminus \Omega') \times (\frac{h}{2},L+1))}$ is bounded algebraically in terms of $h^{-1}$, $L$ and $\Vert a \Vert_{C^1(\Omega_2\setminus\Omega')}$. Combining \eqref{eq:proof_lemma_second_grad_1} and \eqref{eq:proof_lemma_second_grad_2} and using that $\Vert x_{n+1}^{\frac{1-2s}{2}} \nabla \tilde{u}^f \Vert_{L^2(\R^{n+1}_+)} \leq C \Vert f \Vert_{\widetilde{H}^s(W)}$, we find that indeed
\begin{align*}
\Vert x_{n+1}^{\frac{1-2s}{2}} \nabla^2 \tilde{u}^{f} \Vert_{L^2((\Omega_1 \setminus \Omega) \times (h,L))} \leq C (L^{p_L} + h^{-p_h}) \Vert f \Vert_{\widetilde{H}^s(W)},
\end{align*}
with the constants $p_h$, $p_L$ and $C$ depending on stated quantities.
\end{proof}

With this lemma in hand, we can derive the following smallness estimate in a suitable neighbourhood of $\Omega \times (h,L)$. In what follows below, this will allow us to keep a suitable ``safety distance'' to $\partial \Omega \times (h,L)$ in our propagation of smallness arguments.

\begin{prop}\label{prop:estimate_thin_annuli}
Let $\Omega \subset \R^n$ be open, non-empty, bounded and Lipschitz and let $W \subset \Omega_e$ be open, bounded, non-empty and Lipschitz. Let $\theta_1\in(0,1)$. Let $a \in L^\infty(\R^n, \R^{n \times n}_{\text{sym}})$ satisfy the assumption (A1) with the given $\theta_1$. Let $f \in \widetilde{H}^{s}(W)$ and let $\tilde{u}^f \in \dot{H}^1(\R^{n+1}_+, x_{n+1}^{1-2s})$ be the solution to the Caffarelli-Silvestre type extension equation \eqref{eq:CS_fract_cond} with boundary data $f$. Let $\Omega'$, $\Omega_{+\delta}$, $\Omega_1$ be open, bounded, Lipschitz sets such that $\Omega' \Subset \Omega \Subset \Omega_{+\delta} \Subset \Omega_1$ and such that $\dist(\Omega, \partial\Omega_{+\delta}) \in (c_1\delta, c_1^{-1}\delta)$ for some $c_1\in(0,1)$ and some $\delta>0$ small. Let $0 < h \leq 1$, $L \geq 1$.

Then there exist $p_\delta, p_h, p_L > 0$ and a constant $C>1$ such that
\begin{align*}
\Vert \tilde{u}^f \Vert_{H^1((\Omega_{+\delta}\setminus\Omega) \times (h,L)), \ x_{n+1}^{1-2s})} \leq C \delta^{p_\delta} (L^{p_L} + h^{-p_h}) \Vert f \Vert_{\widetilde{H}^s(W)}.
\end{align*}
Here, the exponents $p_\delta$, $p_h$ and $p_L$ depend on $n$ and $s$, and the constant $C$ depends on $n$, $s$, $\theta_1$, $c_1$, $\Omega$, $\Omega'$, $\Omega_1$ and $\Vert a \Vert_{C^1(\Omega_{1}\setminus\Omega')}$.
\end{prop}

\begin{proof}
We bound the weighted $L^2$-norm of the function and of its first order derivative separately. Let $p \in (2,\infty)$ be the Sobolev-exponent associated with the Muckenhoupt-weight $x_{n+1}^{\frac{1-2s}{2}}$ (c.f. Theorem 1.2 and 1.3 in \cite{FKS82}, which is recalled in the preliminaries section). Then, we have for $q$ with $\frac{1}{p}+\frac{1}{q}=\frac{1}{2}$
\begin{align*}
\Vert x_{n+1}^{\frac{1-2s}{2}} \tilde{u}^f &\Vert_{L^2((\Omega_{+\delta} \setminus \Omega) \times (h,L))} \leq \Vert x_{n+1}^{\frac{1-2s}{q}} \Vert_{L^q((\Omega_{+\delta} \setminus \Omega) \times (h,L))} \Vert x_{n+1}^{\frac{1-2s}{p}} \tilde{u}^f \Vert_{L^p((\Omega_{+\delta} \setminus \Omega) \times (h,L))}\\
&\leq C L \big\vert \Omega_{+\delta} \setminus \Omega \big\vert^{1/q} \left( \int_0^L x_{n+1}^{1-2s} dx_{n+1} \right)^{1/q} \Vert x_{n+1}^{\frac{1-2s}{2}} \nabla \tilde{u}^f \Vert_{L^2(\R^{n+1}_+)}\\
&\leq C \delta^{p_\delta} L^{(2-2s)(\frac{1}{2}-\frac{1}{p}) +1} \Vert f \Vert_{\widetilde{H}^s(W)}.
\end{align*}
Let $\Omega_{\frac{1}{2}} \subset \R^n$ be open, bounded and Lipschitz such that $\Omega_{+\delta} \Subset \Omega_{\frac{1}{2}} \Subset \Omega_1$. For the weighted $\dot{H}^1$-norm we deduce by a similar line of inequalities as before
\begin{align*}
\Vert x_{n+1}^{\frac{1-2s}{2}} \nabla \tilde{u}^f \Vert_{L^2((\Omega_{+\delta} \setminus \Omega) \times (h,L))} \leq C \delta^{p_\delta} L^{(2-2s)(\frac{1}{2}-\frac{1}{p}) +1} \Vert x_{n+1}^{\frac{1-2s}{2}} \nabla^2 \tilde{u}^f \Vert_{L^2((\Omega_{\frac{1}{2}} \setminus \Omega) \times (h,L))}.
\end{align*}
Putting both inequalities together and applying the result of Lemma \ref{lem:estimate_second_order_grad} yields
\begin{align*}
\Vert \tilde{u}^f \Vert_{H^1((\Omega_{+\delta} \setminus \Omega) \times (h,L), x_{n+1}^{1-2s})} \leq C \delta^{p_\delta} (L^{p_L} + h^{-p_h}) \Vert f \Vert_{\widetilde{H}^s(W)}
\end{align*}
for some constants $p_\delta, p_h, p_L > 0$ depending on $n$, $s$. The constant $C$ depends on the stated quantities.
\end{proof}

\subsubsection{Two propagation of smallness results}

In this part of the text we provide a quantitative boundary-bulk-unique-continuation statement and a three-balls-inequality. The first one has been proven in \cite{RS20}, the latter one is an adaptation of a three-balls-inequality in \cite{BR25}.

\begin{prop}[Boundary-bulk unique continuation, Proposition 5.13 in \cite{RS20}]\label{prop:bbucp}
Let $\Omega \subset \R^n$ be open, non-empty, bounded and Lipschitz and let $W \subset \Omega_e$ be open, non-empty and Lipschitz. Let $\tilde{w} \in \dot{H}^1(\R^{n+1}_+, x_{n+1}^{1-2s})$ satisfy
\begin{align*}
-\nabla \cdot x_{n+1}^{1-2s} \nabla \tilde{w} = 0 \quad \text{in } \Omega_e \times \R_+,
\end{align*}
with $w(\cdot, 0) \in H^s(\R^n)$. Let $x^0 \in W \times \{0\}$ and let $r>0$ be such that $B_{4r}^+(x^0) \subset W \times \R_+$. Then, there exist constants $\alpha = \alpha(s,n) \in (0,1)$, $c = c(s,n) \in (0,\frac{1}{2})$ and $C = C(s,n) > 0$ such that
\begin{align*}
&\Vert \tilde{w} \Vert_{L^2(B_{cr}^+(x^0), x_{n+1}^{1-2s})}\\
&\qquad \leq C \left( \Vert \tilde{w} \Vert_{H^s(B_{3r}'(x^0))} + \Vert \lim_{x_{n+1}\to0} x_{n+1}^{1-2s} \partial_{n+1} \tilde{w} \Vert_{H^{-s}(B_{3r}'(x^0))} \right)^{1-\alpha}\\
&\qquad\qquad \times \left( \Vert \tilde{w} \Vert_{L^2(B_r^+(x^0), x_{n+1}^{1-2s})} + \Vert w \Vert_{H^s(B_{3r}'(x^0))} + \Vert \lim_{x_{n+1}\to0} x_{n+1}^{1-2s} \partial_{n+1} \tilde{w} \Vert_{H^{-s}(B_{3r}'(x^0))} \right)^{\alpha}.
\end{align*}
\end{prop}

Next, we provide a version of a three-balls-inequality. Similarly as in \cite{BR25}, the key point of this estimate is to obtain a controlled behaviour of the relevant constants in the vertical direction in which the equation is no longer translation invariant. Compared to the estimate from \cite{BR25}, in what follows below, we are more precise about the dependence of the constants $C$ on the outer radius $R = 4r$ under the additional assumption that the radius is uniformly bounded by the height, i.e. $\frac{r}{x_{n+1}} \leq Q$ for some $Q>0$. This will allow us to apply the three-balls-inequality for balls increasing in size when also their height is increasing simultaneously. The proof is very similar to the one in \cite{BR25}.

\begin{prop}[Three-balls-inequality]\label{prop:3ballsinequ2}
Let $\Omega \subset \R^n$ be open, bounded, non-empty and Lipschitz. Let $\tilde{w} \in \dot{H}^1(\R^{n+1}_+, x_{n+1}^{1-2s})$ satisfy
\begin{align*}
-\nabla \cdot x_{n+1}^{1-2s} \nabla \tilde{w} = 0 \quad \text{in } \Omega_e \times \R_+.
\end{align*}
Let $z = (z', z_{n+1}) \in \Omega_e \times \R_+$ and let $r>0$ be such that $B_{4r}(z) \subset \Omega_e \times \R_+$. Let $Q>0$ and assume that for all $x = (x',x_{n+1}) \in B_{4r}(z)$ it holds that $\frac{r}{x_{n+1}} \leq Q$.

Then there exist constants $C>0$ and $\alpha\in(0,1)$ such that
\begin{align*}
\Vert \tilde{w} \Vert_{L^2(B_{2r}(z), x_{n+1}^{1-2s})} \leq C \max\{r^4,r^{-3}\} \Vert \tilde{w} \Vert_{L^2(B_{r}(z), x_{n+1}^{1-2s})}^{\alpha} \Vert \tilde{w} \Vert_{L^2(B_{4r}(z), x_{n+1}^{1-2s})}^{1-\alpha}.
\end{align*}
Here the constant $C$ depends on $n$, $s$, $Q$ and the parameter $\alpha$ depends (mildly) on $r$, more precisely, $\alpha\in(b_0,b_1)$ for some constants $0<b_0<b_1<1$, which are independent of $r$.
\end{prop}

The proof relies on the following Carleman estimate. The statement and the proof are very similar to the ones of Lemma 3.7 in \cite{BR25}. The important change compared to the result in \cite{BR25} is that we give a precise choice of $\tau_0$ in dependence of $R$ under the assumption that $\frac{R}{x_{n+1}}$ is uniformly bounded. Again the key difficulty is the lack of translation invariance of the equation in the vertical direction if the center point of the balls under consideration is not located at the boundary of $\R^{n+1}_+$.

\begin{lem}\label{lem:Carleman_estimate_bulk}
Let $z = (z',z_{n+1}) \in \Omega_e \times \R_+$ and $R>0$ be such that $B_{R}(z) \subset \Omega_e \times \R_+$. Assume that for some $Q>0$ it holds for all $x = (x',x_{n+1}) \in B_{R}(z)$ that $\frac{R}{x_{n+1}} \leq Q$. Set $\phi(x) = \widetilde{\phi}(\ln(\vert x-z \vert))$ with
\begin{align*}
\widetilde{\phi}(t) = - t + \frac{1}{10} \left( t \arctan(t) - \frac{1}{2} \ln( 1 + t^2) \right).
\end{align*}  
Assume that $\overline{w} \in H^1(\R^{n+1}_+, x_{n+1}^{1-2s})$ satisfies $\supp(\overline{w}) \subset B_{R}(z) \setminus B_r(z)$ for some $0 < r < R$.

Then for any $\tau \geq \tau_0 := C_1 (1+\ln(R)^2)^{\frac{1}{3}}$ we have
\begin{equation}\label{eq:Carleman_estimate}
\begin{aligned}
&\tau \left\Vert e^{\tau\phi} \left( 1 + \ln(\vert x-z \vert)^2 \right)^{-\frac{1}{2}} \vert x-z \vert^{-1} x_{n+1}^{\frac{1-2s}{2}} \overline{w} \right\Vert_{L^2(\Omega_e \times \R_+)}\\
&\hspace{20mm} + \left\Vert e^{\tau\phi} \left( 1 + \ln(\vert x-z \vert)^2 \right)^{-\frac{1}{2}} x_{n+1}^{\frac{1-2s}{2}} \nabla \overline{w} \right\Vert_{L^2(\Omega_e \times \R_+)}\\
&\qquad \leq C \tau^{-\frac{1}{2}} \left\Vert e^{\tau\phi} \vert x-z \vert x_{n+1}^{\frac{2s-1}{2}} \big( \nabla \cdot x_{n+1}^{1-2s} \nabla \overline{w} \big) \right\Vert_{L^2(\Omega_e \times \R_+)}
\end{aligned}
\end{equation}
for some constant $C>0$. Here, $C_1$ and $C$ only depend on $n$, $s$, $Q$ and the precise form of $\widetilde{\phi}$.
\end{lem}

\begin{proof}
The proof relies on a reduction to the constant coefficient Laplacian. Indeed, for the constant coefficient Laplacian, it holds (see, for instance, \cite{KT01} and the references therein)
\begin{equation}\label{eq:Carleman_estimate_Laplace}
\begin{aligned}
&\tau \left\Vert e^{\tau\phi} \left( 1 + \ln(\vert x-z \vert)^2 \right)^{-\frac{1}{2}} \vert x-z \vert^{-1}  \overline{u} \right\Vert_{L^2(\Omega_e \times \R_+)} + \left\Vert e^{\tau\phi} \left( 1 + \ln(\vert x-z \vert)^2 \right)^{-\frac{1}{2}}  \nabla \overline{u} \right\Vert_{L^2(\Omega_e \times \R_+)}\\
&\qquad \leq C \tau^{-\frac{1}{2}} \left\Vert e^{\tau\phi} \vert x-z \vert \Big(  (-\Delta') - \partial_{n+1}^2 ) \Big) \overline{u} \right\Vert_{L^2(\Omega_e \times \R_+)}.
\end{aligned}
\end{equation}
We rewrite our equation with the weight as follows: For $\overline{w} = x_{n+1}^{\frac{2s-1}{2}} \overline{u}$ (note that for fixed $z$ and fixed $R$ the support of $\overline{w}$ is bounded away from $\R^n \times \{0\}$ and thus no singularities arise), we obtain by expanding the equation
\begin{align*}
x_{n+1}^{\frac{2s-1}{2}} \Big( x_{n+1}^{1-2s} (-\Delta') - \partial_{n+1} (x_{n+1}^{1-2s} \partial_{n+1}) \Big) \overline{w} 
&= x_{n+1}^{\frac{2s-1}{2}} \Big( x_{n+1}^{1-2s} (-\Delta') - \partial_{n+1} (x_{n+1}^{1-2s} \partial_{n+1}) \Big) x_{n+1}^{\frac{2s-1}{2}} \overline{u} \\
&=  ((-\Delta') - \partial_{n+1}^2) \overline{u} - \frac{1-4s^2}{4} x_{n+1}^{-2} \overline{u}.
\end{align*}
Thus, it follows that 
\begin{align*}
((-\Delta') - \partial_{n+1}^2) \overline{u} = x_{n+1}^{\frac{2s-1}{2}} \Big( x_{n+1}^{1-2s} (-\Delta') - \partial_{n+1} (x_{n+1}^{1-2s} \partial_{n+1}) \Big) \overline{w} + \frac{1-4s^2}{4} x_{n+1}^{-2} \overline{u}.
\end{align*}
Plugging this back into our Carleman estimate \eqref{eq:Carleman_estimate_Laplace} for the Laplacian, we obtain
\begin{equation}\label{eq:Carleman_estimate_Laplace1}
\begin{aligned}
\tau &\left\Vert e^{\tau\phi} \left( 1 + \ln(\vert x-z \vert)^2 \right)^{-\frac{1}{2}} \vert x-z \vert^{-1}  \overline{u} \right\Vert_{L^2(\Omega_e \times \R_+)} + \left\Vert e^{\tau\phi} \left( 1 + \ln(\vert x-z \vert)^2 \right)^{-\frac{1}{2}}  \nabla \overline{u} \right\Vert_{L^2(\Omega_e \times \R_+)}\\
& \leq C \tau^{-\frac{1}{2}} \left\Vert e^{\tau\phi} \vert x-z \vert \Big(  (-\Delta') - \partial_{n+1}^2 ) \Big) \overline{u} \right\Vert_{L^2(\Omega_e \times \R_+)} \\
& \leq C\tau^{-\frac{1}{2}} \left\Vert e^{\tau\phi} \vert x-z \vert \left( x_{n+1}^{\frac{2s-1}{2}} \Big( x_{n+1}^{1-2s} (-\Delta') - \partial_{n+1} (x_{n+1}^{1-2s} \partial_{n+1}) \Big) \overline{w} + \frac{1-4s^2}{4} x_{n+1}^{-2} \overline{u} \right) \right\Vert_{L^2(\Omega_e \times \R_+)}\\
& \leq C\tau^{-\frac{1}{2}} \left\Vert e^{\tau\phi} \vert x-z \vert x_{n+1}^{\frac{2s-1}{2}} \Big( x_{n+1}^{1-2s} (-\Delta') - \partial_{n+1} (x_{n+1}^{1-2s} \partial_{n+1}) \Big) \overline{w} \right\Vert_{L^2(\Omega_e \times \R_+)}\\
&\qquad + C\tau^{-\frac{1}{2}} \left\Vert e^{\tau\phi} \vert x-z \vert x_{n+1}^{-2} \overline{u} \right\Vert_{L^2(\Omega_e \times \R_+)}.
\end{aligned}
\end{equation}
By the uniform boundedness assumption, $\frac{R}{x_{n+1}} \leq Q$, we can choose $\tau_0 := C_1 (1+\ln(R)^2)^{\frac{1}{3}}$ (with the constant $C_1$ depending on the stated quantities) such that for all $\tau\geq\tau_0$ it holds
\begin{align*}
C \tau^{-\frac{1}{2}} R \max_{(x',x_{n+1}) \in B_R(z)} \{ x_{n+1}^{-2} \} \leq \frac{1}{2} \tau (1+\ln(R)^2)^{-\frac{1}{2}} R^{-1}.
\end{align*}
Thus, we can absorb the second right hand side contribution in \eqref{eq:Carleman_estimate_Laplace1} into the left hand side. Also, when we plug back into $\overline{u} = x_{n+1}^{\frac{1-2s}{2}} \overline{w}$, this leads to 
\begin{equation}\label{eq:Carleman_estimate_Laplace2}
\begin{aligned}
&\tau \left\Vert e^{\tau\phi} \left( 1 + \ln(\vert x-z \vert)^2 \right)^{-\frac{1}{2}} \vert x-z \vert^{-1} x_{n+1}^{\frac{1-2s}{2}} \overline{w} \right\Vert_{L^2(\Omega_e \times \R_+)}\\
&\qquad\qquad + \left\Vert e^{\tau\phi} \left( 1 + \ln(\vert x-z \vert)^2 \right)^{-\frac{1}{2}}  \nabla ( x_{n+1}^{\frac{1-2s}{2}} \overline{w} ) \right\Vert_{L^2(\Omega_e \times \R_+)}\\
&\qquad \leq C\tau^{-\frac{1}{2}} \left\Vert e^{\tau\phi} \vert x-z \vert x_{n+1}^{\frac{2s-1}{2}} \Big( x_{n+1}^{1-2s} (-\Delta') - \partial_{n+1} (x_{n+1}^{1-2s} \partial_{n+1}) \Big) \overline{w} \right\Vert_{L^2(\Omega_e \times \R_+)}.
\end{aligned}
\end{equation}
We note that (for abbreviation we now write $h = e^{\tau \phi} (1+\ln(\vert \cdot-z \vert)^2)^{-\frac{1}{2}}$)
\begin{equation}\label{eq:Carleman_estimate_Laplace3}
\begin{aligned}
\Vert h \nabla (x_{n+1}^{\frac{1-2s}{2}} \overline{w}) \Vert_{L^2(\Omega_e \times \R_+)} \geq \Vert h x_{n+1}^{\frac{1-2s}{2}} \nabla \overline{w} \Vert_{L^2(\Omega_e \times \R_+)} - C \Vert h x_{n+1}^{\frac{1-2s}{2}} x_{n+1}^{-1} \overline{w} \Vert_{L^2(\Omega_e \times \R_+)}.
\end{aligned}
\end{equation}
Again, we choose $\tau_0$ large enough (actually here it would even suffice to take $\tau_0 \sim O(1)$) such that for all $\tau\geq\tau_0$ it holds
\begin{align*}
C \max_{(x',x_{n+1}) \in B_R(z)} \{x_{n+1}^{-1}\} \leq \frac{1}{2} \tau R^{-1}.
\end{align*}
Finally, we can modify the estimate \eqref{eq:Carleman_estimate_Laplace2} by absorbing the error term in \eqref{eq:Carleman_estimate_Laplace3} into the first contribution on the left hand side to arrive at the desired inequality.
\end{proof}

With the Carleman estimate in hand, we proceed to the proof of the three-balls-inequality, Proposition \ref{prop:3ballsinequ2}. The proof is essentially the same as the one of Proposition 3.6 in \cite{BR25}, with the main difference being a more careful tracking of the dependence of the relevant constants.

\begin{proof}[Proof of Proposition \ref{prop:3ballsinequ2}]
We will denote by $A_{r_1,r_2} := B_{r_2}(z) \setminus \overline{B}_{r_1}(z)$ the annuli with outer radius $r_2$ and inner radius $r_1$ centered at $z$.\\
Define $\overline{w} := \tilde{w}\eta$, where $\eta(x) = \widetilde{\eta}(\vert x-z \vert)$ is a radially symmetric (centred at $z$) cut-off function such that
\begin{align*}
\widetilde{\eta}(t) = 0 \text{ for } t \in (0,\frac{r}{4}) \cup (\frac{7r}{2}, \infty), \quad \widetilde{\eta}(t) = 1 \text{ for } t \in (\frac{r}{2},3r), \quad \vert \widetilde{\eta}' \vert \leq \frac{C}{r}, \quad \vert \widetilde{\eta}'' \vert \leq \frac{C}{r^2}
\end{align*}
for some $C>0$. Since $\supp(\overline{w}) \subset B_{7r/2}(x^0) \setminus B_{r/4}(x^0)$, we can apply the Carleman estimate from Lemma \ref{lem:Carleman_estimate_bulk} to $\overline{w}$.

Expanding the term on the right hand side of the Carleman estimate \eqref{eq:Carleman_estimate} and using the equation for $\tilde{w}$, we arrive at
\begin{align*}
x_{n+1}^{\frac{2s-1}{2}} &\nabla \cdot \left( x_{n+1}^{1-2s} \nabla (\tilde{w}\eta) \right) = x_{n+1}^{\frac{1-2s}{2}} \tilde{w} \left( x_{n+1}^{2s-1} \nabla \cdot (x_{n+1}^{1-2s} \nabla \eta) \right) + 2 x_{n+1}^{\frac{1-2s}{2}} \nabla \tilde{w} \cdot \nabla \eta.
\end{align*}
We insert this into the Carleman inequality \eqref{eq:Carleman_estimate} (only keeping the first term on the left hand side of the inequality) and use that $\widetilde{\phi}$ is monotonically decreasing to infer for $\tau \geq \tau_0 \geq 1$
\begin{align*}
&e^{\tau \widetilde{\phi}(\ln(2r))} \Vert x_{n+1}^{\frac{1-2s}{2}} \tilde{w} \Vert_{L^2(A_{\frac{r}{2}, 2r})}\\
&\qquad \leq C \max\{1,r^3\} \bigg( e^{\tau\widetilde{\phi}(\ln(\frac{r}{4}))} \Vert x_{n+1}^{\frac{1-2s}{2}} \tilde{w} \left( x_{n+1}^{2s-1} \nabla \cdot (x_{n+1}^{1-2s} \nabla \eta) \right) \Vert_{L^2(A_{\frac{r}{4},\frac{r}{2}})}\\
&\hspace{30mm} + e^{\tau\widetilde{\phi}(\ln(3r))} \Vert x_{n+1}^{\frac{1-2s}{2}} \tilde{w} \left( x_{n+1}^{2s-1} \nabla \cdot (x_{n+1}^{1-2s} \nabla \eta) \right) \Vert_{L^2(A_{3r,\frac{7r}{2}})}\\
&\hspace{30mm} + e^{\tau	\widetilde{\phi}(\ln(\frac{r}{4}))} \Vert x_{n+1}^{\frac{1-2s}{2}} \nabla \tilde{w} \cdot \nabla \eta \Vert_{L^2(A_{\frac{r}{4},\frac{r}{2}})} + e^{\tau \widetilde{\phi}(\ln(3r))} \Vert x_{n+1}^{\frac{1-2s}{2}} \nabla \tilde{w} \cdot \nabla \eta \Vert_{L^2(A_{3r,\frac{7r}{2}})} \bigg)\\
&\qquad \leq C \max\{1,r^3\} \max\{1,r^{-2}\} \bigg( e^{\tau	\widetilde{\phi}(\ln(\frac{r}{4}))} \Vert x_{n+1}^{\frac{1-2s}{2}} \tilde{w} \Vert_{L^2(A_{\frac{r}{8},r})} + e^{\tau	\widetilde{\phi}(\ln(3r))} \Vert x_{n+1}^{\frac{1-2s}{2}} \tilde{w} \Vert_{L^2(A_{2r,4r})} \bigg)\\
&\qquad \leq C \max\{r^3,r^{-2}\} \bigg( e^{\tau \widetilde{\phi}(\ln(\frac{r}{4}))} \Vert x_{n+1}^{\frac{1-2s}{2}} \tilde{w} \Vert_{L^2(B_r)} + e^{\tau \widetilde{\phi}(\ln(3r))} \Vert x_{n+1}^{\frac{1-2s}{2}} \tilde{w} \Vert_{L^2(B_{4r})} \bigg).
\end{align*}
Here, for the first inequality we have estimated $(1+\ln(\vert x-z \vert)^2)^{-\frac{1}{2}} \vert x-z \vert^{-1} \geq C \min\{1,r^{-2}\}$ and that $\vert x-z \vert \leq C \max\{1,r\}$ for all $x \in B_{4r}(z)$. For the second inequality we used the bounds on the derivative of $\tilde{\eta}$ and we applied Caccioppoli's inequality (cf. Lemma \ref{lem:Caccioppoli_inequality}) to bound the gradient contributions. Additionally we used the assumption $\frac{r}{x_{n+1}} \leq Q$ and the bound on the derivative of $\eta$ to bound $x_{n+1}^{2s-1} \partial_{n+1} (x_{n+1}^{1-2s} \partial_{n+1} \eta)$ in terms of $\max\{1,r^{-2}\}$ and $Q$ (this is particularly relevant when $x_{n+1}$ and thus $r$ are small). Denote by $C_r$ the constant $C_r = C \max\{r^3,r^{-2}\}$. Dividing the previous estimate by $e^{\tau\widetilde{\phi}(\ln(2r))}$ yields
\begin{align*}
\Vert x_{n+1}^{\frac{1-2s}{2}} \tilde{w} \Vert_{L^2(A_{\frac{r}{2}, 2r})} \leq C_r \left( e^{\tau P_1} \Vert x_{n+1}^{\frac{1-2s}{2}} \tilde{w} \Vert_{L^2(B_r)} + e^{-\tau P_2} \Vert x_{n+1}^{\frac{1-2s}{2}} \tilde{w} \Vert_{L^2(B_{4r})} \right),
\end{align*}
where $P_1 := \widetilde{\phi}(\ln(\frac{r}{4})) - \widetilde{\phi}(\ln(2r))$ and $P_2 := -(\widetilde{\phi}(\ln(3r)) - \widetilde{\phi}(\ln(2r)))$. It holds true that $c^{-1} \leq P_1,P_2 \leq c$ for some constant $c > 0$ uniformly in $r$. In particular, since $P_1$ is uniformly bounded away from $0$, we can add $\Vert x_{n+1}^{\frac{1-2s}{2}} \tilde{w} \Vert_{L^2(B_{\frac{r}{2}})}$ to both sides to fill up the annuli. By possibly changing $C_r$ (note that the $r$ dependence, however, stays the same), we thus infer
\begin{align}\label{eq:preCarleman}
\Vert x_{n+1}^{\frac{1-2s}{2}} \tilde{w} \Vert_{L^2(B_{2r})} \leq C_r \left( e^{\tau P_1} \Vert x_{n+1}^{\frac{1-2s}{2}} \tilde{w} \Vert_{L^2(B_r)} + e^{-\tau P_2} \Vert x_{n+1}^{\frac{1-2s}{2}} \tilde{w} \Vert_{L^2(B_{4r})} \right).
\end{align}
We seek to find $\tau \geq \tau_0$ (recall that $\tau_0$ only depended on $n$, $s$, $Q$ and $r$) such that
\begin{align*}
C_re^{-\tau P_2} \Vert x_{n+1}^{\frac{1-2s}{2}} \tilde{w} \Vert_{L^2(B_{4r})} \leq \frac{1}{2} \Vert x_{n+1}^{\frac{1-2s}{2}} \tilde{w} \Vert_{L^2(B_{2r})}.
\end{align*}
We achieve this by choosing
\begin{align*}
\tau = \tau_0 - \frac{1}{P_2} \ln \left( \frac{\Vert x_{n+1}^{\frac{1-2s}{2}} \tilde{w} \Vert_{L^2(B_{2r})}}{ 2C_r\Vert x_{n+1}^{\frac{1-2s}{2}} \tilde{w} \Vert_{L^2(B_{4r})}} \right).
\end{align*}
From \eqref{eq:preCarleman} we then infer
\begin{align*}
\Vert x_{n+1}^{\frac{1-2s}{2}} \tilde{w} \Vert_{L^2(B_{2r})} &\leq C_r^{\frac{P_1+P_2}{P_2}} e^{P_1\tau_0} \frac{\Vert x_{n+1}^{\frac{1-2s}{2}} \tilde{w} \Vert_{L^2(B_{4r})}^{P_1/P_2}}{\Vert x_{n+1}^{\frac{1-2s}{2}} \tilde{w} \Vert_{L^2(B_{2r})}^{P_1/P_2}} \Vert x_{n+1}^{\frac{1-2s}{2}} \tilde{w} \Vert_{L^2(B_r)}.
\end{align*}
By rearranging terms and choosing $\alpha := \frac{P_2}{P_1+P_2}$ this yields
\begin{align*}
\Vert x_{n+1}^{\frac{1-2s}{2}} \tilde{w} \Vert_{L^2(B_{2r})} \leq C e^{\frac{P_1 P_2}{P_1 + P_2} \tau_0} \max\{r^3,r^{-2}\} \Vert x_{n+1}^{\frac{1-2s}{2}} \tilde{w} \Vert_{L^2(B_r)}^\alpha \Vert x_{n+1}^{\frac{1-2s}{2}} \tilde{w} \Vert_{L^2(B_{4r})}^{1-\alpha}.
\end{align*}
Since $\tau_0 = C_1 (1+\ln(r)^2)^{\frac{1}{3}}$ and since $e^{C(1+\ln(r)^2)^{1/3}} \leq C \max\{r,r^{-1}\}$, we finally arrive at the desired three-balls-inequality
\begin{align*}
\Vert x_{n+1}^{\frac{1-2s}{2}} \tilde{w} \Vert_{L^2(B_{2r})} \leq C \max\{r^4,r^{-3}\} \Vert x_{n+1}^{\frac{1-2s}{2}} \tilde{w} \Vert_{L^2(B_r)}^\alpha \Vert x_{n+1}^{\frac{1-2s}{2}} \tilde{w} \Vert_{L^2(B_{4r})}^{1-\alpha}
\end{align*}
and the proof is finished.
\end{proof}

\subsubsection{Main estimate of this subsection}

We turn to the main proposition of this subsection. It will eventually be applied to the difference of two solutions to our fractional Calder\'on problem (see the proof of Proposition \ref{prop:QUCP_1} below).

\begin{prop}\label{prop:estimate_Omega_ext}
Let $\Omega \subset \R^n$ be open, non-empty, bounded and Lipschitz such that $\R^n \setminus \Omega$ is connected and let $W \subset \Omega_e$ be open, bounded, non-empty and Lipschitz. Let $\theta_1\in(0,1)$ and assume that $a \in L^\infty(\R^n, \R^{n \times n}_{\text{sym}})$ satisfies the assumptions (A1) with the given $\theta_1$, and (A3). Let $\Omega_1$, $\Omega_{+\delta}$ be open, bounded, Lipschitz such that $\Omega \Subset \Omega_{+\delta} \Subset \Omega_1$ with $\dist(\Omega, \partial\Omega_{+\delta}) \in (c_1\delta, c_1^{-1}\delta)$ for some $\delta>0$ small, $c_1 \in (0,1)$. Let $\tilde{w} \in \dot{H}^1(\R^{n+1}_+, x_{n+1}^{1-2s})$ satisfy
\begin{equation*}
\begin{cases}
\begin{alignedat}{2}
-\nabla \cdot x_{n+1}^{1-2s} \tilde{a} \nabla  \tilde{w} &= 0 \quad &&\text{in } \R^{n+1}_+,\\
\tilde{w} &= 0 \quad &&\text{on } W \times \{0\}.
\end{alignedat}
\end{cases}
\end{equation*}
Assume that the following estimate holds for all $\bar{L} \geq 1$
\begin{align*}
\Vert \tilde{w} \Vert_{L^2(B_{\kappa \bar{L}}' \times (0,\kappa \bar{L}), x_{n+1}^{1-2s})} \leq C_1 \bar{L}^{\bar{p}} A
\end{align*}
for some $C_1, \bar{p}, A>0$ and some sufficiently large $\kappa>0$.

Let $0 < h \leq \frac{1}{2}$ and $2 \leq L < \infty$. There exist constants $c\in(0,1)$, $C>0$, $p>0$ and $\alpha \in (0,1)$ such that, if for some $\varepsilon \in (0,\frac{1}{2})$
\begin{align*}
\Vert \restr{\lim_{x_{n+1}\to0} x_{n+1}^{1-2s} \partial_{n+1} \tilde{w}}{W} \Vert_{H^{-s}(W)} \leq \varepsilon A.
\end{align*}
then it holds
\begin{align*}
\Vert \tilde{w} \Vert_{L^2((\Omega_1 \setminus \Omega_{+\delta}) \times (h,L), x_{n+1}^{1-2s})} \leq C A \Big( (&L^{p_L} +\delta^{-p_\delta}) \varepsilon^{c\alpha^{C(\log(L) + \log(\delta^{-1}))}}\\
&+ \big( h^{-p_h} + \delta^{-p_\delta} \big) \varepsilon^{c\alpha^{C(\log(h^{-1}) + \log(\delta^{-1}))}} \Big).
\end{align*}
The constants $c$, $C$, $p_L$, $p_h$ and $p_\delta$ only depend on $n$, $s$, $C_1$, $\bar{p}$, $c_1$, $\Omega$, $\Omega_1$  and $W$.
\end{prop}

Here, the constant $A$ plays the role of $\Vert f \Vert_{\widetilde{H}^s(W)}$ in the later application of this result, see the proof of Proposition \ref{prop:QUCP_1}. The idea of the argument for Proposition \ref{prop:estimate_Omega_ext} is to first transfer the smallness of the data on the boundary into the interior of $\R^{n+1}_+$ via the boundary-bulk unique continuation estimate, Proposition \ref{prop:bbucp}, and then to propagate the smallness along a sequence of balls using the three-balls-inequality, Proposition \ref{prop:3ballsinequ2}.

\begin{proof}
We divide the proof of Proposition \ref{prop:estimate_Omega_ext} into several steps. We first prove the estimate on $(\Omega_1 \setminus \Omega_{+\delta}) \times (1,L)$ and then on $(\Omega_1 \setminus \Omega_{+\delta}) \times (h,1)$.

\textit{Step 1: Estimate on $(\Omega_1 \setminus \Omega_{+\delta}) \times (1,L)$.} We divide this step itself into several substeps. A schematic visualization of this part of the proof, in particular of Step 1.1 and Step 1.3, is depicted in Figure \ref{fig:schematic_argument_propagate_smallness_1}.

Let $\kappa>0$ to be determined later. By assumption, we know that
\begin{align}\label{eq:proof_intermediate_integral_1}
K := \Vert \tilde{w} \Vert_{L^2(B_{\kappa L}' \times (0,\kappa L), x_{n+1}^{1-2s})} \leq C_1 L^{\bar{p}} A.
\end{align}

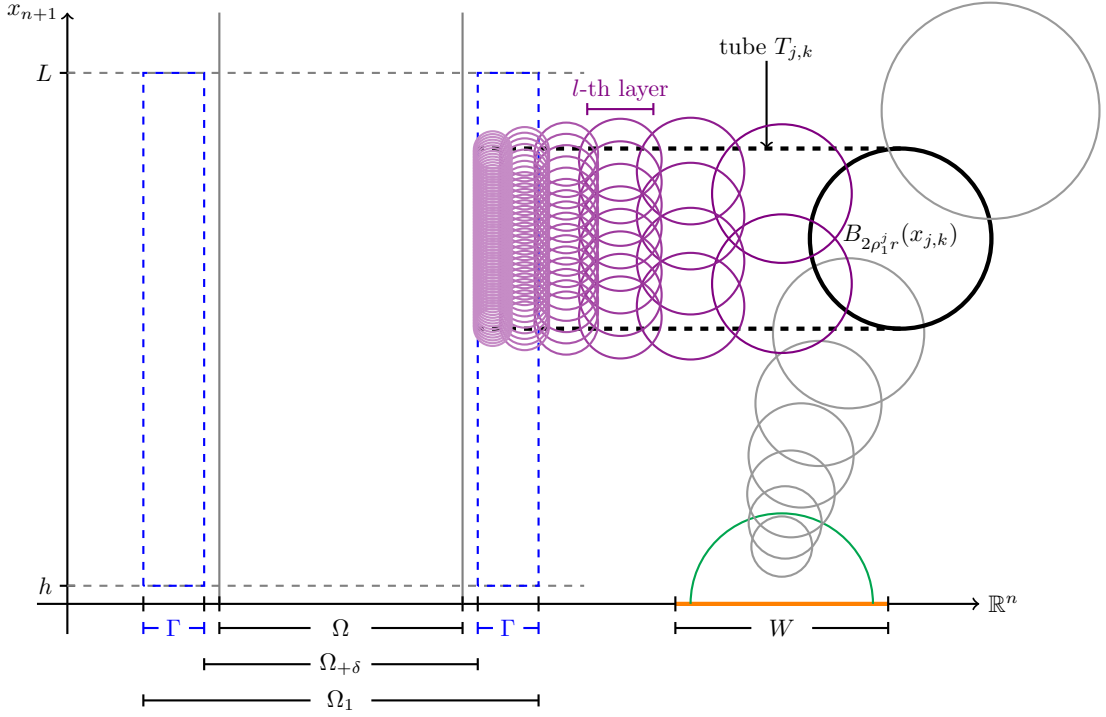
\begin{figure}
	\begin{center}
  	\begin{tikzpicture}[thick,scale=0.8,every node/.style={scale=0.9}]
  	\draw[->] (-0.5,0) -- (15,0) node[right] {$\R^n$};
  	\draw[->] (0,-0.5) -- (0,9.8) node[left] {$x_{n+1}$};
  	\draw (-1/8,8.8) -- (1/8,8.8) node[left=2mm] {$L$};
  	\draw (-1/8,0.3) -- (1/8,0.3) node[left=2mm] {$h$};
  	
  	\draw[gray] (2.5,0) -- (2.5,9.8); \draw[gray] (6.5,0) -- (6.5,9.8);
  	\draw[dashed,gray] (0,0.3) -- (8.5,0.3); \draw[dashed,gray] (0,8.8) -- (8.5,8.8);
  	\draw (2.5,-1/8) -- (2.5,1/8); \draw (6.5,-1/8) -- (6.5,1/8);
  	\draw (1.25,-1/8) -- (1.25,1/8); \draw (7.75,-1/8) -- (7.75,1/8);
  	\draw (2.25,-1/8) -- (2.25,1/8); \draw (6.75,-1/8) -- (6.75,1/8);
  	\draw[thick] (2.5,-0.4) -- (4,-0.4); \draw[thick] (5,-0.4) -- (6.5,-0.4); \draw[thick] (2.5,-0.5) -- (2.5,-0.3); \draw[thick] (6.5,-0.5) -- (6.5,-0.3);
  	\node at (4.5,-0.4) {$\Omega$};
  	\draw[thick] (2.25,-1) -- (4,-1); \draw[thick] (5,-1) -- (6.75,-1); \draw[thick] (2.25,-1.1) -- (2.25,-0.9); \draw[thick] (6.75,-1.1) -- (6.75,-0.9);
  	\node at (4.5,-1) {$\Omega_{+\delta}$};
  	\draw[thick] (1.25,-1.6) -- (4,-1.6); \draw[thick] (5,-1.6) -- (7.75,-1.6); \draw[thick] (1.25,-1.7) -- (1.25,-1.5); \draw[thick] (7.75,-1.7) -- (7.75,-1.5);
  	\node at (4.5,-1.6) {$\Omega_1$};
  	
  	\draw[ultra thick,orange] (10,0) -- (13.5,0); \draw (10,-1/8) -- (10,1/8); \draw (13.5,-1/8) -- (13.5,1/8);
  	\draw[thick] (10,-0.4) -- (11.25,-0.4); \draw[thick] (12.25,-0.4) -- (13.5,-0.4); \draw[thick] (10,-0.5) -- (10,-0.3); \draw[thick] (13.5,-0.5) -- (13.5,-0.3);
  	\node at (11.75,-0.4) {$W$};
  	
  	\draw[thick,blue] (1.25,-0.4) -- (1.5,-0.4); \draw[thick,blue] (2,-0.4) -- (2.25,-0.4); \draw[thick,blue] (1.25,-0.3) -- (1.25,-0.5); \draw[thick,blue] (2.25,-0.3) -- (2.25,-0.5);
  	\node[blue] at (1.75,-0.4) {$\Gamma$};
  	\draw[thick,blue] (6.75,-0.4) -- (7,-0.4); \draw[thick,blue] (7.5,-0.4) -- (7.75,-0.4); \draw[thick,blue] (6.75,-0.3) -- (6.75,-0.5); \draw[thick,blue] (7.75,-0.3) -- (7.75,-0.5);
  	\node[blue] at (7.25,-0.4) {$\Gamma$};
  	\draw[dashed,blue] (1.25,0.3) -- (1.25,8.8) -- (2.25,8.8) -- (2.25,0.3) -- (1.25,0.3);
  	\draw[dashed,blue] (6.75,0.3) -- (6.75,8.8) -- (7.75,8.8) -- (7.75,0.3) -- (6.75,0.3);
  	
  	%%%% set-up for Step 1.1 when propagating smallness upwards in x_{n+1}-direction
  	\draw[Green] (13.25,0) arc(0:180:1.5);  	
  	\tikzmath{\q1 = 1.2; \q2 = 1.35; \x' = 11.75;} % q1 and q2 play the role of rho_1 and rho_2, respectively
  	\draw[gray!80] (\x', 0.95) circle(0.5);
  	\foreach \j in {1,...,7}{
  		\pgfmathparse{\x'+\q1^(3*(\j-1))/18}
    		\xdef\x'{\pgfmathresult}
    		\tikzmath{\y = (\q2^\j)); \r = 2*(\q1^\j)/4);}
    		\ifthenelse{\j = 6}{ % then-part of "if" statement
    			\draw[ultra thick] (\x', \y) circle(\r); \node at (\x',\y) {$B_{2\rho_1^jr}(x_{j,k})$};
    			
    			%%%% for some specific j, we also depict the set-up of balls for Step 1.3
    			\tikzmath{\y1 = \y-\r; \y2 = \y1+2*\r; \q3 = 1.3;}
    			\draw[dashed, ultra thick] (\x',\y1) -- (6.75,\y1);
    			\draw[dashed, ultra thick] (\x',\y2) -- (6.75,\y2);
    			\draw[->] (11.5,9) -- (11.5,\y2); \node at (11.5,9.2) {tube $T_{j,k}$};
    			
    			\tikzmath{\posx = \x';}
    			\foreach \k in {1,...,6}{
    				\tikzmath{\Nk = 2^\k; \radius = \r*(\q3^(-\k)); \opacity = max(40, 100-11*(\k-1));}
    				\pgfmathparse{\posx - 1.7*\radius}
    				\xdef\posx{\pgfmathresult}
    				\foreach \l in {1,...,\Nk}{
    					%\tikzmath{\posx = \x'-2*(\k^(60/100)); \posy = \y1 + ((2*\l-1)/\Nk)*\r;}
    					\tikzmath{\posy = \y1 + ((2*\l-1)/\Nk)*\r;}
   					\draw[violet!\opacity] (\posx, \posy) circle(\radius);
    				}
    				
    				\ifthenelse{\k = 3}{
    					\tikzmath{\leftdash = \posx-(4/5)*\radius; \rightdash = \posx+(4/5)*\radius;}
    					\draw[violet] (\leftdash,8.1) -- (\leftdash,8.3); \draw[violet] (\rightdash,8.1) -- (\rightdash,8.3);
    					\draw[violet] (\leftdash,8.2) -- (\rightdash,8.2) node[above,midway] {$l$-th layer};
    				}
    				{} % else-part of "if" statement
    			}
    			
    			}
    			{\draw[gray!80] (\x', \y) circle(\r);} % else-part of "if" statement
  	}
  	\end{tikzpicture}
	\end{center}
	\caption{Schematic visualization of the propagation of smallness argument in the proof of Proposition \ref{prop:estimate_Omega_ext}. We seek to prove smallness of $\Vert \tilde{w} \Vert_{L^2((\Omega_1 \setminus \Omega_{+\delta}) \times (h,L), x_{n+1}^{1-2s})}$ on $(\Omega_1 \setminus \Omega_{+\delta}) \times (h,L)$ (blue part, in the figure: $\Gamma := \Omega_1 \setminus \Omega_{+\delta}$). In particular, Steps 1.1 and 1.3 of the proof are depicted here. By assumption, we have smallness of the data on the boundary $W \times \{0\}$ (orange part). Applying the quantitative boundary-bulk unique continuation argument, Proposition \ref{prop:bbucp}, we transfer the smallness into the bulk, more precisely, to the $(n+1)$-dimensional (green) half-ball. Then, we propagate the smallness upwards along a sequence of balls increasing exponentially in size, fading away from $\Omega \times (0,\infty)$ (gray balls), using Proposition \ref{prop:3ballsinequ2}. In Step 1.3, we then propagate the smallness back towards $\Omega \times (h,L)$ along a sequence of balls, which cover the tube $T_{j,k}$ and which decrease exponentially in size (violet balls). For visualization, sizes and positions of balls are not true-to-scale.}
\label{fig:schematic_argument_propagate_smallness_1}
\end{figure}

\textit{Step 1.1: Propagation of smallness upwards in $x_{n+1}$-direction.} Let $\bar{x} \in W \times \{0\}$, $\bar{r}>0$ be such that $B_{4\bar{r}}^+(\bar{x}) \subset W \times \R_+$. Using the assumption on the boundary data and the a-priori estimate for the interior (with $\bar{L}=1$), the boundary-bulk unique continuation result, Proposition \ref{prop:bbucp}, implies that for some constants $c \in (0,\frac{1}{2})$, $\bar{\alpha}\in(0,1)$ and $C>0$
\begin{align}\label{eq:proof_intermediate_integral_2}
\Vert \tilde{w} \Vert_{L^2(B_{c\bar{r}}^+(\bar{x}),x_{n+1}^{1-2s})} \leq C A \varepsilon^{1-\bar{\alpha}}.
\end{align}

Let $r>0$ and $x_0 = (x_0', 5r) \in W \times \R_+$ be such that $B_{4r}(x_0) \subset B_{c\bar{r}}^+(\bar{x})$. Note that $r$ only depends on $n$, $s$ and $W$. Let $\rho_1 \in (1,2)$ and $x_j = (x_j', 5\rho_1^j r)$ be such that $B_{\rho_1^j r}(x_j) \subset B_{2\rho_1^{j-1}r}(x_{j-1})$ and $B_{4\rho_1^j r}(x_j) \subset \Omega_e \times \R_+$ for all $j \in \{1,2,\dots,N_j\}$. In other words, in each step the height $x_{n+1}$ and the radius increase with the same factor. Note that, since $B_{4\rho_1^j r}(x_j) \subset \Omega_e \times \R_+$, the balls fade away from $\Omega$ (see also Figure \ref{fig:schematic_argument_propagate_smallness_1}). Moreover, we can choose $Q>0$ uniformly such that for all $j\in\{1,2,\dots,N_j\}$ and all $y = (y',y_{n+1}) \in B_{4\rho_1^j r}(x_j)$ it holds
\begin{align*}
\frac{\rho_1^j r}{y_{n+1}} \leq \frac{\rho_1^j r}{5\rho_1^j r - 4\rho_1^j r} \leq Q.
\end{align*}
Here, $N_j$ is chosen minimal such that $5\rho_1^{N_j} r \geq L$. In particular, it holds that $N_j \sim \log(L)$.

Applying first the boundary-bulk unique continuation result, Proposition \ref{prop:bbucp}, and then iteratively the three-balls-inequality, Proposition \ref{prop:3ballsinequ2}, along the given sequence of balls, we infer
\begin{equation}
\begin{aligned}\label{eq:proof_intermediate_integral_3}
\Vert \tilde{w} &\Vert_{L^2(B_{2\rho_1^j r}(x_j),x_{n+1}^{1-2s})} \leq C \rho_1^{4j} \Vert \tilde{w} \Vert_{L^2(B_{\rho_1^j r}(x_j),x_{n+1}^{1-2s})}^{\alpha_j} \Vert \tilde{w} \Vert_{L^2(B_{4\rho_1^j r}(x_j),x_{n+1}^{1-2s})}^{1-\alpha_j}\\
&\qquad \leq C L^4 K^{1-\alpha_j} \Vert \tilde{w} \Vert_{L^2(B_{2\rho_1^{j-1} r}(x_{j-1}),x_{n+1}^{1-2s})}^{\alpha_j}\\
&\qquad \leq C L^{4 (1 + \sum_{\ell_1 = 0}^{j-2} \prod_{\ell_2=0}^{\ell_1} \alpha_{j-\ell_2} )} K^{\sum_{\ell_1=0}^{j-1} (1-\alpha_{j-\ell_1}) \prod_{\ell_2=0}^{\ell_1 - 1} \alpha_{j-\ell_2}} \Vert \tilde{w} \Vert_{L^2(B_{c\bar{r}}^+(\bar{x}),x_{n+1}^{1-2s})}^{\prod_{\ell_2=0}^{j-1} \alpha_{j-\ell_2}}\\
&\qquad \leq C L^{p_L} A \varepsilon^{(1-\bar{\alpha}) \prod_{\ell_2=0}^{j-1} \alpha_{j-\ell_2}}\\
&\qquad \leq C L^{p_L} A \varepsilon^{(1-\bar{\alpha}) \alpha^j}.
\end{aligned}
\end{equation}
Here, for the second inequality we have used \eqref{eq:proof_intermediate_integral_1} and that $\rho_1^j \leq CL$ for all $j\in\{1,2,\dots,N_j\}$, the third inequality is a consequence of the iterative application of the three-balls-inequality and for the fourth one we used the estimates \eqref{eq:proof_intermediate_integral_1} and \eqref{eq:proof_intermediate_integral_2}. For the last inequality we used that $\alpha_j \in (b_0,b_1)$ for all $j\in\{1,\dots,N_j\}$ for some uniform $0<b_0<b_1<1$, and denoted the lower bound by $\alpha := \min_{j\in\{1,\dots,N_j\}} \alpha_j$. We remark that by the uniform choice of $Q>0$ the constant $C$ is independent of $j$. Also the exponent $p_L>0$ is independent of $j$ since $\alpha_j \in (b_0,b_1)$ for all $j\in\{1,\dots,N_j\}$ (compare to a geometric sum).

\textit{Step 1.2: Propagation of smallness to a ring ``surrounding'' $\Omega \times (1,L)$.} Next, we propagate the smallness from the previous step around the domain $\Omega \times (1,L)$. For this, for fixed $j\in\{1,\dots,N_j\}$, we consider a sequence of points $x_{j,k} = (x_{j,k}',5\rho_1^j r)$, $k\in\{0,1,\dots,N_k\}$, (i.e. the $x_{n+1}$-coordinate for each point $x_{j,k}$ is the same as for the point $x_j$, and also the radius will stay the same) such that $x_{j,0}=x_j$, $B_{\rho_1^jr}(x_{j,k}) \subset B_{2\rho_1^jr}(x_{j,k-1})$, $B_{4\rho_1^jr}(x_{j,k}) \subset \Omega_e \times \R_+$ and such that $B_{2\rho_1^jr}(x_{j,N_k}) \cap B_{2\rho_1^jr}(x_{j,0}) \neq \emptyset$ (the latter condition ensuring that we ``surround'' the domain $\Omega \times (1,L)$). Note that also for all these balls, by construction and since the estimate is true for $B_{4\rho_1^jr}(x_{j})$, it holds that for all $y = (y',y_{n+1}) \in B_{4\rho_1^jr}(x_{j,k})$ we have
\begin{align*}
\frac{\rho_1^j r}{y_{n+1}} \leq Q.
\end{align*}
Moreover, we also note that $N_k$ can be chosen of order $N_k \sim O(1)$.

Applying the three-balls-inequality, Proposition \ref{prop:3ballsinequ2}, iteratively as above, we infer
\begin{align*}
\Vert \tilde{w} \Vert_{L^2(B_{2\rho_1^jr}(x_{j,k}),x_{n+1}^{1-2s})} \leq C L^{p_L} A \varepsilon^{(1-\bar{\alpha})\alpha^{j+k}},
\end{align*}
with the constant $C$ depending on the same quantities as in \eqref{eq:proof_intermediate_integral_3} and $p_L$ can be chosen as in \eqref{eq:proof_intermediate_integral_3} (again compare to a geometric series).

\textit{Step 1.3: Propagation of smallness towards $(\Omega_1 \setminus \Omega_{+\delta}) \times (1,L)$.} Since, $B_{4\rho_1^jr}(x_{j,k}) \subset \Omega_e \times \R_+$, at this point we can now ensure that $\dist(B_{2\rho_1^jr}(x_{j,k}), \Omega) \sim \rho_1^j r$. In order to estimate the contribution
 $\Vert \tilde{w} \Vert_{L^2((\Omega_1 \setminus \Omega_{+\delta}) \times (1,L), x_{n+1}^{1-2s})}$ we need to propagate the smallness back towards $(\Omega_1 \setminus \Omega_{+\delta}) \times (1,L)$. We do this by a sequence of three-balls-inequalities along balls decreasing exponentially in size and by that approaching and covering $(\Omega_1 \setminus \Omega_{+\delta}) \times (1,L)$ (see also Figure \ref{fig:schematic_argument_propagate_smallness_1}).

Denote by $T_{j,k}$ the tube corresponding to moving the ball $B_{2\rho_1^j r}(x_{j,k})$ perpendicular to the $x_{n+1}$-direction and towards $\Omega$ such that $(\Omega_1 \setminus \Omega_{+\delta}) \times (1,L) \subset \bigcup_{j=1}^{N_j} \bigcup_{k=0}^{N_k} T_{j,k}$. By construction, the diameter of such a tube $T_{j,k}$ will be of the order $O(\rho_1^jr)$.

For each tube, we again apply the three-balls-inequality, Proposition \ref{prop:3ballsinequ2}, iteratively. For this, we think of the tubes to be (essentially) divided into $N_l$ layers of balls with each layer having at most $m$-times (where $m$ is some fixed number) as many balls as the previous layer and the radius of the balls in each layer is decreasing exponentially. More precisely, we consider a sequence of centers of balls $x_{j,k,l_{i}}$, where $l\in\{0,1\dots,N_l\}$ denotes the layer and the sub-index $i\in\{1,2,\dots,m^l\}$ denotes one of the (at most) $m^l$-balls in the $l$-th layer, such that $x_{j,k,0_1} = x_{j,k}$, $B_{\rho_1^j \rho_2^{-l} r}(x_{j,k,l_{i_0}}) \subset \bigcup_{i=1}^{m^{l-1}} B_{2\rho_1^j \rho_2^{-(l-1)}r}(x_{j,k,(l-1)_i})$, $B_{4\rho_1^j \rho_2^{-l} r}(x_{j,k,l_i}) \subset \Omega_e \times \R_+$ for all $l\in\{1,\dots,N_l\}$ and $i_0 \in \{1,\dots,m^l\}$. Here $N_l \in \N$, $\rho_2 \in (1,2)$ and the centers of the balls, $x_{j,k,l_i}$, are chosen appropriately such that $(\Omega_1 \setminus \Omega_{+\delta}) \times (1,L) \subset \bigcup_{j=1}^{N_j} \bigcup_{k=0}^{N_k} \bigcup_{l=0}^{N_l} \bigcup_{i=1}^{m^l} B_{2\rho_1^j \rho_2^{-l} r}(x_{j,k,l_i})$. In particular, we can choose $N_l$ such that $N_l \sim O(\log(\frac{\rho_1^jr}{\delta})) \lesssim O(\log(L) + \log(\delta^{-1}))$.

Denote by $r_{j,k,l} := \rho_1^j \rho_2^{-l} r$, the radius of the balls in the $l$-th layer of the tube $T_{j,k}$. As a consequence of the above construction and an iterative application of the three-balls-inequality along the sequence of balls, we infer for the $l$-th layer of the tube $T_{j,k}$
\begin{equation}\label{eq:proof_intermediate_integral_4}
\begin{aligned}
\Vert \tilde{w} \Vert_{L^2(l \text{-th layer}, T_{j,k}, \ x_{n+1}^{1-2s})} &\leq C \max\{r_{j,k,l}^4, r_{j,k,l}^{-3}\} K^{1-\alpha_{j,k,l}} \Vert \tilde{w} \Vert_{L^2( (l-1) \text{-th layer}, T_{j,k}, \ x_{n+1}^{1-2s})}^{\alpha_{j,k,l}}\\
&\leq C L^{p_L} \max\{L^{p_L}, \delta^{-p_\delta}\} A \varepsilon^{(1-\bar{\alpha}) \alpha^{j+k+l}}\\
&\leq C (L^{p_L} + \delta^{-p_\delta}) A \varepsilon^{(1-\bar{\alpha}) \alpha^{j+k+l}}.
\end{aligned}
\end{equation}
For the second to last inequality we have used that $\delta \lesssim r_{j,k,l} \lesssim L$ for all $j\in\{1,\dots,N_j\}$, $k\in\{0,\dots,N_k\}$, $l\in\{0,\dots,N_l\}$. The constants $C$, $p_L$ and $p_\delta$ depend on the same quantities as in \eqref{eq:proof_intermediate_integral_3} (again compare with a geometric sum) and $\alpha$ is chosen similarly as before.

\textit{Step 1.4: Conclusion.} First, we now specify the choice of $\kappa$. We choose $\kappa>0$ so large such that all the balls that were part of the above arguments are contained in $B_{\kappa L}' \times (0,\kappa L)$. More precisely, $\kappa$ is so large such that
\begin{align*}
\bigcup_{j=1}^{N_j} \bigcup_{k=0}^{N_k} \bigcup_{l=0}^{N_l} \bigcup_{i=1}^{m^l} B_{4\rho_1^j \rho_2^{-l} r}(x_{j,k,l_i}) \subset B_{\kappa L}' \times (0,\kappa L).
\end{align*}
Note that the choice of $\kappa$ is independent of the height $L$.

Next, let $l_0^j \in \{0,1,\dots,N_j\}$ be so large such that 
\begin{align*}
(\Omega_1 \setminus \Omega_{+\delta}) \times (1,L) \subset \bigcup_{j=1}^{N_j} \bigcup_{k=0}^{N_k} \bigcup_{l=l_0^j}^{N_l} \{ l \text{-th layer of } T_{j,k} \}.
\end{align*}
 Note that we can choose $l_0^j$ such that $N_l - l_0^j \sim O(\log(\delta^{-1}))$. Using inequality \eqref{eq:proof_intermediate_integral_4} we then estimate
\begin{equation}\label{eq:proof_intermediate_integral_5}
\begin{aligned}
\Vert \tilde{w} \Vert_{L^2((\Omega_1 \setminus \Omega_{+\delta}) \times (1,L),x_{n+1}^{1-2s})} &\leq \sum_{j=1}^{N_j} \sum_{k=0}^{N_k} \sum_{l=l_0^j}^{N_l} \Vert \tilde{w} \Vert_{L^2(l \text{-th layer}, T_{j,k}, \ x_{n+1}^{1-2s})}\\
&\leq C (L^{p_L} +\delta^{-p_\delta}) A \sum_{j=1}^{N_j} \sum_{k=0}^{N_k} \sum_{l=l_0^j}^{N_l} \varepsilon^{(1-\bar{\alpha}) \alpha^{j+k+l}}\\
&\leq C (L^{p_L} +\delta^{-p_\delta}) N_j N_k (N_l - l_0^j) A \varepsilon^{(1-\bar{\alpha})\alpha^{N_j + N_k + N_l}}\\
&\leq C (L^{p_L} +\delta^{-p_\delta}) \log(L) \log(\delta^{-1}) A \varepsilon^{c\alpha^{C (\log(L) + \log(\delta^{-1}))}}\\
&\leq C(L^{p_L} +\delta^{-p_\delta}) A \varepsilon^{c\alpha^{C (\log(L) + \log(\delta^{-1}))}}.
\end{aligned}
\end{equation}
Here, for the second-to-last inequality we have used that $N_j \sim \log(L)$, $N_k \sim O(1)$, $N_l \lesssim O(\log(L) + \log(\delta^{-1}))$ and $N_l - l_0^j \sim O(\log(\delta^{-1}))$. In the last step, we incorporated the logarithmic terms in the algebraic terms (changing the exponents). This finishes the first step.

\textit{Step 2: Estimate on $(\Omega_1 \setminus \Omega_{+\delta}) \times (h,1)$.} For a schematic visualization of this step see Figure \ref{fig:schematic_argument_propagate_smallness_2}. In this second step it suffices to use the a-priori estimate
\begin{align*}
K := \Vert \tilde{w} \Vert_{L^2(B_R'\times(0,2),x_{n+1}^{1-2s})} \leq C_1 A,
\end{align*}
where $R>0$ is chosen large enough such that $\Omega \cup W \subset B_R'(0)$ and such that all the balls which play a role in the following argument are contained in $B_R'\times(0,2)$.

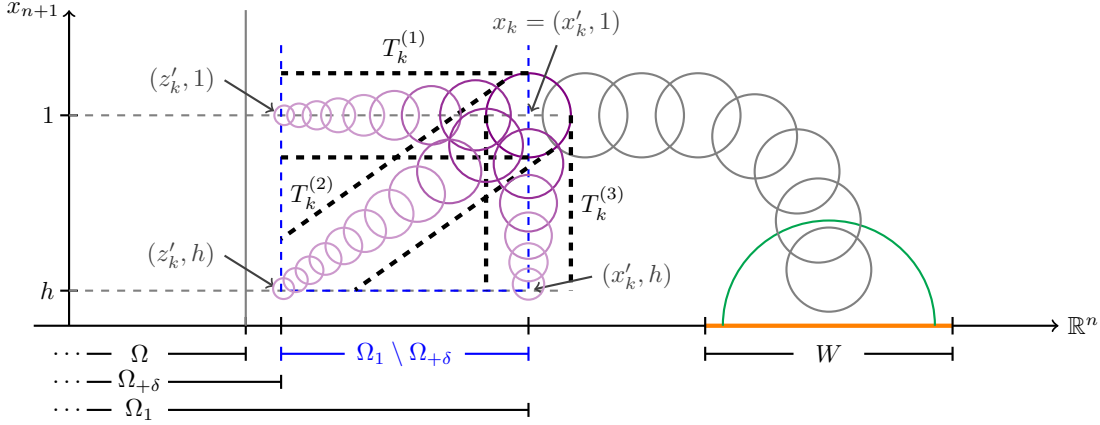
\begin{figure}
	\begin{center}
  	\begin{tikzpicture}[thick,scale=0.93,every node/.style={scale=0.93}]
  	\draw[->] (-0.5,0) -- (14,0) node[right] {$\R^n$};
  	\draw[->] (0,0) -- (0,4.5) node[left] {$x_{n+1}$};
  	\draw (-1/8,3) -- (1/8,3) node[left=2mm] {$1$};
  	\draw (-1/8,0.5) -- (1/8,0.5) node[left=2mm] {$h$};
  	\draw[gray,dashed] (0,0.5) -- (7,0.5);
  	\draw[gray,dashed] (0,3) -- (7,3);
  	
  	\draw (2.5,-1/8) -- (2.5,1/8); \draw[gray] (2.5,0) -- (2.5,4.5);
  	\draw (0.25,-0.4) -- (0.6,-0.4); \draw (1.4,-0.4) -- (2.5,-0.4); \draw (2.5,-0.3) -- (2.5,-0.5); \node at (1,-0.4) {$\Omega$}; \node at (0,-0.4) {$\dots$};
  	
  	\draw[blue] (3,-0.3) -- (3,-0.5); \draw[blue] (3,-0.4) -- (3.9,-0.4); \draw[blue] (5.6,-0.4) -- (6.5,-0.4); \draw[blue] (6.5,-0.3) -- (6.5,-0.5); \node[blue] at (4.75,-0.4) {$\Omega_1 \setminus \Omega_{+\delta}$};
  	
  	\draw (3,-1/8) -- (3,1/8);
  	\draw (0.25,-0.8) -- (0.6,-0.8); \draw (1.4,-0.8) -- (3,-0.8); \draw (3,-0.7) -- (3,-0.9); \node at (1,-0.8) {$\Omega_{+\delta}$}; \node at (0,-0.8) {$\dots$};
  	
  	\draw (6.5,-1/8) -- (6.5,1/8);
  	\draw (0.25,-1.2) -- (0.6,-1-.2); \draw (1.4,-1.2) -- (6.5,-1.2); \draw (6.5,-1.1) -- (6.5,-1.3); \node at (1,-1.2) {$\Omega_1$}; \node at (0,-1.2) {$\dots$};
  	
  	\draw[ultra thick,orange] (9,0) -- (12.5,0); \draw (9,-1/8) -- (9,1/8); \draw (12.5,-1/8) -- (12.5,1/8);
  	\draw[thick] (9,-0.4) -- (10.25,-0.4); \draw[thick] (11.25,-0.4) -- (12.5,-0.4); \draw[thick] (9,-0.5) -- (9,-0.3); \draw[thick] (12.5,-0.5) -- (12.5,-0.3);
  	\node at (10.75,-0.4) {$W$};
  	\draw[Green] (12.25,0) arc(0:180:1.5);
  	
  	\draw[blue,dashed] (3,4) -- (3,0.5) -- (6.5,0.5) -- (6.5,4);
  	
  	\draw[gray] (10.75,0.8) circle(0.6); \draw[gray] (10.6,1.5) circle(0.6); \draw[gray] (10.3,2.2) circle(0.6); \draw[gray] (9.7,2.7) circle(0.6); \draw[gray] (8.9,3) circle(0.6); \draw[gray] (8.1,3) circle(0.6); \draw[gray] (7.3,3) circle(0.6); \draw[violet] (6.5,3) circle(0.6);
  	
  	\draw[black,ultra thick,dashed] (6.5,3.6) -- (3,3.6) node[midway,above] {$T_k^{(1)}$}; \draw[black,ultra thick,dashed] (6.5,2.4) -- (3,2.4);
  	\draw[black,ultra thick,dashed] (5.9,3) -- (5.9,0.5); \draw[black,ultra thick,dashed] (7.1,3) -- (7.1,0.5) node[midway,right] {$T_k^{(3)}$};
  	\draw[black,ultra thick,dashed] (6.151,3.488) -- (3.001,1.238); \draw[black,ultra thick,dashed] (6.848,2.511) -- (4.048,0.511); \node at (3.45,1.9) {$T_k^{(2)}$};
  	
  	\draw[->,darkgray] (6.9,4) -- (6.54,3.05) node[pos=0,above] {$x_k = (x_k',1)$};
  	\draw[->,darkgray] (7.4,0.7) -- (6.58,0.52) node[pos=0,right] {$(x_k',h)$};
  	\draw[->,darkgray] (2.2,3.5) -- (2.9,3.06) node[pos=0,left] {$(z_k',1)$};
  	\draw[->,darkgray] (2.2,1) -- (2.9,0.56) node[pos=0,left] {$(z_k',h)$};
  	
  	\tikzmath{\posx=6.5; \posy=3; \r=0.6; \q4=1.2;}
  	\foreach \j in {1,...,8}{
  		\tikzmath{\radius = \r*\q4^(-\j); \opacity = max(40,100-18*\j);}
  		\pgfmathparse{\posx - 1.5*\radius}
  		\xdef\posx{\pgfmathresult}
  		\draw[violet!\opacity] (\posx,\posy) circle(\radius);
  	}
  	
  	\tikzmath{\posx=6.5; \posy=3; \r=0.6; \q4=1.22;}
  	\foreach \j in {1,...,5}{
  		\tikzmath{\radius = \r*\q4^(-\j); \opacity = max(40,100-18*\j);}
  		\pgfmathparse{\posy - 1.4*\radius}
  		\xdef\posy{\pgfmathresult}
  		\draw[violet!\opacity] (\posx,\posy) circle(\radius);
  	}
  	
  	\tikzmath{\posx=6.5; \posy=3; \r=0.6; \q4=1.15;}
  	\foreach \j in {1,...,10}{
  		\tikzmath{\radius = \r*\q4^(-\j); \opacity = max(40,100-18*\j);}
  		\pgfmathparse{\posx - 1.15*\radius}
  		\xdef\posx{\pgfmathresult}
  		\pgfmathparse{\posy - 0.82*\radius}
%  		\pgfmathparse{\posy - 40*\radius^4 + 125*\radius^6}
  		\xdef\posy{\pgfmathresult}
  		\draw[violet!\opacity] (\posx,\posy) circle(\radius);
  	}
    	\end{tikzpicture}
	\end{center}
	\caption{Schematic visualization of Step 2 in the proof of Proposition \ref{prop:estimate_Omega_ext}. Similarly to Step 1, we propagate the smallness of the boundary data (orange part) into the interior (green half-ball) via the boundary-bulk unique continuation, Proposition \ref{prop:bbucp}, and then to $\partial\Omega_1 \times \{1\}$ via the three-balls-inequality (gray balls), Proposition \ref{prop:3ballsinequ2}. Then we cover $(\Omega_1 \setminus \Omega_{+\delta}) \times (h,1)$ with (violet) balls such that for each ball there exists a sequence of at most $O(\log(h^{-1}) + \log(\delta^{-1}))$ of these balls along which we can apply the three-balls-inequality, Proposition \ref{prop:3ballsinequ2}, to infer smallness on $(\Omega_1 \setminus \Omega_{+\delta}) \times (h,1)$. For visualization, sizes and positions of balls are not true-to-scale.}
\label{fig:schematic_argument_propagate_smallness_2}
\end{figure}

We start similarly as in Step 1. By the quantitative boundary-bulk unique continuation result, Proposition \ref{prop:bbucp}, we transfer the assumed smallness of the boundary data into the bulk. Along a sequence of balls (see Figure \ref{fig:schematic_argument_propagate_smallness_2}) we propagate the smallness with the help of the three-balls-inequality, Proposition \ref{prop:3ballsinequ2}, to $\partial\Omega_1 \times \{1\}$. Note that we can choose the balls to be all of the same radii $2r$. In particular $r \sim O(1)$ and since the first ball will satisfy the condition $\frac{r}{y_{n+1}} \leq Q$ for some $Q>0$ for all $y=(y',y_{n+1}) \subset B_{4r}(x)$, this condition will hold true for all the balls in this argument with the same $Q$. Next, we cover $\partial\Omega_1 \times \{1\}$ (cf. Step 1.2 above) with balls $B_{2r}(x_k)$ with $x_k = (x_k',1)$, $x_k' \in \partial\Omega_1$ such that $\partial\Omega_1 \times \{1\} \subset \bigcup_{k=1}^{N_k} B_{2r}(x_k)$, $N_k \sim O(1)$. Since we need to apply the three-balls-inequality at most $N_1 \sim O(1)$-times, we have for all balls $B_{2r}(x_k)$
\begin{align*}
\Vert \tilde{w} \Vert_{L^2(B_{2r}(x_k), x_{n+1}^{1-2s})} \leq C A \varepsilon^{c \alpha^{N_1}}.
\end{align*}

In a next step, we cover the set $(\Omega_1 \setminus \Omega_{+\delta}) \times (h,1)$ with balls $B_{2r_{k,l}}(x_{k,l})$, $k\in\{1,\dots,N_k\}$, $l\in\{1,\dots,N_l\}$, decreasing exponentially in size when approaching $\partial\Omega_{+\delta} \times (h,1)$ and $(\Omega_1 \setminus \Omega_{+\delta}) \times \{h\}$. These balls satisfy the following conditions:
\begin{itemize}
\item $B_{4r_{k,l}}(x_{k,l}) \subset \Omega_e \times \R_+$.
\item For some $c\in(0,1)$ we have $r_{k,l} \geq c \min\{h,\delta\}$ for all $k\in\{1,\dots,N_k\}$, $l\in\{1,\dots,N_l\}$.
\item For some uniform $Q>0$ (uniform in $k$ and $l$), we have $\frac{r_{k,l}}{y_{n+1}} \leq Q$ for all $y=(y',y_{n+1}) \in B_{4r_{k,l}}(x_{k,l})$ for all $k\in\{1,\dots,N_k\}$, $l \in \{1,\dots,N_l\}$. In particular, this means that $r_{k,l}$ is bounded in terms of the height $x_{n+1}$, which can be achieved easily since, when approaching $(\Omega_1 \setminus \Omega_{+\delta}) \times \{h\}$, we simply decrease the size of the balls appropriately.
\item For each of these balls there exists a sequence of balls starting at $B_{2r}(x_k)$ for some $k\in\{1,\dots,N_k\}$ of length at most $N_{h,\delta} \sim O(\log(h^{-1})+\log(\delta^{-1}))$ such that the three-balls-inequality, Proposition \ref{prop:3ballsinequ2}, can be applied along this sequence (cf. Figure \ref{fig:schematic_argument_propagate_smallness_2}). More precisely, we construct the described sequences of balls by considering similar tubes $T_{k}$ as in Step 1. However, this time, from one ball, $B_{2r}(x_k)$, we do not only consider one tube perpendicular to the $x_{n+1}$-direction towards $\Omega$, but multiple tubes $T_k^{(i)}$, $i\in\{1,\dots,N_i\}$, into different directions such that the full slice $\operatorname{Sl}_k$ is covered by these tubes. Here, $\operatorname{Sl}_k \subset (\Omega_1 \setminus \Omega_{+\delta}) \times (h,1)$ denotes the slice of $(\Omega_1 \setminus \Omega_{+\delta}) \times (h,1)$, spanned by the four points $(x_k',1)$, $(z_k',1)$, $(z_k', h)$ and $(x_k',h)$ but with the width of the ball $B_{2r}(x_k)$. The point $z_k'$ denotes the projection of $x_k'$ onto $\partial \Omega_{+\delta}$ when moving $x_k$ perpendicular to the $x_{n+1}$-direction towards $\Omega$  (see Figure \ref{fig:schematic_argument_propagate_smallness_2} for a schematic visualization). Note that we need at most $N_l \sim O(\log(h^{-1})+\log(\delta^{-1}))$ of these balls to cover the slice $\operatorname{Sl}_k$ in the described manner.
\end{itemize}
For each ball $B_{2r_{k,l}}(x_{k,l})$, $k\in\{1,\dots,N_k\}$, $l\in\{1,\dots,N_l\}$, we then have
\begin{align*}
\Vert \tilde{w} \Vert_{L^2(B_{2r_{k,l}}(x_{k,l}), x_{n+1}^{1-2s})} \leq C \min_{\substack{k\in\{1,\dots,N_k\} \\ l\in\{1,\dots,N_l\}}} \{r_{k,l}\}^{-p_r} A \varepsilon^{c \alpha^{N_1 + N_{h,\delta}}} \leq C (h^{-p_h} + \delta^{-p_\delta}) A \varepsilon^{c \alpha^{N_1 + N_{h,\delta}}}
\end{align*}
for some $p_r, p_h, p_\delta>0$ (again compare with a geometric series to get uniform exponents $p_r$, $p_h$ and $p_\delta$ over $k\in\{1,,\dots,N_k\}$, $l \in \{1,\dots,N_l\}$). Consequently, we estimate
\begin{equation}\label{eq:proof_intermediate_integral_6}
\begin{aligned}
\Vert \tilde{w} \Vert_{L^2((\Omega_1 \setminus \Omega_{+\delta}) \times (h,1), x_{n+1}^{1-2s})} &\leq \sum_{k=1}^{N_k} \sum_{l=1}^{N_l} \Vert \tilde{w} \Vert_{L^2(B_{2r_{k,l}}(x_{k,l}), x_{n+1}^{1-2s})}\\
&\leq C A (h^{-p_h} + \delta^{-p_\delta}) (\log(h^{-1}) + \log(\delta^{-1})) \varepsilon^{c \alpha^{C (\log(h^{-1}) + \log(\delta^{-1}))}}\\
&\leq C A (h^{-p_h} + \delta^{-p_\delta}) \varepsilon^{c \alpha^{C (\log(h^{-1}) + \log(\delta^{-1}))}},
\end{aligned}
\end{equation}
where in the last inequality we have incorporated the logarithmic terms in the algebraic terms (changing the exponents)

\textit{Step 3: Conclusion.} We conclude the proof by combining the estimates \eqref{eq:proof_intermediate_integral_5} and \eqref{eq:proof_intermediate_integral_6}.
\end{proof}

Stepping back, we highlight that both unique continuation estimates in the regimes $(\Omega_1 \setminus \Omega_{+\delta}) \times (1,L)$ and in $(\Omega_1 \setminus \Omega_{+\delta}) \times (h,1)$ rely on propagation of smallness estimates. The detailed strategy however deviates slightly due to the size restrictions on the balls imposed by our three-balls inequality which in turn are a consequence of the lack of translation invariance of our equation. In particular, the propagation direction in the region $(\Omega_1 \setminus \Omega_{+\delta}) \times (h,1)$ is slightly more involved in order to avoid additional (logarithmic) losses due to the decreasing size of the balls which have to decrease both towards $\Omega_{+\delta}$ and towards $x_{n+1}=h$.

\subsection{Estimating the inhomogeneous term}
After having deduced the central propagation of smallness estimates in the previous section, we next turn to the final ingredient in the proof of Proposition \ref{prop:QUCP_1} in which we seek to deduce the following estimate for the inhomogeneous term. As outlined above, such a contribution arises in our analysis in the domain $\Omega_1 \setminus \Omega$.

\begin{prop}\label{prop:inhomogeneous_term}
Let $\Omega \subset \R^n$ be open, non-empty, bounded and Lipschitz such that $\R^n \setminus \Omega$ is connected and let $W \subset \Omega_e$ be open, bounded, non-empty and Lipschitz such that $\overline{\Omega} \cap \overline{W} = \emptyset$. Let $\theta_1\in(0,1)$. Let $a_1, a_2 \in L^\infty(\R^n, \R^{n \times n}_{\text{sym}})$ satisfy the assumptions (A1) with the given $\theta_1$, and (A3). Let $\Omega_1 \subset \R^n$ be an open, bounded, Lipschitz set such that $\Omega \Subset \Omega_1$ and $\overline{\Omega_1} \cap \overline{W} = \emptyset$. Let $f \in \widetilde{H}^s(W)$ and let $\tilde{u}_j^f \in \dot{H}^1(\R^{n+1}_+, x_{n+1}^{1-2s})$, $j \in \{1,2\}$, satisfy
\begin{equation*}
\begin{cases}
\begin{alignedat}{2}
-\nabla \cdot x_{n+1}^{1-2s} \tilde{a}_j \nabla \tilde{u}_j^f &= 0 \quad &&\text{in } \R^{n+1}_+,\\
-c_s \lim_{x_{n+1}\to0} x_{n+1}^{1-2s} \partial_{n+1} \tilde{u}_j^f &= 0 \quad &&\text{on } \Omega\times\{0\},\\
\tilde{u}_j^f &= f \quad &&\text{on } \Omega_e \times \{0\}.
\end{alignedat}
\end{cases}
\end{equation*}

There exist constants $C>0$ and $\beta>0$ such that, if for some $\varepsilon \in (0,\frac{1}{2})$
\begin{align*}
\Vert \restr{\lim_{x_{n+1}\to0} x_{n+1}^{1-2s} \partial_{n+1} (\tilde{u}_1^f - \tilde{u}_2^f)}{W} \Vert_{H^{-s}(W)} \leq \varepsilon \Vert f \Vert_{\widetilde{H}^s(W)}.
\end{align*}
then it holds
\begin{align*}
\Vert \lim_{x_{n+1}\to0} x_{n+1}^{1-2s} \partial_{n+1} (\tilde{u}_1^f - \tilde{u}_2^f) \Vert_{H^{-1}(\Omega_1 \setminus \Omega)} \leq C \vert \log(\varepsilon) \vert^{-\beta} \Vert f \Vert_{\widetilde{H}^s(W)}.
\end{align*}
The constants $C$ and $\beta$ only depend on $n$, $s$, $\Omega$, $\Omega_1$, $W$ and $\theta_1$.
\end{prop}

\begin{proof}
Let $\Omega_{+\delta}, \Omega_2, \Omega_3 \subset \R^n$ be open, bounded, Lipschitz sets such that $\Omega \Subset \Omega_{+\delta} \Subset \Omega_1 \Subset \Omega_2 \Subset \Omega_3$ with $\overline{\Omega_3} \cap \overline{W} = \emptyset$ and such that $\dist(\Omega, \partial\Omega_{+\delta}) \sim \delta$ for $\delta>0$ small. Note that
\begin{equation}\label{eq:proof_inhomogeneous_term_1}
\begin{aligned}
\Vert \lim_{x_{n+1}\to0} &x_{n+1}^{1-2s} \partial_{n+1} (\tilde{u}_1^f - \tilde{u}_2^f) \Vert_{H^{-1}(\Omega_1 \setminus \Omega)}\\
&\leq \Vert \lim_{x_{n+1}\to0} x_{n+1}^{1-2s} \partial_{n+1} \tilde{u}_1^f \Vert_{H^{-1}(\Omega_{+\delta} \setminus \Omega)} + \Vert \lim_{x_{n+1}\to0} x_{n+1}^{1-2s} \partial_{n+1} \tilde{u}_2^f \Vert_{H^{-1}(\Omega_{+\delta} \setminus \Omega)}\\
&\hspace{10mm}+ \Vert \lim_{x_{n+1}\to0} x_{n+1}^{1-2s} \partial_{n+1} (\tilde{u}_1^f - \tilde{u}_2^f) \Vert_{H^{-1}(\Omega_1 \setminus \Omega_{+\delta})}.
\end{aligned}
\end{equation}
We will estimate the different terms on the right hand side separately.

\textit{Step 1: Estimate on $\Omega_{+\delta} \setminus \Omega$.}
For the first two terms in \eqref{eq:proof_inhomogeneous_term_1} we observe that for $D_\delta \subset \Omega_e$ and $F \in H^{-s}(D_\delta)$ it holds that
\begin{align*}
\Vert F \Vert_{H^{-1}(D_\delta)} &= \sup_{\substack{ G \in \widetilde{H}^1(D_\delta) \\ \Vert G \Vert_{\widetilde{H}^1(D_\delta)}=1}} \langle F, G \rangle_{H^{-1}(D_\delta),\widetilde{H}^1(D_\delta)} \leq \sup_{\substack{ G \in \widetilde{H}^{1}(D_\delta) \\ \Vert G \Vert_{\widetilde{H}^1(D_\delta)}=1}} \langle F, G \rangle_{H^{-s}(D_\delta),\widetilde{H}^s(D_\delta)}\\
&\leq \sup_{\substack{ G \in \widetilde{H}^{1}(D_\delta) \\ \Vert G \Vert_{\widetilde{H}^1(D_\delta)}=1}} \Vert F \Vert_{H^{-s}(D_\delta)} \Vert G \Vert_{\widetilde{H}^s(D_\delta)} \leq \sup_{\substack{ G \in \widetilde{H}^1(D_\delta) \\ \Vert G \Vert_{\widetilde{H}^1(D_\delta)}=1}} C \vert D_\delta \vert^{p_\delta} \Vert F \Vert_{H^{-s}(D_\delta)} \Vert G \Vert_{\widetilde{H}^1(D_\delta)}\\
&\leq C \vert D_\delta \vert^{p_\delta} \Vert F \Vert_{H^{-s}(D_\delta)}.
\end{align*}
Here, for the fourth inequality we used that $\Vert G \Vert_{\widetilde{H}^s(D_\delta)} \leq C \vert D_\delta \vert^{p_\delta} \Vert G \Vert_{\widetilde{H}^1(D_\delta)}$. Indeed, on the one hand we have by Hölder's inequality and Sobolev embedding for $p = \frac{2n}{n-2} > 2$
\begin{align*}
\Vert G \Vert_{L^2(D_\delta)} \leq C \vert D_\delta \vert^{p_\delta} \Vert G \Vert_{L^p(D_\delta)} \leq C \vert D_\delta \vert^{p_\delta} \Vert G \Vert_{\widetilde{H}^1(D_\delta)}.
\end{align*}

On the other hand, for $G \in \widetilde{H}^1(D_{\delta})$
\begin{align*}
\|(-\D)^{s/2} G\|_{L^2(\R^n)} \leq C |D_{\delta}|^{\frac{1-s}{n}} \|(-\D)^{1/2} G\|_{L^2(\R^n)}
\leq C |D_{\delta}|^{\frac{1-s}{n}} \| G\|_{\dot{H}^1(\R^n)},
\end{align*}
see, for instance, \cite[Remark 3.8, Theorems 1.5 and 3.17]{CMR20}, applied to $(-\D)^{s/2}G$.
By the definition of the $L^2$-based (fractional) Sobolev norms as $\|G\|_{H^s(\R^n)} \sim \|G\|_{L^2(\R^n)} + \|(-\D)^{s/2} G\|_{L^2(\R^n)}$, the previous two estimates imply the desired bound.

Consequently, taking $F = \lim_{x_{n+1}\to0} x_{n+1}^{1-2s} \partial_{n+1} \tilde{u}_j^f$ and $D_\delta = \Omega_{+\delta} \setminus \Omega$, we infer
\begin{align}\label{eq:proof_inhomogeneous_term_2}
\Vert \lim_{x_{n+1}\to0} x_{n+1}^{1-2s} \partial_{n+1} \tilde{u}_j^f \Vert_{H^{-1}(\Omega_{+\delta} \setminus \Omega)} \leq C \delta^{p_\delta} \Vert \lim_{x_{n+1}\to0} x_{n+1}^{1-2s} \partial_{n+1} \tilde{u}_j^f \Vert_{H^{-s}(\Omega_{+\delta} \setminus \Omega)} \leq C \delta^{p_\delta} \Vert f \Vert_{\widetilde{H}^s(W)}
\end{align}
for some $p_\delta>0$ depending only on $n$ and $s$.

\textit{Step 2: Estimate on $\Omega_1 \setminus \Omega_{+\delta}$.}
We split this step itself into two substeps.

\textit{Step 2.1: Higher integrability of $\tilde{u}_j^f$.}
We will prove that for some $p>2$ we have $\tilde{u}_j^f \in W^{1,p}((\Omega_2 \setminus \Omega_{+\delta}) \times (0,1), x_{n+1}^{1-2s})$ with
\begin{align}\label{eq:proof_inhomogeneous_term_3}
\Vert x_{n+1}^{\frac{1-2s}{p}} \nabla \tilde{u}_j^f \Vert_{L^p((\Omega_2 \setminus \Omega_{+\delta}) \times(0,1))} \leq C \delta^{-\bar{p}_\delta} \Vert x_{n+1}^{\frac{1-2s}{2}} \nabla \tilde{u}_j^f \Vert_{L^2((\Omega_3 \setminus \Omega_{+\frac{\delta}{2}}) \times (0,2))} \leq C \delta^{-\bar{p}_\delta} \Vert f \Vert_{\widetilde{H}^s(W)}
\end{align}
for some $\bar{p}_\delta>0$. We derive this by applying Gehring's lemma (see Lemma \ref{lem:Gehring} for the precise formulation). To this end, we first verify that the assumptions of Gehring's lemma are satisfied. After an odd reflection (note that we work in $(\Omega_3 \setminus \Omega_{+\frac{\delta}{2}}) \times \R_+$ where we have vanishing Dirichlet data, and thus an odd reflection can be used), we argue by the following line of inequalities for any ball $B_r \subset (\Omega_3 \setminus \Omega_{+\frac{\delta}{2}}) \times (-2,2)$ with $B_{2r} \subset (\Omega_3 \setminus \Omega_{+\frac{\delta}{2}}) \times (-2,2)$: For some $q \in (1,2)$, it holds that
\begin{align*}
\left( \frac{1}{w(B_r)} \int_{B_r} \vert x_{n+1} \vert^{1-2s} \vert \nabla \tilde{u}_j^f \vert^2 dx \right)^{\frac{1}{2}} &\leq Cr^{-1} \left( \frac{1}{w(B_{2r})} \int_{B_{2r}} \vert x_{n+1} \vert^{1-2s} \vert \tilde{u}_j^f \vert^2 dx \right)^{\frac{1}{2}}\\
& \leq C \left( \frac{1}{w(B_{2r})} \int_{B_{2r}} \vert x_{n+1} \vert^{1-2s} \vert \nabla \tilde{u}_j^f \vert^{q} dx \right)^{\frac{1}{q}}.
\end{align*}
Here, we applied Caccioppoli's inequality for the first inequality (cf. Lemma \ref{lem:Caccioppoli_inequality}), and the weighted Sobolev embedding (see Theorem 1.2 in \cite{FKS82}) for the second inequality for some $q<2$ such that $\Vert \tilde{u}_j^f \Vert_{L^2(B_{2r}, x_{n+1}^{1-2s})} \leq \Vert \nabla \tilde{u}_j^f \Vert_{L^q(B_{2r}, x_{n+1}^{1-2s})}$ (i.e. $q$ is chosen such that the weighted Sobolev exponent $q^*$ satisfies $q^* \geq 2$). We used the notation $w(B_r) := \int_{B_r} \vert x_{n+1} \vert^{1-2s} dx$. Consequently, we infer that
\begin{align*}
\left( \frac{1}{w(B_r)} \int_{B_r} \vert x_{n+1} \vert^{1-2s} ( \vert \nabla \tilde{u}_j^f \vert^q )^{\frac{2}{q}} dx \right)^{\frac{q}{2}} \leq C \frac{1}{w(B_{2r})} \int_{B_{2r}} \vert x_{n+1} \vert^{1-2s} \vert \nabla \tilde{u}_j^f \vert^{q} dx.
\end{align*}
Since $q < 2$ and thus $\frac{2}{q}>1$, we can apply the (weighted) Gehring Lemma (see Lemma \ref{lem:Gehring}) and sum over all those balls to get for some $q_1>\frac{2}{q}$ and some $\bar{p}_\delta>0$
\begin{align*}
\left( \int_{(\Omega_2 \setminus \Omega_{+\delta}) \times (0,1)} x_{n+1}^{1-2s} (\vert \nabla \tilde{u}_j^f \vert^{q} )^{q_1} dx \right)^{\frac{1}{q_1}} &\leq C \delta^{-\bar{p}_\delta} \left( \int_{(\Omega_3 \setminus \Omega_{+\frac{\delta}{2}}) \times (0,2)} x_{n+1}^{1-2s} \vert \nabla \tilde{u}_j^f \vert^{2} dx \right)^{\frac{q}{2}}.
\end{align*}
Choosing $p = q q_1 > 2$ finishes the proof of the claim \eqref{eq:proof_inhomogeneous_term_3}.

\textit{Step 2.2: Proof of the main estimate.}
We note that
\begin{equation}\label{eq:proof_inhomogeneous_term_4}
\begin{aligned}
\Vert \lim_{x_{n+1}\to0}& x_{n+1}^{1-2s} \partial_{n+1} (\tilde{u}_1^f - \tilde{u}_2^f) \Vert_{H^{-1}(\Omega_1 \setminus \Omega_{+\delta})} \leq \Vert \lim_{x_{n+1}\to0} x_{n+1}^{1-2s} \partial_{n+1} (\tilde{u}_1^f - \tilde{u}_2^f) \Vert_{H^{-s}(\Omega_1 \setminus \Omega_{+\delta})}\\
& \leq C \delta^{-\bar{p}_\delta} \Vert x_{n+1}^{\frac{1-2s}{2}} \nabla (\tilde{u}_1^f - \tilde{u}_2^f) \Vert_{L^2((\Omega_2 \setminus \Omega_{+\frac{\delta}{2}}) \times (0,1))} \\
&\leq C \delta^{-\bar{p}_\delta} \Big( \Vert x_{n+1}^{\frac{1-2s}{2}} \nabla \tilde{u}_1^f \Vert_{L^2((\Omega_2 \setminus \Omega_{+\frac{\delta}{2}}) \times (0,h))} + \Vert x_{n+1}^{\frac{1-2s}{2}} \nabla \tilde{u}_2^f \Vert_{L^2((\Omega_2 \setminus \Omega_{+\frac{\delta}{2}}) \times (0,h))}\\
&\hspace{20mm} + \Vert x_{n+1}^{\frac{1-2s}{2}} \nabla (\tilde{u}_1^f - \tilde{u}_2^f) \Vert_{L^2((\Omega_2 \setminus \Omega_{+\frac{\delta}{2}}) \times (h,1))} \Big)
\end{aligned}
\end{equation}
for some $\bar{p}_{\delta}>0$. For the first two terms on the right hand side of 
\eqref{eq:proof_inhomogeneous_term_4}, we then estimate, using \eqref{eq:proof_inhomogeneous_term_3}, for $p>2$ as in Step 2.1 and  $q>1$ such that $\frac{1}{p} + \frac{1}{q} = \frac{1}{2}$
\begin{align*}
\Vert x_{n+1}^{\frac{1-2s}{2}} \nabla \tilde{u}_j^f \Vert_{L^2((\Omega_2 \setminus \Omega_{+\frac{\delta}{2}}) \times (0,h))} &\leq \Vert x_{n+1}^{\frac{1-2s}{q}} \Vert_{L^q((\Omega_2 \setminus \Omega_{+\frac{\delta}{2}}) \times (0,h))} \Vert x_{n+1}^{\frac{1-2s}{p}} \nabla \tilde{u}_j^f \Vert_{L^p((\Omega_2 \setminus \Omega_{+\frac{\delta}{2}}) \times (0,h))}\\
&\leq C \delta^{-\bar{p}_\delta} h^{p_h} \Vert f \Vert_{\widetilde{H}^s(W)}.
\end{align*}
For the third term on the right hand side of \eqref{eq:proof_inhomogeneous_term_4}, we observe that we are now exactly in the position to apply Proposition \ref{prop:estimate_Omega_ext} to $\tilde{w} := \tilde{u}_1^f - \tilde{u}_2^f$ with $A := \Vert f \Vert_{\widetilde{H}^s(W)}$ to get
\begin{align*}
\Vert x_{n+1}^{\frac{1-2s}{2}} \nabla (\tilde{u}_1^f - \tilde{u}_2^f) \Vert_{L^2((\Omega_2 \setminus \Omega_{+\frac{\delta}{2}}) \times (h,1))} \leq C (h^{-\bar{p}_h} + \delta^{-\bar{p}_\delta}) \varepsilon^{c \alpha^{C (\log(h^{-1}) + \log(\delta^{-1}))}} \Vert f \Vert_{\widetilde{H}^s(W)}.
\end{align*}
Putting the above observations together, we arrive at
\begin{align*}
\Vert \lim_{x_{n+1}\to0} x_{n+1}^{1-2s}& \partial_{n+1} (\tilde{u}_1^f - \tilde{u}_2^f) \Vert_{H^{-1}(\Omega_1 \setminus \Omega_{+\delta})}\\
&\leq C \delta^{-\bar{p}_\delta} h^{p_h} \Vert f \Vert_{\widetilde{H}^s(W)} + C (h^{-\bar{p}_h} + \delta^{-\bar{p}_{\delta}}) \varepsilon^{c \alpha^{C (\log(h^{-1}) + \log(\delta^{-1}))}} \Vert f \Vert_{\widetilde{H}^s(W)}. 
\end{align*}

\textit{Step 3: Conclusion.}
To conclude the estimate, we just need to balance $h$ and $\delta$ correctly. We abbreviate the notation by writing $\omega(\varepsilon) := \vert\log(\varepsilon)\vert^{-\beta}$, where $\beta>0$ will be determined below. We combine the previous estimates from Step 1 and Step 2 to get
\begin{equation}\label{eq:proof_inhomogeneous_term_5}
\begin{aligned}
\Vert \lim_{x_{n+1}\to0} &x_{n+1}^{1-2s} \partial_{n+1} (\tilde{u}_1^f - \tilde{u}_2^f) \Vert_{H^{-1}(\Omega_1 \setminus \Omega)}\\
&\leq \Vert \lim_{x_{n+1}\to0} x_{n_+1}^{1-2s} \partial_{n+1} \tilde{u}_1^f \Vert_{H^{-1}(\Omega_{+\delta} \setminus \Omega)} + \Vert \lim_{x_{n+1}\to0} x_{n_+1}^{1-2s} \partial_{n+1} \tilde{u}_2^f \Vert_{H^{-1}(\Omega_{+\delta} \setminus \Omega)}\\
&\hspace{10mm}+ \Vert \lim_{x_{n+1}\to0} x_{n_+1}^{1-2s} \partial_{n+1} (\tilde{u}_1^f - \tilde{u}_2^f) \Vert_{H^{-1}(\Omega_1 \setminus \Omega_{+\delta})}\\
&\leq C \delta^{p_\delta} \Vert f \Vert_{\widetilde{H}^s(W)} + C \delta^{-\bar{p}_\delta} h^{p_h} \Vert f \Vert_{\widetilde{H}^s(W)}\\ 
& \qquad + C (h^{-\bar{p}_h} + \delta^{-\bar{p}_{\delta}}) \varepsilon^{c \alpha^{C (\log(h^{-1}) + \log(\delta^{-1}))}} \Vert f \Vert_{\widetilde{H}^s(W)}.
\end{aligned}
\end{equation}
If we choose $\delta = \omega(\varepsilon)^{\frac{1}{p_\delta}}$ and $h = \omega(\varepsilon)^{\frac{1+\frac{\bar{p}_\delta}{p_\delta}}{p_h}}$, the first two terms on the right hand side of \eqref{eq:proof_inhomogeneous_term_5} are already in order. For the last term, we observe that with this choice of $\delta$ and $h$ we have
\begin{align*}
(h^{-\bar{p}_h} + \delta^{-\bar{p}_{\delta}}) \varepsilon^{c \alpha^{C (\log(h^{-1}) + \log(\delta^{-1}))}} &\leq C \omega(\varepsilon)^{-p_\omega} \varepsilon^{c\alpha^{C \vert \log(\omega(\varepsilon)) \vert}} = C\omega(\varepsilon)^{-p_{\omega}} \varepsilon^{c\omega(\varepsilon)^{-C\log(\alpha)}}\\
&= C \omega(\varepsilon)^{-p_\omega} \exp\left( c\log(\varepsilon) \omega(\varepsilon)^{C \vert\log(\alpha)\vert} \right)
\end{align*}
for some $p_\omega>0$. Note that the constant $C$ did not depend on $\beta$. Thus, we can take
\begin{align*}
0 < \beta < \frac{1}{C \vert\log(\alpha)\vert}.
\end{align*}
If we plug back into $\omega(\varepsilon) = \vert\log(\varepsilon)\vert^{-\beta}$, we find that
\begin{align*}
C \omega(\varepsilon)^{-p_\omega} \exp\left( c\log(\varepsilon) \omega(\varepsilon)^{C \vert\log(\alpha)\vert} \right) \leq C \vert\log(\varepsilon)\vert^{\beta p_\omega} \exp\left( -c \frac{\vert\log(\varepsilon)\vert}{\vert\log(\varepsilon)\vert^{C \beta \vert\log(\alpha)\vert}} \right) \leq C_1 \vert\log(\varepsilon)\vert^{-\beta},
\end{align*}
where $C_1$ may now depend on $\beta$. The last inequality holds true since (writing $z=\vert\log(\varepsilon)\vert$) the term $z^a \exp(-cz^b)$ converges quicker to $0$ than $z^{-d}$ for any $a,b,c,d>0$ as $z\to\infty$. Thus, we have shown that with this choice of $\delta$ and $h$ it indeed follows that
\begin{align*}
\Vert \lim_{x_{n+1}\to0}& x_{n+1}^{1-2s} \partial_{n+1} (\tilde{u}_1^f - \tilde{u}_2^f) \Vert_{H^{-1}(\Omega_1 \setminus \Omega)} \leq C_1 \vert \log(\varepsilon) \vert^{-\beta} \Vert f \Vert_{\widetilde{H}^s(W)}
\end{align*}
for $\beta>0$ as above. This finishes the proof.
\end{proof}

\subsection{Proof of Proposition \ref{prop:QUCP_1} -- estimate for full integral}

Now we come to the proof of the main proposition of this section, Proposition \ref{prop:QUCP_1}. Here, we will apply Lemma \ref{lem:apriori_fract_cond}, Propositions \ref{prop:estimate_thin_annuli} and \ref{prop:estimate_Omega_ext} to $v_1^f - v_2^f := \int_0^\infty t^{1-2s} \left( \tilde{u}_1^f (\cdot,t) - \tilde{u}_2^f (\cdot,t) \right) dt$, and Proposition \ref{prop:inhomogeneous_term}.

\begin{proof}[Proof of Proposition \ref{prop:QUCP_1}]
Let $\Omega_{+\delta}, \Omega_1 \subset \R^n$ be open, bounded, Lipschitz such that $\Omega \Subset \Omega_{+\delta} \Subset \Omega_1$ with $\dist(\Omega,\partial\Omega_{+\delta}) \sim \delta$ for $\delta>0$ small and such that $\overline{\Omega_1} \cap \overline{W} = \emptyset$.

In the following, we will write $a := a_1 = a_2$ in $\Omega_e$. Moreover, $C$ will denote a generic constant, which may only depend on $n$, $s$, $\Omega$, $\Omega'$, $\Omega_1$, $W$, $\theta_1$, $\theta_2$ and $\theta_3$, whereas the constants $C_1$ and $C_2$ may additionally depend on $\beta$. In order to save notation, for $\beta_1>0$ to be determined later, we will sometimes write
\begin{align*}
\omega(\varepsilon) := \vert \log(\varepsilon) \vert^{-\beta_1}
\end{align*}
and for $j \in \{1,2\}$ we write
\begin{align*}
v_j^f := \int_0^\infty t^{1-2s} \tilde{u}_j^{f}(\cdot,t) dt.
\end{align*}
We divide the proof into several steps.

\textit{Step 1: Splitting the integral.} By Theorem 3 in \cite{CGRU23} we know that $v_1^f-v_2^f$ solves in $D \subset \Omega_e \setminus \Omega$
\begin{align*}
\nabla'\cdot a\nabla'(v_1^f-v_2^f) = (-\nabla'\cdot a_1 \nabla')^s u_1^f - (-\nabla'\cdot a_2 \nabla')^s u_2^f = - c_s \lim_{x_{n+1} \to 0} x_{n+1}^{1-2s} \partial_{n+1} (\tilde{u}_1^{f} - \tilde{u}_2^{f}),
\end{align*}
where $a = a_1 = a_2$ in $D$ by assumption. Thus, by the weak definition of the normal derivative, we have
\begin{equation}\label{eq:estimate_normal_derivative}
\begin{aligned}
\Vert \partial_\nu^a (v_1^f-v_2^f) &\Vert_{H^{-1/2}(\partial D)} = \sup_{\substack{h \in H^{1/2}(\partial D), \\ \Vert h \Vert_{H^{1/2}(\partial D)} = 1}} \langle( \partial_\nu^a (v_1^f-v_2^f), \ h \rangle_{H^{-1/2}(\partial D),H^{1/2}(\partial D)}\\
&= \sup_{\substack{h \in H^{1/2}(\partial D), \\ \Vert h \Vert_{H^{1/2}(\partial D)} = 1}} \ \left( \int_D a \nabla'(v_1^f-v_2^f) \cdot \nabla'\phi_h \ dx' \right.\\
&\hspace{30mm}  \left. -  c_s \langle \lim_{x_{n+1} \to 0} x_{n+1}^{1-2s} \partial_{n+1} (\tilde{u}_1^{f} - \tilde{u}_2^{f}), \ \phi_h \rangle_{H^{-1}(D),H^{1}(D)} \right)\\
&\leq \sup_{\substack{h \in H^{1/2}(\partial D), \\ \Vert h \Vert_{H^{1/2}(\partial D)} = 1}} \left( \Vert a \Vert_{L^\infty(D)} \Vert \nabla'(v_1^f-v_2^f) \Vert_{L^2(D)} \Vert \nabla'\phi_h \Vert_{L^2(D)} \right.\\
&\hspace{30mm} \left. + C \Vert \lim_{x_{n+1} \to 0} x_{n+1}^{1-2s} \partial_{n+1} (\tilde{u}_1^{f} - \tilde{u}_2^{f}) \Vert_{H^{-1}(D)} \Vert \phi_h \Vert_{H^1(D)} \right)\\
&\leq C \left( \Vert v_1^f - v_2^f \Vert_{H^1(D)} + \Vert \lim_{x_{n+1} \to 0} x_{n+1}^{1-2s} \partial_{n+1} (\tilde{u}_1^{f} - \tilde{u}_2^{f}) \Vert_{H^{-1}(D)} \right),
\end{aligned}
\end{equation}
where $\phi_h \in H^1(D)$ is some extension of $h$, i.e. $\phi|_{\partial D} = h$, with $\Vert \phi_h \Vert_{H^1(D)} \leq C \Vert h \Vert_{H^{1/2}(\partial D)}$.

We then split the desired estimate from Proposition \ref{prop:QUCP_1} into several bounds. Using \eqref{eq:estimate_normal_derivative}, we obtain 
\begin{equation}\label{eq:proof_full_integral_splitting}
\begin{aligned}
&\left\Vert \int_0^\infty t^{1-2s} \left( \tilde{u}_1^{f} (\cdot,t) - \tilde{u}_2^{f} (\cdot,t) \right) dt \right\Vert_{H^{\frac{1}{2}}(\partial\Omega)} + \left\Vert \partial_\nu^a \left( \int_0^\infty t^{1-2s} \left( \tilde{u}_1^{f} (\cdot,t) - \tilde{u}_2^{f} (\cdot,t) \right) dt \right) \right\Vert_{H^{-\frac{1}{2}}(\partial\Omega)}\\
&\leq C \left\Vert \int_0^\infty t^{1-2s} \left( \tilde{u}_1^{f}(\cdot,t) - \tilde{u}_2^{f}(\cdot,t) \right) dt \right\Vert_{H^1(\Omega_1 \setminus \Omega)} + C \Vert \lim_{x_{n+1}\to0} x_{n+1}^{1-2s} \partial_{n+1} (\tilde{u}_1^f - \tilde{u}_2^f) \Vert_{H^{-1}(\Omega_1 \setminus \Omega)}\\
&\leq C \Bigg( \left\Vert \int_0^h t^{1-2s} \left( \tilde{u}_1^{f}(\cdot,t) - \tilde{u}_2^{f}(\cdot,t) \right) dt \right\Vert_{H^1(\Omega_1 \setminus \Omega)} + \left\Vert \int_L^\infty t^{1-2s} \left( \tilde{u}_1^{f}(\cdot,t) - \tilde{u}_2^{f}(\cdot,t) \right) dt \right\Vert_{H^1(\Omega_1 \setminus \Omega)}\\
&\hspace{15mm} + \left\Vert \int_h^L t^{1-2s} \left( \tilde{u}_1^{f}(\cdot,t) - \tilde{u}_2^{f}(\cdot,t) \right) dt \right\Vert_{H^1(\Omega_{+\delta} \setminus \Omega)} \\
&\hspace{15mm}+ \left\Vert \int_h^L t^{1-2s} \left( \tilde{u}_1^{f}(\cdot,t) - \tilde{u}_2^{f}(\cdot,t) \right) dt \right\Vert_{H^1(\Omega_1 \setminus \Omega_{+\delta})}\\
&\hspace{15mm} + \Vert \lim_{x_{n+1}\to0} x_{n+1}^{1-2s} \partial_{n+1} (\tilde{u}_1^f - \tilde{u}_2^f) \Vert_{H^{-1}(\Omega_1 \setminus \Omega)} \Bigg).
\end{aligned}
\end{equation}

In the next steps, the upper and lower integrals will be estimated by Lemma \ref{lem:apriori_fract_cond}, the intermediate integrals by using Propositions \ref{prop:estimate_thin_annuli} and \ref{prop:estimate_Omega_ext}, respectively, and the inhomogeneous term we estimate by Proposition \ref{prop:inhomogeneous_term}.

\textit{Step 2: Estimating the lower and upper integrals.} First of all, we recall that by Lemma \ref{lem:apriori_fract_cond} we have
\begin{align*}
&\left\Vert \int_0^h t^{1-2s} \left( \tilde{u}_1^{f}(\cdot,t) - \tilde{u}_2^{f}(\cdot,t) \right) dt \right\Vert_{H^1(\Omega_1\setminus\Omega)} \\
&\hspace{15mm} \leq \left\Vert \int_0^h t^{1-2s} \tilde{u}_1^{f}(\cdot,t) dt \right\Vert_{H^1(\Omega_1\setminus\Omega)} + \left\Vert \int_0^h t^{1-2s} \tilde{u}_2^{f}(\cdot,t) dt \right\Vert_{H^1(\Omega_1\setminus\Omega)}\\
&\hspace{15mm} \leq C h^{1-s} \Vert f \Vert_{\widetilde{H}^s(W)},
\end{align*}
and similarly
\begin{align*}
\left\Vert \int_L^\infty t^{1-2s} \left( \tilde{u}_1^{f}(\cdot,t) - \tilde{u}_2^{f}(\cdot,t) \right) dt \right\Vert_{H^1(\Omega_1\setminus\Omega)} \leq C L^{2-n-2s} \Vert f \Vert_{\widetilde{H}^s(W)}.
\end{align*}
Hence, if we choose
\begin{align}\label{eq:def_h_L}
h = \omega(\varepsilon)^{\frac{1}{1-s}} \qquad \text{and} \qquad L = \omega(\varepsilon)^{\frac{1}{(2-n-2s)}},
\end{align}
we can guarantee that
\begin{equation}\label{eq:proof_full_integral_lower_upper_integrals}
\begin{aligned}
&C \left( \left\Vert \int_0^h t^{1-2s} \left( \tilde{u}_1^{f}(\cdot,t) - \tilde{u}_2^{f}(\cdot,t) \right) dt \right\Vert_{H^1(\Omega_1\setminus\Omega)} + \left\Vert \int_L^\infty t^{1-2s} \left( \tilde{u}_1^{f}(\cdot,t) - \tilde{u}_2^{f}(\cdot,t) \right) dt \right\Vert_{H^1(\Omega_1\setminus\Omega)} \right)\\
&\hspace{2cm}\leq C\omega(\varepsilon) \Vert f \Vert_{\widetilde{H}^s(W)} = C \vert \log(\varepsilon) \vert^{-\beta_1} \Vert f \Vert_{\widetilde{H}^s(W)}.
\end{aligned}
\end{equation}

\textit{Step 3: Estimating the inhomogeneous term.} The inhomogeneous term was already estimated in Proposition \ref{prop:inhomogeneous_term}. There we proved that there exist constants $C_2>0$ and $\beta_{2}>0$ such that
\begin{equation}\label{eq:proof_full_integral_inhomogeneous_term}
\begin{aligned}
\Vert \lim_{x_{n+1}\to0} x_{n+1}^{1-2s} \partial_{n+1} (\tilde{u}_1^f - \tilde{u}_2^f) \Vert_{H^{-1}(\Omega_1 \setminus \Omega)} \leq C_2 \vert \log(\varepsilon) \vert^{-\beta_{2}} \Vert f \Vert_{\widetilde{H}^s(W)}.
\end{aligned}
\end{equation} 

\textit{Step 4: Estimating the intermediate integrals.} We estimate the two intermediate integrals in \eqref{eq:proof_full_integral_splitting} separately and then we choose $\delta>0$ appropriately to deduce the desired estimate.

\textit{Step 4.1: Estimate for the intermediate integral on $\Omega_{+\delta} \setminus \Omega$.} Using Hölder's inequality with respect to the $x_{n+1}$-variable and applying the result of Proposition \ref{prop:estimate_thin_annuli}, we infer
\begin{equation}\label{eq:proof_full_integral_intermediate_integral_thin}
\begin{aligned}
\Big\Vert \int_h^L t^{1-2s} \left( \tilde{u}_1^f(\cdot,t) - \tilde{u}_2^f(\cdot,t) \right) dt &\Big\Vert_{H^1(\Omega_{+\delta}\setminus\Omega)} \leq \Vert t^{\frac{1-2s}{2}} \Vert_{L^2((h,L))} \Vert \tilde{u}_1^f - \tilde{u}_2^f \Vert_{H^1( (\Omega_{+\delta} \setminus \Omega) \times (h,L), x_{n+1}^{1-2s})}\\
&\leq C \delta^{p_\delta^+} (L^{p_L} + h^{-p_h}) \Vert f \Vert_{\widetilde{H}^s(W)}
\end{aligned}
\end{equation}
for some $p_{\delta}^+, p_L, p_h  >0$.

\textit{Step 4.2: Estimating the intermediate integral on $\Omega_1 \setminus \Omega_{+\delta}$.} For the second intermediate integral we first observe that since 
\begin{align*}
-\nabla'\cdot a \nabla' \Big( \int_0^\infty t^{1-2s} (\tilde{u}_1^f - \tilde{u}_2^f) dt \Big) = c_s \lim_{x_{n+1}\to0} x_{n+1}^{1-2s} \partial_{n+1} (\tilde{u}_1^f - \tilde{u}_2^f) \mbox{ in } \Omega,
\end{align*}
it holds that
\begin{align*}
-\nabla'\cdot a \nabla' &\left( \int_h^L t^{1-2s} (\tilde{u}_1^f(\cdot,t) - \tilde{u}_2^f(\cdot,t)) dt \right)\\
&= \nabla'\cdot a \nabla' \left( \int_0^h t^{1-2s} (\tilde{u}_1^f(\cdot,t) - \tilde{u}_2^f(\cdot,t)) dt + \int_L^\infty t^{1-2s} (\tilde{u}_1^f(\cdot,t) - \tilde{u}_2^f(\cdot,t)) dt \right)\\
&\hspace{15mm} + c_s \lim_{x_{n+1}\to0} x_{n+1}^{1-2s} \partial_{n+1} (\tilde{u}_1^f - \tilde{u}_2^f).
\end{align*}
By the (standard) Caccioppoli inequality (cf. Lemma \ref{lem:Caccioppoli_inequality} with $s=\frac{1}{2}$) and using the estimates \eqref{eq:proof_full_integral_lower_upper_integrals} from Step 2 for the upper and lower integrals and \eqref{eq:proof_full_integral_inhomogeneous_term} from Step 3 for the inhomogeneous term, we infer
\begin{equation}\label{eq:proof_full_integral_intermediate_integral_1}
\begin{aligned}
\Big\Vert \int_h^L t^{1-2s} &(\tilde{u}_1^f(\cdot,t) - \tilde{u}_2^f(\cdot,t)) dt \Big\Vert_{H^1(\Omega_1 \setminus \Omega_{+\delta})}\\
&\leq C\delta^{-1} \Big\Vert \int_h^L t^{1-2s} (\tilde{u}_1^f(\cdot,t) - \tilde{u}_2^f(\cdot,t)) dt \Big\Vert_{L^2(\Omega_2 \setminus \Omega_{+\frac{\delta}{2}})} + C \omega(\varepsilon) \Vert f \Vert_{\widetilde{H}^s(W)}\\
&\hspace{15mm} + C_2 \vert \log(\varepsilon) \vert^{-\beta_2} \Vert f \Vert_{\widetilde{H}^s(W)},
\end{aligned}
\end{equation}
where we have used that $\Vert a \Vert_{L^\infty(\R^n)}$ is uniformly bounded and that $\Omega_2$, $\Omega_{+\frac{\delta}{2}}$ are open, bounded, Lipschitz with $\Omega_2 \Supset \Omega_1$ and $\Omega \Subset \Omega_{+\frac{\delta}{2}} \Subset \Omega_{+\delta}$.

It thus remains to treat the $L^2$-norm of the intermediate integral on $\Omega_2 \setminus \Omega_{+\frac{\delta}{2}}$. For this we first apply Hölder's inequality with respect to the $x_{n+1}$-variable
\begin{equation}\label{eq:proof_full_integral_intermediate_integral_2}
\begin{aligned}
\left\Vert \int_h^L t^{1-2s} \left( \tilde{u}_1^{f}(\cdot,t) - \tilde{u}_2^{f}(\cdot,t) \right) dt \right\Vert_{L^2(\Omega_2 \setminus \Omega_{+\frac{\delta}{2}})} &\leq \Vert t^{\frac{1-2s}{2}} \Vert_{L^2((h,L))} \Vert \tilde{u}_1^{f} - \tilde{u}_2^{f} \Vert_{L^2((\Omega_2 \setminus \Omega_{+\frac{\delta}{2}}) \times (h,L), x_{n+1}^{1-2s})}\\
&\leq L^{1-s} \Vert \tilde{u}_1^{f} - \tilde{u}_2^{f} \Vert_{L^2((\Omega_2 \setminus \Omega_{+\frac{\delta}{2}}) \times (h,L), x_{n+1}^{1-2s})}.
\end{aligned}
\end{equation}

Next, we will apply Proposition \ref{prop:estimate_Omega_ext} to $\tilde{w} := \tilde{u}_1^f - \tilde{u}_2^f$. For this, we first of all note that by the weighted Poincaré inequality (see Theorem 1.2 in \cite{FKS82}) we have for sufficiently large $\kappa>0$
\begin{align*}
\Vert \tilde{w} \Vert_{L^2(B_{\kappa L}' \times (0,\kappa L), x_{n+1}^{1-2s})} &\leq \Vert \tilde{u}_1^f \Vert_{L^2(B_{\kappa L}' \times (0,\kappa L), x_{n+1}^{1-2s})} + \Vert \tilde{u}_2^f \Vert_{L^2(B_{\kappa L}' \times (0,\kappa L), x_{n+1}^{1-2s})}\\
&\leq CL (\Vert \tilde{u}_1^f \Vert_{\dot{H}^1(B_{\kappa L}' \times (0,\kappa L), x_{n+1}^{1-2s})} + \Vert \tilde{u}_2^f \Vert_{\dot{H}^1(B_{\kappa L}' \times (0,\kappa L), x_{n+1}^{1-2s})})\\
&\leq C L \Vert f \Vert_{\widetilde{H}^s(W)},
\end{align*}
where $C>0$ depends on $\kappa$.
Moreover, by assumption it holds true that
\begin{align*}
\Vert \restr{\lim_{x_{n+1}\to0} x_{n+1}^{1-2s} \partial_{n+1} \tilde{w}}{W} \Vert_{H^{-s}(W)} &= \Vert \restr{\lim_{x_{n+1}\to0} x_{n+1}^{1-2s} \partial_{n+1} (\tilde{u}_1^f - \tilde{u}_2^f)}{W} \Vert_{H^{-s}(W)}\\
& = \Vert (\Lambda_s^{a_1} - \Lambda_s^{a_2}) f \Vert_{H^{-s}(W)}\\
&\leq \Vert \Lambda_s^{a_1} - \Lambda_s^{a_2} \Vert_{\widetilde{H}^s(W) \to H^{-s}(W)} \Vert f \Vert_{\widetilde{H}^s(W)} \leq \varepsilon \Vert f \Vert_{\widetilde{H}^s(W)}.
\end{align*}
Applying Proposition \ref{prop:estimate_Omega_ext} to $\tilde{w} := \tilde{u}_1^f - \tilde{u}_2^f$ with $A := \Vert f \Vert_{\widetilde{H}^s(W)}$ yields
\begin{equation}\label{eq:proof_full_integral_intermediate_integral_3}
\begin{aligned}
\Vert &\tilde{u}_1^f - \tilde{u}_2^f \Vert_{L^2((\Omega_2 \setminus \Omega_{+\frac{\delta}{2}}) \times (h,L), x_{n+1}^{1-2s})}\\
&\leq C \Vert f \Vert_{\widetilde{H}^s(W)} \left( (L^{p_L} + \delta^{-p_\delta}) \varepsilon^{c\alpha^{C(\log(L)+\log(\delta^{-1}))}} + (h^{-p_h} + \delta^{-p_\delta}) \varepsilon^{c\alpha^{C(\log(h^{-1})+\log(\delta^{-1}))}} \right).
\end{aligned}
\end{equation}
Thus, combining \eqref{eq:proof_full_integral_intermediate_integral_1}, \eqref{eq:proof_full_integral_intermediate_integral_2} and \eqref{eq:proof_full_integral_intermediate_integral_3}, we obtain 
\begin{equation}\label{eq:proof_full_integral_intermediate_integral_4}
\begin{aligned}
\Big\Vert \int_h^L t^{1-2s} &\left( \tilde{u}_1^{f}(\cdot,t) - \tilde{u}_2^{f}(\cdot,t) \right) dt \Big\Vert_{H^1(\Omega_1\setminus\Omega_{+\delta})}\\
&\leq C \Vert f \Vert_{\widetilde{H}^s(W)} \delta^{-p_\delta} (L^{p_L} + h^{-p_h}) \varepsilon^{c\alpha^{C \left(\log(L)+\log(h^{-1})+\log(\delta^{-1})\right)}} + C \omega(\varepsilon) \Vert f \Vert_{\widetilde{H}^s(W)}\\
&\hspace{20mm} + C_2 \vert \log(\varepsilon) \vert^{-\beta_2} \Vert f \Vert_{\widetilde{H}^s(W)}
\end{aligned}
\end{equation}
for some exponents $p_L, p_h, p_\delta > 0$. 

\textit{Step 4.3: Combining the estimates for the two intermediate integrals.} We recall the choice of $h$ and $L$ from \eqref{eq:def_h_L}, i.e. $h = \omega(\varepsilon)^{\frac{1}{1-s}}$ and $L = \omega(\varepsilon)^{\frac{1}{(2-n-2s)}}$. If we choose
\begin{align*}
\delta = \omega(\varepsilon)^{\frac{1}{p_\delta^+}} \left( \max\{L^{p_L}, h^{-p_h}\} \right)^{-\frac{1}{p_\delta^+}} = \min\left\{ \omega(\varepsilon)^{\frac{p_L}{(n-2+2s)p_\delta^+} + \frac{1}{p_\delta^+}}, \omega(\varepsilon)^{\frac{p_h}{(1-s)p_\delta^+} + \frac{1}{p_\delta^+}} \right\},
\end{align*}
then \eqref{eq:proof_full_integral_intermediate_integral_thin} guarantees that
\begin{align}\label{eq:proof_full_integral_intermediate_integral_thin_2}
\Big\Vert \int_h^L t^{1-2s} \left( \tilde{u}_1^f(\cdot,t) - \tilde{u}_2^f(\cdot,t) \right) dt &\Big\Vert_{H^1(\Omega_{+\delta}\setminus\Omega)} \leq C \vert \log(\varepsilon) \vert^{-\beta_1} \Vert f \Vert_{\widetilde{H}^s(W)}.
\end{align}

Plugging in the choices of $\delta$, $h$ and $L$ into \eqref{eq:proof_full_integral_intermediate_integral_4}, yields
\begin{equation}\label{eq:proof_full_integral_intermediate_integral_5}
\begin{aligned}
\Big\Vert \int_h^L t^{1-2s} \left( \tilde{u}_1^{f}(\cdot,t) - \tilde{u}_2^{f}(\cdot,t) \right) dt \Big\Vert_{H^1(\Omega_1 \setminus \Omega_{+\delta})} 
&\leq C \Vert f \Vert_{\widetilde{H}^s(W)} \omega(\varepsilon)^{-p_{\omega}} \varepsilon^{c\alpha^{C\vert\log(\omega(\varepsilon))\vert}} \\
& \qquad + C \omega(\varepsilon) \Vert f \Vert_{\widetilde{H}^s(W)}\\
&\qquad + C_2 \vert \log(\varepsilon) \vert^{-\beta_2} \Vert f \Vert_{\widetilde{H}^s(W)}
\end{aligned}
\end{equation}
for some $p_{\omega}>0$. Next, we observe that
\begin{align*}
\omega(\varepsilon)^{-p_{\omega}} \varepsilon^{c\alpha^{C\vert\log(\omega(\varepsilon))\vert}} = \omega(\varepsilon)^{-p_{\omega}} \varepsilon^{c\omega(\varepsilon)^{-C\log(\alpha)}} = \omega(\varepsilon)^{-p_{\omega}} \exp \left( c \log(\varepsilon) \omega(\varepsilon)^{C\vert\log(\alpha)\vert} \right).
\end{align*}
At this point it is now essential that the constant $C$ did not depend on $\beta$. Hence, we can now choose
\begin{align*}
0 < \beta_{1} < \frac{1}{C \vert\log(\alpha)\vert}.
\end{align*}
Plugging this choice back into $\omega(\varepsilon) = \vert \log(\varepsilon) \vert^{-\beta_{1}}$, we infer
\begin{align*}
\omega(\varepsilon)^{-p_\omega} \exp\left( c\log(\varepsilon) \omega(\varepsilon)^{C\vert\log(\alpha)\vert} \right) \leq C \vert \log(\varepsilon) \vert^{\beta_{1} p_\omega} \exp\left( - c\frac{\vert\log(\varepsilon)\vert}{\vert\log(\varepsilon)\vert^{C \beta_{1} \vert\log(\alpha)\vert}} \right),
\end{align*}
which converges quicker to $0$ than $\vert\log(\varepsilon)\vert^{-\beta_{1}}$ as $\varepsilon \to 0$ since (writing $z = \vert\log(\varepsilon)\vert$) the term $z^a \exp(-cz^b)$ converges quicker to $0$ than $z^{-d}$ for any $a,b,c,d>0$ as $z \to \infty$. Thus, from \eqref{eq:proof_full_integral_intermediate_integral_5} and the last considerations, it follows that there exists a constant $C_{1}>0$ such that
\begin{equation}\label{eq:proof_full_integral_intermediate_integral_6}
\begin{aligned}
&\left\Vert \int_h^L t^{1-2s} \left( \tilde{u}_1^{f}(\cdot,t) - \tilde{u}_2^{f}(\cdot,t) \right) dt \right\Vert_{H^1(\Omega_1\setminus\Omega_{+\delta})}\\
&\hspace{20mm} \leq C_{1} \vert \log(\varepsilon) \vert^{-\beta_{1}} \Vert f \Vert_{\widetilde{H}^s(W)} + C \omega(\varepsilon) \Vert f \Vert_{\widetilde{H}^s(W)} + C_2 \vert \log(\varepsilon) \vert^{-\beta_2} \Vert f \Vert_{\widetilde{H}^s(W)}.
\end{aligned}
\end{equation}

\textit{Step 5: Conclusion.} Combining the estimates \eqref{eq:proof_full_integral_splitting}, \eqref{eq:proof_full_integral_lower_upper_integrals}, \eqref{eq:proof_full_integral_inhomogeneous_term}, \eqref{eq:proof_full_integral_intermediate_integral_thin_2} and \eqref{eq:proof_full_integral_intermediate_integral_6} and choosing $\beta = \min\{\beta_1,\beta_2\}$ finally finishes the proof.
\end{proof}

\section{Quantitative Runge approximation}\label{sec:QRA}

In this section we will provide the proof of the second main ingredient for the quantitative reduction argument. We prove a quantitative Runge approximation result, Proposition \ref{prop:QRA}.

We consider the following Caffarelli-Silvestre extension interpretation associated with the Calde\-rón problem, i.e. for $f \in \widetilde{H}^s(W)$ we let $\tilde{u}^f \in \dot{H}^1(\R^{n+1}_+, x_{n+1}^{1-2s})$ be the solution to
\begin{equation}\label{eq:QRA_CSExtension}
\begin{cases}
\begin{aligned}
- \nabla \cdot x_{n+1}^{1-2s} \tilde{a} \nabla \tilde{u}^f &= 0 \quad &&\text{in } \R^{n+1}_+,\\
\lim_{x_{n+1}\to0} x_{n+1}^{1-2s} \partial_{n+1} \tilde{u}^f &= 0 \quad &&\text{ on } \Omega\times\{0\},\\
\tilde{u}^f &= f \quad &&\text{ on } \Omega_e\times\{0\}.
\end{aligned}
\end{cases}
\end{equation}

For $\widehat{\Omega} \subset \R^n$ open, bounded, non-empty and Lipschitz, we define the sets
\begin{align*}
S_1(\widehat{\Omega}) &:= \{ v \in H^1(\widehat{\Omega}): \ -\nabla'\cdot a \nabla' v = 0 \text{ in } \widehat{\Omega} \},\\
S_2(\widehat{\Omega}) &:= \left\{ \restr{\int_0^\infty t^{1-2s} \tilde{u}^f(\cdot,t) dt}{\widehat{\Omega}} \in H^1(\widehat{\Omega}): \ \tilde{u}^f \in \dot{H}^1(\R^{n+1}_+, x_{n+1}^{1-2s}) \text{ solves \eqref{eq:QRA_CSExtension}}, f \in \widetilde{H}^s(W) \right\}.
\end{align*}

Note that $\widehat{\Omega}$ only affects the restriction operator in the definition of $S_1(\widehat{\Omega})$ and $S_2(\widehat{\Omega})$. The sets $\Omega$ and $W$ in the defining equation \eqref{eq:QRA_CSExtension} remain as they are. The goal is to prove density of $S_2(\Omega)$ in $S_1(\Omega)$ in a quantitative way.

\begin{prop}\label{prop:QRA}
Let $\Omega \subset \R^n$ be open, non-empty, bounded and Lipschitz such that $\R^n \setminus \Omega$ is connected and let $W \subset \Omega_e$ be open, bounded, non-empty and Lipschitz. Let $\theta_1\in(0,1)$ and let $\Omega_1\Supset\Omega$ be open, bounded, Lipschitz. Let $a \in C^2(\R^n, \R^{n \times n}_{\text{sym}})$ satisfy the assumptions (A1) with the given $\theta_1$, and (A3).

There are $C,\mu>0$ such that for any $v \in S_1(\Omega_1)$ and any $\varepsilon \in (0,1)$ there exists $f \in \widetilde{H}^s(W)$ such that
\begin{align*}
\left\Vert \restr{v}{\Omega} - \restr{\int_0^\infty t^{1-2s} \tilde{u}^f(\cdot,t) dt}{\Omega} \right\Vert_{L^2(\Omega)} \leq \varepsilon \Vert v \Vert_{H^1(\Omega_1)}, \qquad \Vert f \Vert_{\widetilde{H}^s(W)} \leq C e^{C\varepsilon^{-\mu}} \Vert v \Vert_{L^2(\Omega)},
\end{align*}
where $\tilde{u}^f$ is a weak solution to \eqref{eq:QRA_CSExtension} with Dirichlet data $f$. Here, the constants $C$ and $\mu$ only depend on $n$, $s$, $\Omega$, $\Omega_1$, $W$ and $\theta_1$.
\end{prop}

The estimate from Proposition \ref{prop:QRA} is strongly reminiscent of similar bounds from control theory. More precisely, the first estimate encodes the approximability, while the second one measures the ``cost'' of approximation. Also our arguments for this estimate borrow ideas from control theory.

\subsection{The adjoint problem}\label{sec:QRA_AdjointProblem}

Before we turn to the proof of the quantitative Runge approximation of Proposition \ref{prop:QRA}, we recall and collect some preliminary results for the following auxiliary problem, which we will refer to as the adjoint problem. The original result stems from Lemma 3.2 in \cite{CGRU23}. We complement it by some a-priori estimates on the solution.

\begin{lem}\label{lem:QRA_adjoint_problem}
Let $n \geq 3$, let $\Omega \subset \R^n$ be open, non-empty, bounded and Lipschitz and $W \subset \Omega_e$ be open, non-empty, bounded and Lipschitz. Let $\theta_1 \in (0,1)$ and assume that $a \in L^\infty(\R^n, \R^{n \times n}_{\text{sym}})$ satisfies the assumptions (A1) with the given $\theta_1$, and (A3). Let $w \in \widetilde{H}^{-1}(\Omega)$.

Then there exists a weak solution $\tilde{h} \in \dot{H}^1_{\text{loc}}(\R^{n+1}_+,x_{n+1}^{1-2s})$ to the problem
\begin{equation}
\begin{cases}
\begin{alignedat}{2}
-\nabla\cdot x_{n+1}^{1-2s} \tilde{a} \nabla \tilde{h} &= x_{n+1}^{1-2s} w \chi_{\Omega\times(0,\infty)} \quad &&\text{in } \R^{n+1}_+,\\
\lim_{x_{n+1}\to0} x_{n+1}^{1-2s} \partial_{n+1} \tilde{h} &= 0 \quad &&\text{on } \Omega \times \{0\},\\
\tilde{h} &= 0 \quad &&\text{on } \Omega_e \times \{0\}.
\end{alignedat}
\end{cases}
\end{equation}
More precisely, there exists a unique function $\tilde{h} \in \dot{H}^1_{\text{loc}}(\R^{n+1}_+, x_{n+1}^{1-2s})$ with the following properties:
\begin{itemize}
\item its trace vanishes on $\Omega_e \times \{0\}$ and it satisfies
\begin{align*}
\int_{\R^{n+1}_+} x_{n+1}^{1-2s} \nabla \tilde{\phi} \cdot \tilde{a} \nabla \tilde{h} dx = \langle w, \int_0^\infty x_{n+1}^{1-2s} \tilde{\phi}(\cdot,x_{n+1}) dx_{n+1} \rangle_{\widetilde{H}^{-1}(\Omega), H^1(\Omega)}
\end{align*}
for all
\begin{align*}
\tilde{\phi} \in H^1_{c,0}(\R^{n+1}_+, x_{n+1}^{1-2s}) 
&:= \{ \tilde{v} \in H^1(\R^{n+1}_+,x_{n+1}^{1-2s}): \\
& \quad \quad \ \tilde{v} \text{ has compact support in } \overline{\R^{n+1}_+}, \ \restr{\tilde{v}}{\Omega_e \times \{0\}} = 0 \},
\end{align*}
\item $\tilde{h}$ can be decomposed into three auxiliary functions $\tilde{h}= \tilde{u}_1 - \bar{u} + \bar{u}_2$ and for $L\geq1$ the following a-priori estimates hold
\begin{align*}
\Vert \tilde{h} \Vert_{H^1(B_{\kappa L}' \times (0,\kappa L), x_{n+1}^{1-2s})} &\leq C L^p \Vert w \Vert_{H^{-1}(\Omega)},\\
\Vert \tilde{u}_1 (\cdot, 0) \Vert_{L^{q_1}(\R^n)} + \Vert \bar{u} (\cdot, 0) \Vert_{L^{q_1}(\R^n)} &\leq C \Vert w \Vert_{H^{-1}(\Omega)},\\
\Vert \bar{u}_2 (\cdot, 0) \Vert_{L^{q_2}(\R^n)} &\leq C \Vert w \Vert_{H^{-1}(\Omega)}.
\end{align*}
Here, $\kappa>0$ is a possibly large but fixed parameter, and $C>0$ and $p>0$ are suitable constants only depending on $n$, $s$, $\Omega$, $\theta_1$ and $\kappa$. The integrability exponents are given by $q_1 := \frac{2n}{n-2}$ and $q_2 := \frac{2n}{n-2s}$.
\item Lastly, it also holds true that $\restr{-\lim_{x_{n+1}\to0} x_{n+1}^{1-2s} \partial_{n+1} \tilde{h}}{W} \in H^{-s}(W)$.
\end{itemize}
\end{lem}

\begin{proof}
We argue in three steps. 

\textit{Step 1: Initial solution.} Firstly, we construct a solution $u_1$ of 
\begin{align*}
-\nabla' \cdot a \nabla' u_1 & = w \chi_{\Omega} \quad \text{in } \R^n.
\end{align*}
The existence and uniqueness of $u_1 \in \dot{H}^1(\R^n) \cap L^{\frac{2n}{n-2}}(\R^n)$ with the bound
\begin{align*}
\|\nabla u_1\|_{L^2(\R^n)} + \|u_1\|_{L^{\frac{2n}{n-2}}(\R^n)} \leq C \|w\|_{H^{-1}(\Omega)}
\end{align*}
 follow from the theorem of Lax-Milgram in the space $\dot{H}^1(\R^n)\cap L^{\frac{2n}{n-2}}(\R^n)$ equipped with the $\dot{H}^1(\R^n)$ inner product. With this auxiliary function in hand, we define $\tilde{u}_1(x, x_{n+1}):= u_1(x)$. By the above considerations, we infer that $\tilde{u}_1 \in L^{\infty}(\R_+, \dot{H}^1(\R^n)) \cap L^{\infty}(\R_+, L^{\frac{2n}{n-2}}(\R^n))$. By construction, it is a weak solution (in the sense of testing with compactly  supported test functions) of the equation
\begin{equation*}
\begin{cases}
\begin{alignedat}{2}
-\nabla \cdot x_{n+1}^{1-2s} \tilde{a} \nabla \tilde{u}_1 &= x_{n+1}^{1-2s} w \chi_{\Omega \times (0,\infty)} \quad &&\text{in } \R^{n+1}_+,\\
\lim\limits_{x_{n+1} \rightarrow 0} x_{n+1}^{1-2s} \p_{n+1} \tilde{u}_1 &= 0 \quad &&\text{on } \Omega \times \{0\},\\
\tilde{u}_1 &= u_1 \quad &&\text{on } \Omega_e \times \{0\}.
\end{alignedat}
\end{cases}
\end{equation*}
It satisfies the bound 
\begin{align*}
\|\tilde{u}_1\|_{L^{\infty}(\R_+, \dot{H}^1(\R^n)) } + \|\tilde{u}_1\|_{L^{\infty}(\R_+, L^{\frac{2n}{n-2}}(\R^n)) } \leq C \|w\|_{H^{-1}(\Omega)}.
\end{align*}

\textit{Step 2: Correcting the initial solution.} Next we ``correct'' the function $\tilde{u}_1$ in order to turn it into a full solution of our original problem. To this end, we solve the problem
\begin{equation}\label{eq:AdjointProblem_CorrectingFunction}
\begin{cases}
\begin{alignedat}{2}
-\nabla \cdot x_{n+1}^{1-2s} \tilde{a} \nabla \tilde{u}_2 &= 0 \quad &&\text{in } \R^{n+1}_+,\\
\lim\limits_{x_{n+1} \rightarrow 0} x_{n+1}^{1-2s} \p_{n+1} \tilde{u}_2 &= 0 \quad &&\text{on } \Omega \times \{0\},\\
\tilde{u}_2 &= u_1 \quad &&\text{on } \Omega_e \times \{0\}.
\end{alignedat}
\end{cases}
\end{equation}

After an extension, this follows from the theorem of Lax-Milgram. Indeed, let $\bar{u}$ be the constant coefficient Caffarelli-Silvestre extension of $u_1$ (see \eqref{eq:CS_intro} with $a=\Id$). We consider a function $\bar{u}_2$ solving
\begin{equation}\label{eq:AdjointProblem_AuxiliaryFunction2}
\begin{cases}
\begin{alignedat}{2}
-\nabla \cdot x_{n+1}^{1-2s} \tilde{a} \nabla \bar{u}_2 &= -\nabla \cdot x_{n+1}^{1-2s} \tilde{a} \nabla \bar{u}  \quad &&\text{in } \R^{n+1}_+,\\
\lim\limits_{x_{n+1} \rightarrow 0} x_{n+1}^{1-2s} \p_{n+1} \bar{u}_2 &= 
\lim\limits_{x_{n+1} \rightarrow 0} x_{n+1}^{1-2s} \p_{n+1} \bar{u} \quad &&\text{on } \Omega \times \{0\},\\
\bar{u}_2 &= 0 \quad &&\text{on } \Omega_e \times \{0\}.
\end{alignedat}
\end{cases}
\end{equation}
Then, $\tilde{u}_2 := \bar{u} - \bar{u}_2$ is a weak solution to \eqref{eq:AdjointProblem_CorrectingFunction}. The weak form of \eqref{eq:AdjointProblem_AuxiliaryFunction2} turns into
\begin{align*}
&\int_{\R^{n+1}_+} x_{n+1}^{1-2s} \tilde{ a}\nabla \bar{u}_2 \cdot \nabla \tilde{\varphi} dx =
\int_{\R^{n+1}_+} x_{n+1}^{1-2s} \tilde{a}\nabla \bar{u} \cdot \nabla \tilde{\varphi} dx \\
&\quad = \int_{\R^{n+1}_+} x_{n+1}^{1-2s} (\tilde{a} - \Id_{(n+1)\times (n+1)}) \nabla \bar{u} \cdot \nabla \tilde{\varphi} dx  +  \int_{\R^{n}} \tilde{\varphi}(\cdot,0) \lim_{x_{n+1} \rightarrow 0} x_{n+1}^{1-2s} \p_{n+1} \bar{u} dx'  \\
&\quad = \int_{B_R' \times \R_+} x_{n+1}^{1-2s} (\tilde{a} - \Id_{(n+1)\times (n+1)}) \nabla \bar{u} \cdot \nabla \tilde{\varphi} dx +  \int_{B_R'} \tilde{\varphi}(\cdot,0) \lim_{x_{n+1} \rightarrow 0} x_{n+1}^{1-2s} \p_{n+1} \bar{u}  dx'  ,
\end{align*}
where the test function $\tilde{\varphi}$ satisfies $\tilde{\varphi} \in \dot{H}^1(\R^{n+1}_+, x_{n+1}^{1-2s})$ with vanishing trace on $\Omega_e \times \{0\}$ and where $R \geq 1$ is such that $\Omega \subset B_R'$. In the above identity, we used that $\bar{u}$ is a solution of the constant coefficient Caffarelli-Silvestre extension and that $\tilde{a} = \Id_{(n+1)\times(n+1)}$ in $\Omega_e \times \R_+$. Now it remains to show that 
\begin{align}\label{eq:bounded_functional}
\tilde{\varphi} \mapsto \int_{B_R' \times \R_+} x_{n+1}^{1-2s} (\tilde{a} - \Id_{(n+1)\times (n+1)}) \nabla \bar{u} \cdot \nabla \tilde{\varphi} dx  + \int_{B_R'} \tilde{\varphi}(\cdot,0) \lim_{x_{n+1} \rightarrow 0} x_{n+1}^{1-2s} \p_{n+1} \bar{u}  dx' 
\end{align}
is a bounded functional on $\dot{H}^1(\R^{n+1}_+, x_{n+1}^{1-2s})$ with vanishing trace on $\Omega_e \times \{0\}$.

\textit{Step 2.1: Boundary term.} We first discuss the boundary contribution in \eqref{eq:bounded_functional}. We note that by duality and the regularity estimates for the Caffarelli-Silvestre extension, we have that
\begin{align*}
\left|\int_{B_R'} \tilde{\varphi}(\cdot,0) \lim_{x_{n+1} \rightarrow 0} x_{n+1}^{1-2s} \p_{n+1} \bar{u}  dx'  \right|
&\leq C \|\lim\limits_{x_{n+1} \rightarrow 0} x_{n+1}^{1-2s} \p_{n+1} \bar{u} \|_{\dot{H}^{1-2s}(\R^n)} \|\tilde{\varphi}(\cdot,0)\|_{\dot{H}^{2s-1}(\R^n)}\\
&\leq  C \| \bar{u}(\cdot,0) \|_{\dot{H}^{1}(\R^n)} \|\tilde{\varphi}(\cdot,0)\|_{\dot{H}^{2s-1}(\R^n)}.
\end{align*}
Here we used that $-c_s \lim\limits_{x_{n+1} \rightarrow 0} x_{n+1}^{1-2s} \p_{n+1} \bar{u} = (-\D')^s \bar{u}$ and that $(-\Delta')^s: \dot{H}^r(\R^n) \to \dot{H}^{r-2s}(\R^n)$ is bounded for any $r\in\R$. It hence remains to bound $\|\tilde{\varphi}(\cdot,0)\|_{\dot{H}^{2s-1}(\R^n)}$ in terms of $\| \tilde{\varphi}(\cdot,0) \|_{H^s(\R^n)}$ as by Poincar\'e's inequality, the compact support of $\tilde{\varphi}(\cdot,0)$ and the Sobolev-trace inequality it holds that 
\begin{align*}
\| \tilde{\varphi}(\cdot,0) \|_{H^s(\R^n)} \leq C_{\Omega} \| \tilde{\varphi}(\cdot,0) \|_{\dot{H}^s(\R^n)} \leq C_{\Omega} \|x_{n+1}^{\frac{1-2s}{2}} \nabla \tilde{\varphi}\|_{L^2(\R^{n+1}_+)}.
\end{align*}
To this end, we note that 
\begin{align*}
\| \tilde{\varphi}(\cdot,0) \|_{\dot{H}^{2s-1}(\R^n)}^2 &= \int_{B_1'}|\xi'|^{4s-2} |\F_{x'} \tilde{\varphi}(\xi',0)|^2 d\xi' + \int_{\R^n \setminus B_1'} |\xi'|^{4s-2} |\F_{x'} \tilde{\varphi}(\xi',0)|^2 d\xi'\\
&\leq \Vert \vert \cdot \vert^{4s-2} \Vert_{L^1(B_1')}^2 \Vert \mathcal{F}_{x'} \tilde{\varphi}(\cdot,0) \Vert_{L^\infty(B_1')}^2 + \int_{\R^n \setminus B_1'} |\xi'|^{2s} |\F_{x'} \tilde{\varphi}(\xi',0)|^2 d\xi'\\
&\leq C \|\tilde{\varphi}(\cdot,0)\|_{L^1(\R^n)}^2 + \|\tilde{\varphi}(\cdot,0) \|_{\dot{H}^s(\R^n)}^2 \leq C_{\Omega} \| \tilde{\varphi}(\cdot,0) \|_{L^2(\R^n)}^2 + \| \tilde{\varphi}(\cdot,0) \|_{\dot{H}^s(\R^n)}^2\\
&\leq C_{\Omega}\| \tilde{\varphi}(\cdot,0) \|_{H^s(\R^n)}^2 .
\end{align*}
Here, we denoted by $\mathcal{F}_{x'}$ the Fourier transform in the tangential variable. In the third line, we used that $\xi' \mapsto |\xi'|^{4s-2}$ is in $L^1(B_1)$ for $n\geq 3$ and in the second-to-last estimate we have invoked the compact support assumption for $\tilde{\varphi}(\cdot,0)$ on $\R^n$. Combined with the above considerations, this concludes the estimate for the boundary contribution.

\textit{Step 2.2: Bulk term.} For the bulk term in \eqref{eq:bounded_functional}, we have that 
\begin{align*}
\| x_{n+1}^{\frac{1-2s}{2}} \nabla \bar{u} \|_{L^2(B_R' \times \R_+)}
\leq \| \eta x_{n+1}^{\frac{1-2s}{2}} \nabla \bar{u} \|_{L^2( \R^{n+1}_+)} 
&\leq \| (\F_{x'}\eta) \ast_{\xi'}  x_{n+1}^{\frac{1-2s}{2}} \F_{x'}(\nabla \bar{u}) \|_{L^2( \R^{n+1}_+)},
\end{align*}
where $\eta$ is a smooth cut-off function in the tangential variables which is equal to one in $B_R'$ and vanishes outside of $B_{2R}'$, $\F_{x'}$ denotes the tangential Fourier transform and $\ast_{\xi'}$ denotes the convolution in the resulting tangential Fourier variable. Now, $\F_{x'}\bar{u}$ is expicitly given in terms of modified Bessel functions of the second kind $K_s$ (see \cite{Olver10} for properties of modified Bessel functions)
\begin{align*}
\F_{x'}\bar{u}(\xi', x_{n+1}) = c_{s} (|\xi'| x_{n+1})^{s} K_s( |\xi'| x_{n+1}) \F_{x'} u_1(\xi').
\end{align*}
Hence, as $\eta$ is independent of the $x_{n+1}$ variable, by carrying out a change of variables in the normal variable, subsequently pulling out the factors which only depend on the normal variable and then performing the normal integration, we obtain for the tangential gradient that
 \begin{align*}
\| (\F_{x'}\eta) &\ast_{\xi'}  x_{n+1}^{\frac{1-2s}{2}} \F_{x'}(\nabla' \bar{u})\|_{L^2(\R_+)}^2
= \int_{\R_+}  x_{n+1}^{1-2s}\left|  \int_{\R^n}((\F_{x'} \eta)(x'-\xi'))\xi' \F_{x'} \bar{u}(\xi') d\xi' \right|^2 dx_{n+1}\\
& = c_s \int_{\R_+}  x_{n+1}^{1-2s} \left|   \int_{\R^n}((\F_{x'} \eta)(x'-\xi'))\xi' (|\xi'| x_{n+1})^{s} K_s( |\xi'| x_{n+1}) \F_{x'} u_1(\xi') d\xi' \right|^2 dx_{n+1}\\
& = c_s \int_{\R_+}  t^{1-2s} \left|   \int_{\R^n}((\F_{x'} \eta)(x'-\xi')) t^{s} K_s(t) \frac{\xi'}{|\xi'|} |\xi'|^{s}\F_{x'} u_1(\xi') d\xi' \right|^2 d t\\ 
& = c_s \int_{\R_+}  t |K_s(t)|^2 dt \left| \int_{\R^n} ((\F_{x'} \eta)(x'-\xi')) \xi' |\xi'|^{s-1}\F_{x'} u_1(\xi') d\xi' \right|^2 \\
& = \tilde{c}_s \left| \int_{\R^n}((\F_{x'} \eta)(x'-\xi')) |\xi'|^{s-1} \F_{x'} (\nabla' u_1) (\xi') d\xi' \right|^2 .
\end{align*}
The argument for the normal gradient is similar. Indeed, using that $K_s'(z) = K_{1-s}(z) - \frac{s}{z} K_s(z)$ (see \cite[Formula 10.29.2]{Olver10}) we derive similarly as before
\begin{align*}
\| (\F_{x'}\eta) \ast_{\xi'} & x_{n+1}^{\frac{1-2s}{2}} \F_{x'}(\partial_{n+1} \bar{u}) \|_{L^2(\R_+)}^2\\
&= c_s \int_{\R_+} t \vert K_{1-s}(t) \vert^2 dt \left\vert \int_{\R^n} (\F_{x'} \eta)(x'-\xi') |\xi'|^{s} \F_{x'} u_1(\xi') d\xi' \right\vert^2\\
&= \tilde{c}_s \left\vert \int_{\R^n} (\F_{x'} \eta)(x'-\xi') |\xi'|^{s} \F_{x'} u_1(\xi') d\xi' \right\vert^2.
\end{align*}

As a consequence,
\begin{equation}\label{eq:integral}
\begin{aligned}
\| (\F_{x'}&\eta) \ast_{\xi'}  x_{n+1}^{\frac{1-2s}{2}} \F_{x'}(\nabla \bar{u}) \|_{L^2( \R^{n+1}_+)}\\
& \leq C_s \big( \| (\F_{x'}\eta) \ast_{\xi'} ( |\cdot'|^{s-1} \F_{x'}(\nabla' u_1)) \|_{L^2( \R^{n})} +  \| (\F_{x'}\eta) \ast_{\xi'} ( |\cdot'|^{s} \F_{x'}(u_1)) \|_{L^2( \R^{n})} \big)\\
& = C_s \big( \| \eta  (-\Delta')^{s/2} u_1 \|_{L^2( \R^{n})} + \| \eta  (-\Delta')^{(s-1)/2} \nabla' u_1 \|_{L^2( \R^{n})} \big)\\
&\leq C_s \big( \|   (-\Delta')^{s/2} u_1 \|_{L^2( B_{2R}')} + \| (-\Delta')^{(s-1)/2} \nabla' u_1 \|_{L^2( B_{2R}')} \big).
\end{aligned}
\end{equation}
It thus remains to estimate $\| (-\Delta')^{s/2} u_1 \|_{L^2( B_{2R}')}$ and $\| (-\Delta')^{(s-1)/2} \nabla' u_1 \|_{L^2(B_{2R}')}$. To this end, we use the compactness of $B_{2R}'$ and the regularity properties of $u_1$. More precisely, we note that
\begin{align*}
(-\Delta')^{s/2} u_1(x') = C_s PV \int_{\R^{n}}\frac{u_1(x')-u_1(y')}{|x'-y'|^{n+s}} dy' ,
\end{align*}
and, with $x' \in B_{2R}'$,
\begin{align}
\label{eq:split_int}
\int_{\R^{n}}\frac{u_1(x')-u_1(y')}{|x'-y'|^{n+s}} dy' 
&= \int_{B_{4R}'}\frac{u_1(x')-u_1(y')}{|x'-y'|^{n+s}} dy' + \int_{\R^n\setminus B_{4R}'}\frac{u_1(x')-u_1(y')}{|x'-y'|^{n+s}} dy'.
\end{align}
We consider the two contributions separately. Using Hölder's inequality, the equivalence of the $H^s(\Omega)$- and $W^{s,2}(\Omega)$-norms (see \cite[Theorem 3.30]{McLean}) and interpolation (see \cite{BL76} and \cite[Corollary 4.7]{CWHM15}) for $\epsilon \in (0,1-s)$ we obtain (for $R\geq 1$)
\begin{align*}
\int_{B_{2R}'} &\left\vert \int_{B_{4R}'} \frac{u_1(x') - u_1(y')}{\vert x'-y' \vert^{n+s}} dy' \right\vert^2 dx'\\
&\leq \int_{B_{4R}'} \left( \int_{B_{4R}'} \frac{\vert u_1(x') - u_1(y') \vert^2}{\vert x'-y' \vert^{n+2s+2\epsilon}} dy' \right) \left( \int_{B_{4R}'} \frac{1}{\vert x'-y' \vert^{n-2\epsilon}} dy' \right) dx'\\
&\leq C R^{2\epsilon} \Vert u_1 \Vert_{W^{s+\epsilon,2}(B_{4R}')}^2  \leq C R^{2s+4\epsilon} \Vert u_1 \Vert_{H^{s+\varepsilon}(B_{4R}')}^2 \leq C R^{4s+6\epsilon} \Vert u_1 \Vert_{H^1(B_{4R}')}^{2s+2\epsilon} \Vert u_1 \Vert_{L^2(B_{4R}')}^{2-2s-2\epsilon}\\
&\leq C R^{n+2+4s+6\epsilon} (\Vert u_1 \Vert_{\dot{H}^1(\R^n)} + \Vert u_1 \Vert_{L^{\frac{2n}{n-2}}(\R^n)}).
\end{align*}
Here, we used that $\int_{B_{4R}'} \frac{1}{\vert x'-y' \vert^{n-2\epsilon}} dy'$ is finite and bounded by $CR^{2\epsilon}$ and for the last inequality, we applied Hölder's inequality in order to bound the $L^2(B_{4R}')$- in terms of the $L^{\frac{2n}{n-2}}(B_{4R}')$-norm. Moreover, for the third inequality we used that by the equivalence of norms and a simple scaling argument $\Vert u_1 \Vert_{W^{s+\epsilon,2}(B_{4R}')} \leq C R^{(s+\epsilon)} \Vert u_1 \Vert_{H^{s+\epsilon}(B_{4R}')}$ and for the fourth inequality we used $\Vert u_1 \Vert_{H^{s+\epsilon}(B_{4R}')} \leq C R^{s+\epsilon} \Vert u_1 \Vert_{H^1(B_{4R}')}^{s+\epsilon} \Vert u_1 \Vert_{L^2(B_{4R}')}^{1-s-\epsilon}$, which follows from interpolation and scaling (recall that $R \geq 1$). For the second contribution in \eqref{eq:split_int}, we in turn obtain
\begin{align*}
\int_{B_{2R}'} &\left| \int_{\R^n\setminus B_{4R}'}\frac{u_1(x')-u_1(y')}{|x'-y'|^{n+s}} dy' \right|^2 dx'\\
&\leq \int_{B_{2R}'} \left\vert \int_{\R^n\setminus B_{4R}'}\frac{|u_1(y')|}{|x'-y'|^{n+s}} dy' \right\vert^2 dx'  + \int_{B_{2R}'} |u_1(x')|^2  \left\vert \int_{\R^n\setminus B_{4R}'}\frac{1}{|x'-y'|^{n+s}} dy' \right\vert^2 dx'\\
&\leq C R^{2(1-s)} \Vert u_1 \Vert_{L^{\frac{2n}{n-2}}(\R^n)}^2 + C R^{2s} \Vert u_1 \Vert_{L^2(B_{2R}')}^2 \leq C \max\{R^{2(1-s)},R^{n+2+2s}\} \Vert u_1 \Vert_{L^{\frac{2n}{n-2}}(\R^n)}^2,
\end{align*}
where we used Hölder's inequality with the exponents $p=\frac{2n}{n-2}$, $p'=\frac{2n}{n+2}$ and the integrability of the resulting integral kernels which do not carry any singularity as $x' \in B_{2R}'$.

The argument for the second contribution in \eqref{eq:integral} is similar. For this, note that $\frac{s-1}{2}<0$ and that the fractional Laplacian of negative order is given via the Riesz potential. We do a similar splitting as above and first observe
\begin{align*}
\int_{B_{2R}'}& \left\vert \int_{B_{4R}'} \frac{\nabla' u_1(y')}{\vert x'-y' \vert^{n+s-1}} dy' \right\vert^2 dx'\\
&\quad \leq \int_{B_{2R}'} \left( \int_{B_{4R}'} \frac{\vert \nabla' u_1(y') \vert^2}{\vert x'-y' \vert^{n+s-1}} dy' \right) \left( \int_{B_{4R}'} \frac{1}{\vert x'-y' \vert^{n+s-1}} dy' \right) dx'\\
&\quad \leq C R^{1-s} \int_{B_{4R}'} \vert \nabla' u_1(y') \vert^2 \left( \int_{B_{2R}'} \frac{1}{\vert x'-y' \vert^{n+s-1}} dx' \right) dy'\\
&\quad \leq C R^{2(1-s)} \Vert \nabla' u_1 \Vert_{L^2(B_{4R}')}^2,
\end{align*}
where we used Hölder's inequality, Fubini and the integrability of $\vert x'-y' \vert^{-n-s+1}$ over $B_R'$. For the second term, we infer by Hölder's inequality and the integrability of $\vert x'-y' \vert^{-2(n+s-1)}$ over $\R^n \setminus B_{4R}'$ for $x' \in B_{2R}'$
\begin{align*}
\int_{B_{2R}'} &\left\vert \int_{\R^n \setminus B_{4R}'} \frac{\nabla' u_1(y')}{\vert x'-y' \vert^{n+s-1}} dy' \right\vert^2 dx'\\
&\quad \leq \int_{B_{2R}'} \left( \int_{\R^n \setminus B_{4R}'} \vert \nabla' u_1(y') \vert^2 dy' \right) \left( \int_{\R^n \setminus B_{4R}'} \frac{1}{\vert x'-y' \vert^{2(n+s-1)}} dy' \right) dx'\\
&\quad \leq C R^{2(1-s)} \Vert \nabla' u_1 \Vert_{L^2(\R^n)}^2.
\end{align*}

Thus, collecting the above arguments we arrive at 
\begin{equation}\label{eq:AdjointProblem_BoundConstCSExt}
\begin{aligned}
\| x_{n+1}^{\frac{1-2s}{2}} \nabla \bar{u} \|_{L^2(B_R' \times \R_+)} \leq C R^p (\|u_1\|_{L^{\frac{2n}{n-2}}(\R^n)} + \|\nabla' u_1\|_{L^{2}(\R^n)}) \leq C R^p \|w\|_{H^{-1}(\Omega)}.
\end{aligned}
\end{equation}
for some $p>0$.

\textit{Step 2.3: Conclusion of Step 2.} As a consequence of Step 2.1 and Step 2.2, by the Theorem of Lax-Milgram, the equation for $\bar{u}_2$ is solvable with (here, $R \sim O(1)$, since we only need that $\Omega \subset B_R'$)
\begin{align}\label{eq:AdjointProblem_BoundAuxilliaryFunction2}
\|\bar{u}_2\|_{\dot{H}^1(\R^{n+1}_+, x_{n+1}^{1-2s})} \leq C \|w\|_{H^{-1}(\Omega)} 
\end{align}
Hence, since $\tilde{u}_2 = \bar{u} - \bar{u}_2$ the estimates for $\bar{u}$ and $\bar{u}_2$ from \eqref{eq:AdjointProblem_BoundConstCSExt} and \eqref{eq:AdjointProblem_BoundAuxilliaryFunction2} imply that
\begin{align*}
\|\tilde{u}_2\|_{\dot{H}^1(B_{\kappa L }\times (0, \kappa L), x_{n+1}^{1-2s})}
&\leq \|\bar{u}\|_{\dot{H}^1(B_{\kappa L} \times (0, \kappa L), x_{n+1}^{1-2s})} + \|\bar{u}_2\|_{\dot{H}^1(\R^{n+1}_+, x_{n+1}^{1-2s})} \\
& \leq CL^p \Vert w \Vert_{H^{-1}(\Omega)}  + C \|w\|_{H^{-1}(\Omega)}.
\end{align*}
By linearity, the function $\tilde{h}:= \tilde{u}_1  -\tilde{u}_2= \tilde{u}_1 - (\bar{u}-\bar{u}_2)$ then satisfies the desired weak equation.

\textit{Step 3: Stability estimates.} Finally, it remains to deduce the desired stability estimates. For the first estimate, we observe that
\begin{align*}
\| x_{n+1}^{\frac{1-2s}{2}} \nabla \tilde{h}\|_{L^2(B_{\kappa L}' \times [0,\kappa L])} 
&\leq \|x_{n+1}^{\frac{1-2s}{2}} \nabla \tilde{u}_1\|_{L^2(B_{\kappa L}' \times [0,\kappa L])}  + \|x_{n+1}^{\frac{1-2s}{2}} \nabla \tilde{u}_2\|_{L^2(B_{\kappa L}' \times [0,\kappa L])} \\
&\leq  C L^{1-s} \|w\|_{H^{-1}(\Omega)} + C L^p \|w\|_{H^{-1}(\Omega)} \leq C L^p \|w\|_{H^{-1}(\Omega)}
\end{align*}
for some $p>0$. This estimate together with the Poincaré inequality (note that $\tilde{h}(\cdot,0) = 0$ in $\Omega_e$) yields the first stability estimate. The second and third estimates from Lemma \ref{lem:QRA_adjoint_problem} follow from the triangle inequality and Sobolev estimates as well as the fact that $\bar{u}(x',0) = \tilde{u}_1(x',0) = u_1(x')$:
\begin{align*}
\|\tilde{u}_1(\cdot,0)\|_{L^{\frac{2n}{n-2}}(\R^n)} + \|\bar{u}(\cdot,0)\|_{L^{\frac{2n}{n-2}}(\R^n)} &\leq C \|w\|_{H^{-1}(\Omega)},\\
\|\bar{u}_2(\cdot,0)\|_{L^{\frac{2n}{n-2s}}(\R^n)} &\leq C \|\bar{u}_2\|_{\dot{H}^1(\R^{n+1}_+, x_{n+1}^{1-2s})} \leq  C \|w\|_{H^{-1}(\Omega)}.
\end{align*}
Finally, we discuss the regularity of the generalized, weighted normal derivative. By definition of the weak normal derivative, for $f \in \widetilde{H}^s(W)$, we have
\begin{align*}
|\langle \lim\limits_{x_{n+1} \rightarrow 0} x_{n+1}^{1-2s} \p_{n+1} \tilde{h}, f \rangle_{H^{-s}(W), \widetilde{H}^s(W)}|
= \left| \int_{\R^{n+1}_+} x_{n+1}^{1-2s} \nabla \tilde{h} \cdot \nabla e^f dx \right|.
\end{align*}
Here $e^f: \R^{n+1}_+ \rightarrow \R$ is an $\dot{H}^{1}(\R^{n+1}_+, x_{n+1}^{1-2s})$-extension of $f$ and we recall that the definition of the Neumann derivative is independent of the precise choice of the extension. In our application above we choose $e^f(x,x_{n+1}):= \tilde{e}^f(x,x_{n+1}) \eta(x,x_{n+1})$, where $\tilde{e}^f$ is the (constant coefficient) Caffarelli-Silvestre extension of the datum $f \in \widetilde{H}^s(W)$ and $\eta: \R^{n+1}_+ \rightarrow [0,1)$ is a smooth cut-off function restricting the support in the vertical direction to be contained in the interval $x_{n+1} \in [0,1)$ and to the ball $B_R' \Supset W$ in the tangential direction. In particular, for $e^f$ we then have the bound
\begin{align*}
\|x_{n+1}^{\frac{1-2s}{2}} \nabla e^f\|_{L^2(\R^{n+1}_+)} 
&\leq C (\|x_{n+1}^{\frac{1-2s}{2}} \nabla \tilde{e}^f\|_{L^2(\R^{n+1}_+)} + \|x_{n+1}^{\frac{1-2s}{2}} \tilde{e}^f \nabla \eta\|_{L^2(\R^{n+1}_+)})\\
&\leq  C(\|f\|_{\widetilde{H}^s(W)} + \|x_{n+1}^{\frac{1-2s}{2}} \tilde{e}^f \|_{L^2(\R^n \times (0,1))})\\
&\leq  C(\|f\|_{\widetilde{H}^s(W)} + \|x_{n+1}^{\frac{1-2s}{2}} \nabla \tilde{e}^f \|_{L^2(\R^{n+1}_+)})\\
& \leq C \|f\|_{\widetilde{H}^s(W)}.
\end{align*}
Here in the second-to-last line, we used Poincar\'e's inequality and the fact that $f\in \widetilde{H}^s(W)$.
Returning to the definition of the weak generalized normal derivative, this then yields
\begin{align*}
|\langle \lim\limits_{x_{n+1} \rightarrow 0} x_{n+1}^{1-2s} \p_{n+1} \tilde{h}, f \rangle_{H^{-s}(W), \widetilde{H}^s(W)}|
&\leq C \|x_{n+1}^{\frac{1-2s}{2}} \nabla \tilde{h}\|_{L^2(B_R' \times (0,1))} \|f\|_{\widetilde{H}^s(W)}\\
&\leq C_{s,R} \|w\|_{H^{-1}(\Omega)}\|f\|_{\widetilde{H}^s(W)}.
\end{align*}
Taking the supremum over $f \in \widetilde{H}^s(W)$ with $\|f\|_{\widetilde{H}^s(W)}=1$, then concludes the argument.
\end{proof}

We will prove the following auxiliary result corresponding to Lemma 4.1 in \cite{RS18a} or to Lemma 4.3 in \cite{BR25}. The proof is very similar to the ones in the references.

\begin{lem}\label{lem:QRA_auxiliary}
Let $\Omega$, $W$ and the metric $a$ be as in Proposition \ref{prop:QRA}. Denote by $X \subset L^2(\Omega)$ the closure of $S_1(\Omega)$ in $L^2(\Omega)$ and define the operator $T: \widetilde{H}^s(W) \to X$, $f \mapsto \restr{\int_0^\infty t^{1-2s} \tilde{u}^f(\cdot,t) dt}{\Omega}$, where $\tilde{u}^f \in \dot{H}^1(\R^{n+1}_+, x_{n+1}^{1-2s})$ is a weak solution to \eqref{eq:QRA_CSExtension}. Then, the operator $T$ is compact and has dense range.

The Banach space adjoint operator $T': X^* \to H^{-s}(W)$ is defined by 
\begin{align*}
T'w = \restr{-\lim_{x_{n+1}\to0} x_{n+1}^{1-2s} \partial_{n+1} \tilde{h}}{W},
\end{align*}
 where $\tilde{h} \in \dot{H}^1_{\text{loc}}(\R^{n+1}_+, x_{n+1}^{1-2s})$ and $w \in X'$ are related through
\begin{equation}
\begin{cases}
\begin{aligned}
- \nabla \cdot x_{n+1}^{1-2s} \tilde{a} \nabla \tilde{h} &= x_{n+1}^{1-2s} w \chi_{\Omega \times (0,\infty)} \quad &&\text{in } \R^{n+1}_+,\\
\lim_{x_{n+1}\to0} x_{n+1}^{1-2s} \partial_{n+1} \tilde{h} &= 0 \quad &&\text{ on } \Omega\times\{0\},\\
\tilde{h} &= 0 \quad &&\text{ on } \Omega_e\times\{0\}.
\end{aligned}
\end{cases}
\end{equation}

The Hilbert space adjoint operator is then given as $T^* = R_{H^{-s}(W)} T' R_X: X \to \widetilde{H}^s(W)$, where $R_H$ are the Riesz isomorphisms between a Hilbert space $H$ and its dual space. It holds that $\Vert T^*w \Vert_{\widetilde{H}^s(W)} = \Vert T'w \Vert_{H^{-s}(W)}$.

Moreover, there exist an orthonormal system $(\varphi_j)_{j\in\N}$ of $\widetilde{H}^s(W)$ (such that, in particular, $\widetilde{H}^s(W) = \operatorname{span}(\varphi_j)_{j\in\N} \bigoplus \operatorname{Null}(T)$) and an orthonormal basis $(\psi_j)_{j\in\N}$ of $X$ such that $T\varphi_j = \sigma_j \psi_j$, where $\sigma_j>0$ are the positive singular values associated with the operator $T$.
\end{lem}

\begin{proof}
\textit{Step 1: Compactness and density.} We first prove compactness of the operator $T$. Let $(f_j)_{j\in\N}$ be a bounded sequence in $\widetilde{H}^s(W)$. By Proposition 1.2 and Theorem 3 in \cite{CGRU23} (see also Lemma \ref{lem:integral_uniform_estimate} below), $\Vert Tf_j \Vert_{H^1(\Omega)}$ is bounded, and for each $j \in \N$, $Tf_j$ satisfies 
\begin{align*}
-\nabla'\cdot a \nabla'(Tf_j) =0 \mbox{ in } \Omega.
\end{align*}
 Thus, there exists an $H^1(\Omega)$-weakly converging subsequence of $(Tf_j)$ to some $\overline{v} \in H^1(\Omega)$. Rellich's theorem implies that this subsequence converges strongly in $L^2(\Omega)$. By weak convergence, $\overline{v}$ also solves $-\nabla'\cdot a \nabla' \overline{v} = 0$ in the weak sense, proving that $\overline{v} \in X$ and that $T$ is indeed compact. It has dense range by the qualitative Runge approximation result from Proposition 3.1 in \cite{CGRU23}.

\textit{Step 2: Adjoint operators.} Next, we verify that the operator $T'$ as defined above is indeed the Banach space adjoint operator. Let $\tilde{u}^f_k \in H^1_c(\R^{n+1}_+,x_{n+1}^{1-2s})$ (recall the definition of the $H_c^1(\R^{n+1}_+, x_{n+1}^{1-2s})$-spaces from Section \ref{sec:prel_FunctionSpaces}) be defined as in the proof of Proposition 3.1 in \cite{CGRU23}, i.e. $\tilde{u}^f_k(x) := \tilde{u}^f(x) \sigma_k(x') \eta_k(x_{n+1})$ for some cut-off functions $\eta_k$ and $\sigma_k$. More precisely, $\eta_k(t) := \eta_1(\frac{t}{k})$, where $\eta_1 \in C_c^\infty([0,2])$ is a smooth cut-off function satisfying $\eta_1 \equiv 1$ in a neighbourhood of $t=0$ and $\int_0^\infty t^{1-2s} \eta_1 dt =1$, and $\sigma_k(x') := \sigma_1(\frac{x'}{k})$, where $\sigma_1 \in C_c^\infty(B_{2R}')$ is a smooth and radially symmetric cut-off function satisfying $\sigma_1 \equiv 1$ in $B_R'$ for $R>0$ large enough such that $\overline{\Omega} \cup \overline{W} \subset B_R'$. Then it holds that
\begin{align*}
\langle f,T'w \rangle &_{\widetilde{H}^s(W),H^{-s}(W)} 
= -\lim_{k\to\infty} \langle \tilde{u}^f_k(\cdot,0), \lim_{x_{n+1}\to0} x_{n+1}^{1-2s} \partial_{n+1} \tilde{h} \rangle_{\widetilde{H}^s(W), H^{-s}(W)}\\
&\qquad = \lim_{k\to\infty} \left( \int_{\R^{n+1}_+} x_{n+1}^{1-2s} w(x') \chi_{\Omega\times(0,\infty)}(x) \tilde{u}^f_k(x) dx - \int_{\R^{n+1}_+} x_{n+1}^{1-2s} \nabla \tilde{h} \cdot \tilde{a} \nabla \tilde{u}^f_k dx \right)\\
&\qquad = \lim_{k\to\infty} \int_{\Omega\times(0,\infty)} x_{n+1}^{1-2s} \tilde{u}^f_k(x',x_{n+1}) w(x') dx' dx_{n+1}\\
&\qquad = \int_{\Omega\times(0,\infty)} x_{n+1}^{1-2s} \tilde{u}^f(x',x_{n+1}) w(x') dx' dx_{n+1}\\
&\qquad = (Tf, w)_{L^2(\Omega)}.
\end{align*}
Here, for the second equality we used the weak equation for $\tilde{h}$ (see Remark 3.3 in \cite{CGRU23}) and for the third equality we made use of the observation that, as it is proven in Step 1b of the proof of Proposition 3.1 in \cite{CGRU23} (using that $\tilde{u}^f$ is a weak solution to \eqref{eq:QRA_CSExtension} and the associated decay estimates in the vertical direction),
\begin{align*}
- \int_{\R^n} \int_0^\infty x_{n+1}^{1-2s} \nabla \tilde{u}^f_k \cdot \tilde{a} \nabla \tilde{h} dx_{n+1} dx' \longrightarrow 0 \quad \text{as } k \to \infty.
\end{align*}
For the fourth equality we noted that (see also \cite{CGRU23} equation after equation (24))
\begin{align*}
\lim_{k\to\infty} \int_0^\infty x_{n+1}^{1-2s} \sigma_k(x') \eta_k(x_{n+1}) \tilde{u}^f(x',x_{n+1}) dx_{n+1} = \int_0^\infty x_{n+1}^{1-2s} \tilde{u}^f (x',x_{n+1}) dx_{n+1} \quad \text{in } H^1(\Omega).
\end{align*}

The statements about the Hilbert space adjoint operator follow from general functional analysis.

\textit{Step 3: Eigenfunction decomposition.} As a consequence of the Steps 1 and 2, the operator $T^*T: \widetilde{H}^s(W) \to \widetilde{H}^s(W)$ is a compact, self-adjoint, positive semi-definite operator. An application of the spectral theorem yields the existence of an orthonormal system (not necessarily basis) $(\varphi_j)_{j\in\N}$ of $\widetilde{H}^s(W)$ and a sequence of positive eigenvalues $(\mu_j)_{j\in\N}$ such that
\begin{align*}
T^*T \varphi_j = \mu_j \varphi_j.
\end{align*}
We define $\sigma_j := \mu_j^{1/2}$ and $\psi_j := \sigma_j^{-1} T\varphi_j \in X$. Then, the functions $(\psi_j)_{j\in\N}$ are orthonormal by definition. By another application of the qualitative Runge approximation they are dense in $X$. Indeed, let $v \in X$ be such that $(v,\psi_j)_{L^2(\Omega)} = 0$ for all $j\in\N$. We extend the orthonormal system $(\varphi_j)_{j\in\N}$ by an orthonormal basis of the null space of $T$, $(n_j^T)_{j\in\N} \subset \widetilde{H}^s(W)$, to obtain an orthonormal basis $(\varphi_j)_{j\in\N} \cup (n_j^T)_{j\in\N}$ of $\widetilde{H}^s(W)$. By density of $\operatorname{span}\left\{ (\varphi_j)_{j\in\N} \cup (n_j^T)_{j\in\N} \right\}$ in $\widetilde{H}^s(W)$, it follows that $(v,Tf)_{L^2(\Omega)} = 0$ for all $f \in \widetilde{H}^s(W)$. The density of $T$ in turn implies that $(v,w)_{L^2(\Omega)} = 0$ for all $w \in X$. The choice $w = v$ then entails that $v = 0$. Hence, the functions $(\psi_j)_{j\in\N}$ indeed form an orthonormal basis and the proof is completed.
\end{proof}

In the proof of the quantitative Runge approximation, we will once more also need a quantitative unique continuation argument, this time applied to the adjoint problem.

\begin{lem}\label{lem:UCP_adjoint_problem}
Let $\Omega \subset \R^n$ be open, non-empty, bounded and Lipschitz such that $\R^n \setminus \Omega$ is connected and let $W \subset \Omega_e$ be open, bounded, non-empty and Lipschitz. Let $\theta_1 \in (0,1)$ and assume that $a \in L^\infty(\R^n, \R^{n \times n}_{\text{sym}})$ satisfies the assumptions (A1) with the given $\theta_1$, and (A3). Let $\Omega_1, \Omega_1' \subset \R^n$ be open, bounded, Lipschitz such that $\Omega \Subset \Omega_1' \Subset \Omega_1$. Let $\tau \in (0,\frac{1}{2})$, $r_\tau \in L^2(\Omega)$ and let $\tilde{h}_\tau \in \dot{H}^1_{\text{loc}}(\R^{n+1}_+, x_{n+1}^{1-2s})$ be a weak solution (in the sense of Lemma \ref{lem:QRA_adjoint_problem}) to
\begin{equation*}
\begin{cases}
\begin{aligned}
- \nabla \cdot x_{n+1}^{1-2s} \tilde{a} \nabla \tilde{h}_\tau &= x_{n+1}^{1-2s} r_\tau \chi_{\Omega \times (0,\infty)} \quad &&\text{in } \R^{n+1}_+,\\
\lim_{x_{n+1}\to0} x_{n+1}^{1-2s} \partial_{n+1} \tilde{h}_\tau &= 0 \quad &&\text{ on } \Omega\times\{0\},\\
\tilde{h}_\tau &= 0 \quad &&\text{ on } \Omega_e\times\{0\}.
\end{aligned}
\end{cases}
\end{equation*}
Assume that
\begin{align*}
\Vert \restr{\lim_{x_{n+1}\to0} x_{n+1}^{1-2s} \partial_{n+1} \tilde{h}_\tau}{W} \Vert_{H^{-s}(W)} \leq \tau \Vert r_\tau \Vert_{L^2(\Omega)}.
\end{align*}
Then for $L \geq 2$ it holds that
\begin{align*}
\Vert \tilde{h}_\tau \Vert_{L^2((\Omega_1 \setminus \Omega_1') \times (1,L), x_{n+1}^{1-2s})} \leq C L^p \Vert r_\tau \Vert_{L^2(\Omega)} \tau^{c\alpha^{C\log(L)}},
\end{align*}
for some $\alpha \in (0,1)$. Here, the constants $C,p>0$ only depend on $n$, $s$, $\Omega$, $\Omega_1$, $\Omega_1'$, $W$ and $\theta_1$.
\end{lem}

\begin{proof}
Note that we are exactly in the setting of Proposition \ref{prop:estimate_Omega_ext}, i.e. $\tilde{h}_\tau$ satisfies
\begin{equation*}
\begin{cases}
\begin{aligned}
- \nabla \cdot x_{n+1}^{1-2s} \tilde{a} \nabla \tilde{h}_\tau &= 0 \quad &&\text{in } \Omega_e \times \R_+,\\
\tilde{h}_\tau &= 0 \quad &&\text{ on } W\times\{0\}.
\end{aligned}
\end{cases}
\end{equation*}
Moreover, by Lemma \ref{lem:QRA_adjoint_problem} it holds for some sufficiently large $\kappa>0$ and a suitable exponent $\bar{p}>0$ that
\begin{align}\label{eq:proof_UCP_adjoint_1}
\Vert \tilde{h}_\tau \Vert_{L^2(B_{\kappa L}' \times (0,\kappa L), x_{n+1}^{1-2s})} \leq C L^{\bar{p}} \Vert r_\tau \Vert_{L^2(\Omega)}
\end{align}
and by assumption
\begin{align*}
\Vert \restr{\lim_{x_{n+1}\to0} x_{n+1}^{1-2s} \partial_{n+1} \tilde{h}_\tau}{W} \Vert_{H^{-s}(W)} \leq \tau \Vert r_\tau \Vert_{L^2(\Omega)}.
\end{align*}
Then, applying the result of Proposition \ref{prop:estimate_Omega_ext} with $\Omega_{+\delta} = \Omega_1'$ (note in particular that in this case $\delta = \dist(\Omega, \partial\Omega_1') \sim O(1)$), $A = \Vert r_\tau \Vert_{L^2(\Omega)}$ and $\varepsilon = \tau$, we get
\begin{align*}
\Vert \tilde{h}_\tau \Vert_{L^2((\Omega_1\setminus\Omega_1') \times (1,L), x_{n+1}^{1-2s})} \leq C L^p \Vert r_\tau \Vert_{L^2(\Omega)} \tau^{c\alpha^{C\log(L)}}.
\end{align*}
This is the claimed estimate and the proof is finished.
\end{proof}

\subsection{Proof of the quantitative Runge approximation}\label{sec:QRA_Proof}

With the preliminary work regarding the adjoint problem in hand, we next turn to the proof of the quantitative Runge approximation result of Proposition \ref{prop:QRA}.

\begin{proof}[Proof of Proposition \ref{prop:QRA}]
Let $X$, $T$, $T'$, $T^*$, $(\sigma_j)_{j\in\N}$, $(\varphi_j)_{j\in\N}$ and $(\psi_j)_{j\in\N}$ be the space, operators, singular values and bases from Lemma \ref{lem:QRA_auxiliary} corresponding to the sets $S_1(\Omega)$ and $S_2(\Omega)$.

\textit{Step 1: Strategy.} Since $\restr{v}{\Omega} \in S_1(\Omega) \subset X$, we may write
\begin{align*}
\restr{v}{\Omega} = \sum_{j=1}^\infty \beta_j \psi_j.
\end{align*}
Define
\begin{align*}
f := \sum_{\sigma_j \geq \tau} \frac{\beta_j}{\sigma_j} \varphi_j,
\end{align*}
where $\tau>0$ small will be determined later. We will show that $v_\varepsilon^f := Tf = \restr{\int_0^\infty t^{1-2s} \tilde{u}^f(\cdot,t) dt}{\Omega}$ satisfies the desired estimates. By orthonormality we have
\begin{align}\label{eq:QRA_Proof_1}
\Vert f \Vert_{\widetilde{H}^s(W)}^2 = \sum_{\sigma_j \geq \tau} \frac{\beta_j^2}{\sigma_j^2} \leq \frac{1}{\tau^2} \sum_{\sigma_j \geq \tau} \beta_j^2 \leq \frac{1}{\tau^2} \Vert \restr{v}{\Omega} \Vert_{L^2(\Omega)}^2.
\end{align}
We set $r_\tau := \restr{v}{\Omega} - Tf = \sum_{\sigma_j < \tau} \beta_j \psi_j$ and define $\tilde{h}_\tau \in \dot{H}^1_{\text{loc}}(\R^{n+1}_+, x_{n+1}^{1-2s})$ as the weak solution (in the sense of Lemma \ref{lem:QRA_adjoint_problem}) to
\begin{equation*}
\begin{cases}
\begin{aligned}
- \nabla \cdot x_{n+1}^{1-2s} \tilde{a} \nabla \tilde{h}_\tau &= x_{n+1}^{1-2s} r_\tau \chi_{\Omega \times (0,\infty)} \quad &&\text{in } \R^{n+1}_+,\\
\lim_{x_{n+1}\to0} x_{n+1}^{1-2s} \partial_{n+1} \tilde{h}_\tau &= 0 \quad &&\text{ on } \Omega\times\{0\},\\
\tilde{h}_\tau &= 0 \quad &&\text{ on } \Omega_e\times\{0\}.
\end{aligned}
\end{cases}
\end{equation*}

We claim that for $k$ large to be chosen, it holds for some $\overline{p}_1, \overline{p}_2 > 0$ and some $\alpha \in (0,1)$ that
\begin{align}\label{eq:claim_QRA}
\Vert Tf - \restr{v}{\Omega} \Vert_{L^2(\Omega)} \leq C k^{-\overline{p}_1} \Vert v \Vert_{L^2(\Omega_1)} + C k^{\overline{p}_2} \tau^{c\alpha^{C\log(k)}} \Vert v \Vert_{H^1(\Omega_1)}.
\end{align}
We postpone the proof of \eqref{eq:claim_QRA} to Step 2. Assuming that the claim is true, we rewrite the expression on the right hand side to infer
\begin{align*}
\Vert Tf - \restr{v}{\Omega} \Vert_{L^2(\Omega)} &\leq C k^{-\overline{p}_1} \Vert v \Vert_{H^1(\Omega_1)} + C k^{\overline{p}_2} \tau^{ck^{C\log(\alpha)}} \Vert v \Vert_{H^1(\Omega_1)}\\
&= C k^{-\overline{p}_1} \Vert v \Vert_{H^1(\Omega_1)} + C k^{\overline{p}_2} \tau^{ck^{-\gamma}} \Vert v \Vert_{H^1(\Omega_1)}
\end{align*}
for some $\gamma = C \vert\log(\alpha)\vert >0$. Then choosing $k \sim \varepsilon^{-\frac{1}{\overline{p}_1}}$ and $\tau \sim e^{-\varepsilon^{-\gamma_1}}$ for some $\gamma_1 > \frac{\gamma}{\overline{p}_1}$ yields that
\begin{align*}
\Vert Tf - \restr{v}{\Omega} \Vert_{L^2(\Omega)} \leq \varepsilon \Vert v \Vert_{H^1(\Omega_1)}.
\end{align*}
Plugging $\tau$ into \eqref{eq:QRA_Proof_1} then finishes the proof.

It thus only remains to prove the claim \eqref{eq:claim_QRA}.

\textit{Step 2: Proof of claim \eqref{eq:claim_QRA}.}
In the following $p_i$ for $i\in\N$ will denote some positive constants only depending on $n$, $s$, $\Omega$, $\Omega_1$, $W$ and $\theta_1$.

First of all, using the relation between $T^*$ and $T'$ and the fact that for the adjoint operator $T^*$ it holds that $T^*\psi_j = \sigma_j \varphi_j$, we derive
\begin{equation}\label{eq:QRA_Proof_2}
\begin{aligned}
\Vert \restr{\lim_{x_{n+1}\to0} x_{n+1}^{1-2s} \partial_{n+1} \tilde{h}_\tau}{W} \Vert_{H^{-s}(W)}^2 &= \Vert T' r_\tau \Vert_{H^{-s}(W)}^2 = \Vert T^* r_\tau \Vert_{\widetilde{H}^s(W)}^2 \leq \sum_{\sigma_j < \tau} \beta_j^2 \Vert T^* \psi_j \Vert_{\widetilde{H}^s(W)}^2\\
&\leq \sum_{\sigma_j < \tau} \sigma_j^2 \beta_j^2 \Vert \varphi_j \Vert_{\widetilde{H}^s(W)}^2 \leq \tau^2 \Vert r_\tau \Vert_{L^2(\Omega)}^2.
\end{aligned}
\end{equation}
This bound will be used in our application of Lemma \ref{lem:UCP_adjoint_problem} below.

We now borrow the strategy from Step 2 in the proof of Proposition 3.1 in \cite{CGRU23}. Take $\beta_k$ as it is constructed in \cite{CGRU23}. In particular $\beta_k$ satisfies
\begin{itemize}
\item $\supp(\beta_k) \subset (k, R_{k,s}+1) \subset (k, Ck^{p_2})$ for some $p_2>0$, $R_{k,s} := k + \frac{1}{1-b_{k,s}}$, $b_{k,s} \in (0,1)$,
\item $\beta_k(t) = b_{k,s}$ for $t \in (k+1, R_{k,s})$,
\item $\vert \nabla^l \beta_k(t) \vert \leq C$ for $l \in \{0,1,2\}$ and $C$ independent of $b_{k,s} \in (0,1)$,
\item $\int_0^\infty t^{1-2s} \beta_k(t) dt = 1$.
\end{itemize}
The first property, more precisely the estimate $R_{k,s} \leq Ck^{p_2}$, is not explicitly stated in \cite{CGRU23}. We provide an argument for this in the appendix, Lemma \ref{lem:supp_beta_k}.

Define $\overline{v} := \eta v$, where $\eta \in C_c^\infty(\Omega_1'')$ is a suitable cut-off function with $\eta = 1$ in $\Omega_1'$ for $\Omega \Subset \Omega_1' \Subset \Omega_1'' \Subset \Omega_1$. We will in particular use that $\beta_k \overline{v} \in H^1_c(\R^{n+1}_+, x_{n+1}^{1-2s})$. We note that since $v$ solves $-\nabla'\cdot a \nabla'v = 0$ in $\Omega_1$, $\overline{v} = \eta v$ solves
\begin{align}\label{eq:QRA_Proof_4}
-\nabla'\cdot a \nabla' \overline{v} = - (\nabla'\cdot a \nabla'\eta) v - 2 \nabla' \eta \cdot a \nabla' v \quad \text{in } \Omega_1.
\end{align}

By the definition of $r_\tau$, the series representation of $\restr{v}{\Omega}$ and orthogonality, we infer
\begin{equation}\label{eq:QRA_Proof_5}
\begin{aligned}
\Vert Tf - \restr{v}{\Omega} &\Vert_{L^2(\Omega)}^2 = \Vert r_\tau \Vert_{L^2(\Omega)}^2 = \left( \restr{v}{\Omega}, r_\tau \right)_{L^2(\Omega)} = \left( \overline{v}, r_\tau \right)_{L^2(\Omega)}\\
& = \langle r_\tau, \overline{v} \int_0^\infty t^{1-2s} \beta_k(t) dt \rangle_{\widetilde{H}^{-1}(\Omega), H^1(\Omega)} = \langle r_\tau, \int_0^\infty t^{1-2s} \beta_k(t) \overline{v}(\cdot) dt \rangle_{\widetilde{H}^{-1}(\Omega), H^1(\Omega)}.
\end{aligned}
\end{equation}
Since $\beta_k \overline{v} \in H^1_c(\R^{n+1}_+,x_{n+1}^{1-2s})$, continuing on \eqref{eq:QRA_Proof_5}, we deduce by the weak equation for $\tilde{h}_\tau$ (see also Remark 3.3 in \cite{CGRU23})
\begin{align*}
&\Vert Tf - \restr{v}{\Omega} \Vert_{L^2(\Omega)}^2\\
&\qquad = \int_{\R^{n+1}_+} x_{n+1}^{1-2s} \nabla (\beta_k \overline{v}) \cdot \tilde{a} \nabla \tilde{h}_\tau dx - \int_{\R^n} \beta_k(0) \overline{v}(x') \left( \lim_{x_{n+1}\to0} x_{n+1}^{1-2s} \partial_{n+1} \tilde{h}_\tau(x',x_{n+1}) \right) dx'.
\end{align*}
Since due to the support of $\beta_k$ the boundary conditions on $\R^n \times \{0\}$ vanish and by the support condition on $\overline{v}$, we obtain
\begin{equation}\label{eq:QRA_Proof_6}
\begin{aligned}
&\Vert Tf - \restr{v}{\Omega} \Vert_{L^2(\Omega)}^2 = \int_{\R^n} \int_0^\infty t^{1-2s} \nabla (\beta_k \overline{v}) \cdot \tilde{a} \nabla \tilde{h}_\tau dt dx'\\
&\qquad = \left[ \int_{\Omega_1} \int_0^\infty t^{1-2s} \overline{v} (\partial_t \beta_k) (\partial_t \tilde{h}_\tau) dt dx'+ \int_{\Omega_1} a \nabla' \overline{v} \cdot \nabla'\left( \int_0^\infty t^{1-2s} \beta_k(t) \tilde{h}_\tau (x',t) dt \right) dx' \right].
\end{aligned}
\end{equation}

On the one hand, for the first term on the right hand side of \eqref{eq:QRA_Proof_6} we argue similarly as in \cite{CGRU23}, Step 2 in the proof of Proposition 3.1 (i.e., by the support of $\partial_t \beta_k$ and the decay behaviour of $\partial_t \tilde{h}_{\tau}$ in the vertical direction). More precisely, using Lemma \ref{lem:apriori_decay} with $r=\infty$ and $p=\frac{q_1}{q_1-1}$ (also recall the remark after Lemma \ref{lem:apriori_decay}), we find that
\begin{equation}\label{eq:QRA_Proof_7}
\begin{aligned}
&\left\vert \int_{\Omega_1} \int_0^\infty t^{1-2s} \overline{v} (\partial_t \beta_k) (\partial_t \tilde{h}_\tau) dt dx' \right\vert \leq \int_{(k,k+1) \cup (R_{k,s},R_{k,s}+1)} t^{1-2s} \vert \partial_t \beta_k \vert \int_{\Omega_1} \vert \overline{v} \vert \vert \partial_t \tilde{h}_\tau \vert dx'dt\\
&\qquad \leq C \Vert v \Vert_{L^1(\Omega_1)} \int_{(k,k+1) \cup (R_{k,s},R_{k,s}+1)} t^{1-2s} \vert \partial_t \beta_k \vert t^{\frac{n(q_1-1)}{q_1} -n -1} \Vert \tilde{h}_\tau(\cdot,0) \Vert_{L^{q_1}(\R^n)} dt\\
&\qquad \leq C \Vert v \Vert_{L^1(\Omega_1)} \Vert \tilde{h}_\tau(\cdot,0) \Vert_{L^{q_1}(\R^n)} \int_{(k,k+1) \cup (R_{k,s},R_{k,s}+1)} t^{n(\frac{q_1-1}{q_1} - 1) -2s} \vert \partial_t \beta_k \vert dt\\
&\qquad \leq C (k^{-p_3} + R_{k,s}^{-p_4}) \Vert \overline{v} \Vert_{L^1(\Omega_1)} \Vert \tilde{h}_\tau (\cdot,0) \Vert_{L^{q_1}(\R^n)}\\
&\qquad \leq C k^{-p_5} \Vert v \Vert_{L^2(\Omega_1)} \Vert r_\tau \Vert_{L^2(\Omega)},
\end{aligned}
\end{equation}
where $q_1 = \frac{2n}{n-2}$ if $\tilde{h}_\tau = \bar{u}$ and $q_1 = \frac{2n}{n-2s}$ if $\tilde{h}_\tau = \bar{u}_2$ ($\bar{u}$ and $\bar{u}_2$ as in Lemma \ref{lem:QRA_adjoint_problem}). Here, for the last inequality we have used that $\Vert v \Vert_{L^1(\Omega_1)} \leq C \Vert v \Vert_{L^2(\Omega_1)}$ by the boundedness of $\Omega_1$ and that $\Vert \tilde{h}_\tau(\cdot,0) \Vert_{L^{q_1}(\R^n)} \leq C \Vert r_\tau \Vert_{L^2(\Omega)}$ by Lemma \ref{lem:QRA_adjoint_problem}.

On the other hand, for the second term on the right hand side of \eqref{eq:QRA_Proof_6}, applying integration by parts and using \eqref{eq:QRA_Proof_4}, we get
\begin{equation}\label{eq:QRA_Proof_8}
\begin{aligned}
&\int_{\Omega_1} a \nabla' \overline{v} \cdot \nabla' \left( \int_0^\infty t^{1-2s} \beta_k \tilde{h}_\tau dt \right) dx' = \int_{\Omega_1} (-\nabla'\cdot a \nabla' \overline{v}) \left( \int_0^\infty t^{1-2s} \beta_k \tilde{h}_\tau dt \right) dx'\\
&\qquad =\int_{\Omega_1 \setminus \Omega_1'} \Big( -(\nabla'\cdot a \nabla'\eta) v - 2 \nabla'\eta \cdot a \nabla'v \Big) \Big( \int_0^\infty t^{1-2s} \beta_k \tilde{h}_\tau dt \Big) dx'\\
&\qquad \leq C \Vert v \Vert_{H^1(\Omega_1)} \Vert \int_0^\infty t^{1-2s} \beta_k \tilde{h}_\tau dt \Vert_{L^2(\Omega_1'' \setminus \Omega_1')},
\end{aligned}
\end{equation}
where for the first equality we do not have any boundary term since $\supp(\eta v) \Subset \Omega_1$. Then, recalling \eqref{eq:QRA_Proof_2} and applying Lemma \ref{lem:UCP_adjoint_problem} (for $\Omega_1'' \setminus \Omega_1'$ and with $L = Ck^{p_2}$), we derive after an application of Hölder's inequality with respect to the $x_{n+1}$-variable
\begin{equation}\label{eq:QRA_Proof_9}
\begin{aligned}
\Vert \int_0^\infty t^{1-2s} \beta_k \tilde{h}_\tau dt \Vert_{L^2(\Omega_1''\setminus\Omega_1')} \leq Ck^{p_6} \Vert \tilde{h}_\tau \Vert_{L^2((\Omega_1''\setminus\Omega_1') \times (k, Ck^{p_2}), x_{n+1}^{1-2s})} \leq C k^{p_7} \Vert r_\tau \Vert_{L^2(\Omega)} \tau^{c \alpha^{C \log(k)}},
\end{aligned}
\end{equation}
for some $\alpha \in (0,1)$.

Thus, combining \eqref{eq:QRA_Proof_6}, \eqref{eq:QRA_Proof_7}, \eqref{eq:QRA_Proof_8} and \eqref{eq:QRA_Proof_9}, we obtain
\begin{align*}
\Vert Tf - \restr{v}{\Omega} \Vert_{L^2(\Omega')}^2 \leq Ck^{-p_5} \Vert v \Vert_{L^2(\Omega_1)} \Vert r_\tau \Vert_{L^2(\Omega)} + Ck^{p_7} \tau^{c \alpha^{C \log(k)}} \Vert v \Vert_{H^1(\Omega_1)} \Vert r_\tau \Vert_{L^2(\Omega)}.
\end{align*}
Dividing by $\Vert Tf - \restr{v}{\Omega} \Vert_{L^2(\Omega)} = \Vert r_\tau \Vert_{L^2(\Omega)}$ proves the claim and thus finishes the proof of the proposition.
\end{proof}

Let us remark that, as we can see from the proof, the logarithmic loss in the quantitative Runge approximation is a consequence of the application of the quantitative unique continuation argument for the adjoint problem (Lemma \ref{lem:UCP_adjoint_problem}).

\section{Proof of the main results}\label{sec:proof_main_results}

To deduce the main results, we will apply the quantitative Runge approximation result, Proposition \ref{prop:QRA}, together with the quantitative unique continuation result, Proposition \ref{prop:QUCP_1}. We will first give the proof for anisotropic coefficients, Theorem \ref{thm:quantitative_reduction_anisotropic}, and then move on to the results for the isotropic setting, Theorem \ref{thm:quantitative_reduction_isotropic} and Corollary \ref{cor:stability_fract_Calderon_isotropic}.

With a slight abuse of notation, in what follows, we will write $(\cdot,\cdot)_{L^2(\partial\Omega)}$ to denote the dual pairing $\langle \cdot,\cdot \rangle_{H^{-\frac{1}{2}}(\partial\Omega), H^{\frac{1}{2}}(\partial\Omega)}$.

\subsection{Anisotropic setting}\label{sec:proof_main_results_anisotropic}

First, we recall some elementary properties of the local Dirichlet-to-Neumann map. Let $\Omega_1\subset\R^n$ be open, bounded, non-empty, Lipschitz, $g \in H^{\frac{1}{2}}(\partial\Omega_1)$ and consider the following conductivity equation for $v_a^g \in H^1(\Omega_1)$
\begin{equation}\label{eq:local_conductivity}
\begin{cases}
\begin{aligned}
-\nabla' \cdot a \nabla' v_a^g &= 0 \quad &&\text{in } \Omega_1,\\
v_a^g &= g \quad &&\text{on } \partial\Omega_1.
\end{aligned}
\end{cases}
\end{equation}
The associated Dirichlet-to-Neumann map $\Lambda_{1,\Omega_1}^a$ is defined by
\begin{align*}
\Lambda_{1,\Omega_1}^a: H^{\frac{1}{2}}(\partial \Omega_1) \to H^{-\frac{1}{2}}(\partial \Omega_1), \quad g \mapsto \partial_{\nu}^a v_a^g|_{\partial \Omega_1} := \nu \cdot a \nabla' v_a^g|_{\partial \Omega_1}  ,
\end{align*}
where $v_a^g \in H^1(\Omega_1)$ is a weak solution to \eqref{eq:local_conductivity}. For $g_1, g_2 \in H^{\frac{1}{2}}(\partial\Omega_1)$, the operator $\Lambda_{1,\Omega_1}^a$ is defined weakly by
\begin{align*}
(\Lambda_{1,\Omega_1}^a g_1, g_2)_{L^2(\partial\Omega_1)} := \int_{\Omega_1} \nabla' v_a^{g_1} \cdot a \nabla' e^{g_2} dx,
\end{align*}
where $v_a^{g_1} \in H^1(\Omega_1)$ is a weak solution to \eqref{eq:local_conductivity} with metric $a$ and Dirichlet boundary data $g_1$, and $e^{g_2} \in H^1(\Omega_1)$ is any extension of $g_2$.
 We note that this weak definition is independent of the choice of the extension. By standard arguments, we infer the following Alessandrini-type identity
\begin{align}\label{eq:Alessandrini_identity}
\left( (\Lambda_{1,\Omega_1}^{a_1} - \Lambda_{1,\Omega_1}^{a_2}) g_1, g_2 \right)_{L^2(\partial\Omega_1)} = \int_{\Omega_1} \nabla' v_{a_1}^{g_1} \cdot (a_1 -a_2) \nabla' v_{a_2}^{g_2} dx.
\end{align}

With the previous work at our disposal, we will now prove the first main result, Theorem \ref{thm:quantitative_reduction_anisotropic}.

\begin{proof}[Proof of Theorem \ref{thm:quantitative_reduction_anisotropic}]
In order to save some notation in what follows we will write $\Vert \Lambda_s^{a_1} - \Lambda_s^{a_2} \Vert_{\text{Op}}$ to denote the operator-norm $\Vert \Lambda_s^{a_1} - \Lambda_s^{a_2} \Vert_{\widetilde{H}^s(W) \to H^{-s}(W)}$.

By the Alessandrini-type identity, \eqref{eq:Alessandrini_identity}, and using that $a_1 = a_2$ in $\Omega_1 \setminus \Omega'$ we infer
\begin{equation}\label{eq:proof_main_eq1}
\begin{aligned}
&\big\vert \langle (\Lambda_{1,\Omega_1}^{a_1} - \Lambda_{1,\Omega_1}^{a_2}) g_1, g_2 \rangle_{H^{-\frac{1}{2}}(\partial\Omega_1), H^{\frac{1}{2}}(\partial\Omega_1)} \big\vert = \big\vert (\nabla' v_{a_1}^{g_1}, (a_1 - a_2) \nabla' v_{a_2}^{g_2})_{L^2(\Omega_1)} \big\vert\\
&\hspace{2cm} = \big\vert (\nabla' v_{a_1}^{g_1}, (a_1 - a_2) \nabla' v_{a_2}^{g_2})_{L^2(\Omega')} \big\vert.
\end{aligned}
\end{equation}
Let $\varepsilon > 0$ to be determined later. The quantitative Runge approximation result, Proposition \ref{prop:QRA}, implies that for $v_{a_j}^{g_j} \in H^1(\Omega_1)$, there exists $f_j \in \widetilde{H}^s(W)$ and $\tilde{u}^{f_j}_{a_j} \in \dot{H}^1(\R^{n+1}_+,x_{n+1}^{1-2s})$ solving \eqref{eq:QRA_CSExtension} with metric $\tilde{a}_j$ and Dirichlet boundary data $f_j$ such that
\begin{align*}
\left\Vert \restr{v_{a_j}^{g_j}}{\Omega} - \restr{\int_0^\infty t^{1-2s} \tilde{u}_{a_j}^{f_j}(\cdot,t) dt}{\Omega} \right\Vert_{L^2(\Omega)} &\leq \varepsilon \Vert v_{a_j}^{g_j} \Vert_{H^1(\Omega_1)},\\
\Vert f_j \Vert_{\widetilde{H}^s(W)} &\leq C_1 e^{C_1\varepsilon^{-\mu}} \Vert v_{a_j}^{g_j} \Vert_{L^2(\Omega)},
\end{align*}
for some constants $C_1 \geq 1$ and $\mu > 0$. Since $-\nabla'\cdot a_j \nabla'(v_{a_j}^{g_j} - \int_0^\infty t^{1-2s} \tilde{u}_{a_j}^{f_j}(\cdot,t) dt) =0$ in $\Omega$ (see Theorem 3 in \cite{CGRU23}), we apply Caccioppoli's inequality to infer
\begin{align}
\label{eq:appl_Cacc_loss}
\begin{split}
\left\Vert \restr{v_{a_j}^{g_j}}{\Omega'} - \restr{\int_0^\infty t^{1-2s} \tilde{u}_{a_j}^{f_j}(\cdot,t) dt}{\Omega'} \right\Vert_{H^1(\Omega')} &\leq C \left\Vert \restr{v_{a_j}^{g_j}}{\Omega} - \restr{\int_0^\infty t^{1-2s} \tilde{u}_{a_j}^{f_j}(\cdot,t) dt}{\Omega} \right\Vert_{L^2(\Omega)}\\
&\leq C \varepsilon \Vert v_{a_j}^{g_j} \Vert_{H^1(\Omega_1)}.
\end{split}
\end{align}
Continuing on \eqref{eq:proof_main_eq1} we then deduce
\begin{equation}\label{eq:proof_main_eq2}
\begin{aligned}
&\big\vert \langle (\Lambda_{1,\Omega_1}^{a_1} - \Lambda_{1,\Omega_1}^{a_2}) g_1, g_2 \rangle_{H^{-\frac{1}{2}}(\partial\Omega_1), H^{\frac{1}{2}}(\partial\Omega_1)} \big\vert\\
&\hspace{1.5cm} \leq \left| \left(\nabla' \left(\int_0^\infty t^{1-2s} \tilde{u}_{a_1}^{f_1}(\cdot,t) dt \right), (a_1 - a_2) \nabla' \left( \int_0^\infty t^{1-2s} \tilde{u}_{a_2}^{f_2} (\cdot,t) dt \right) \right)_{L^2(\Omega')} \right| \\
&\hspace{2.5cm} + C \varepsilon \Vert v_{a_1}^{g_1} \Vert_{H^1(\Omega_1)} \Vert v_{a_2}^{g_2} \Vert_{H^1(\Omega_1)}\\
&\hspace{1.5cm} \leq \left| \left(\nabla' \left(\int_0^\infty t^{1-2s} \tilde{u}_{a_1}^{f_1}(\cdot,t) dt \right), (a_1 - a_2) \nabla' \left( \int_0^\infty t^{1-2s} \tilde{u}_{a_2}^{f_2} (\cdot,t) dt \right) \right)_{L^2(\Omega)} \right| \\
&\hspace{2.5cm} + C \varepsilon \Vert g_1 \Vert_{H^{\frac{1}{2}}(\partial\Omega_1)} \Vert g_2 \Vert_{H^{\frac{1}{2}}(\partial\Omega_1)},
\end{aligned}
\end{equation}
where for the last inequality we have used that $a_1 = a_2$ in $\Omega\setminus\Omega'$ and the a-priori estimates $\Vert v_{a_j}^{g_j} \Vert_{H^1(\Omega_1)} \leq C \Vert g_j \Vert_{H^{\frac{1}{2}}(\partial\Omega_1)}$.

We seek to apply the result from Proposition \ref{prop:QUCP_1} to the previous estimate. For abbreviation we introduce the notation $v_{a_j}^{f_i} := \int_0^\infty t^{1-2s} \tilde{u}_{a_j}^{f_i} (\cdot,t) dt$ and observe that
\begin{equation}\label{eq:proof_main_eq3}
\begin{aligned}
& \left(\nabla' \left(\int_0^\infty t^{1-2s} \tilde{u}_{a_1}^{f_1}(\cdot,t) dt \right), (a_1 - a_2) \nabla' \left( \int_0^\infty t^{1-2s} \tilde{u}_{a_2}^{f_2} (\cdot,t) dt \right) \right)_{L^2(\Omega)}\\
&\qquad = (\nabla' v_{a_1}^{f_1}, a_1 \nabla' v_{a_2}^{f_2})_{L^2(\Omega)} - (\nabla' v_{a_2}^{f_1}, a_2 \nabla' v_{a_2}^{f_2})_{L^2(\Omega)}\\
& \qquad \qquad + (\nabla' v_{a_2}^{f_1}, a_2 \nabla' v_{a_2}^{f_2})_{L^2(\Omega)} - (\nabla' v_{a_1}^{f_1}, a_2 \nabla' v_{a_2}^{f_2})_{L^2(\Omega)}\\
&\qquad = (\partial_\nu^{a_1} v_{a_1}^{f_1}, v_{a_2}^{f_2})_{L^2(\partial\Omega)} - (\partial_\nu^{a_2} v_{a_2}^{f_1}, v_{a_2}^{f_2})_{L^2(\partial\Omega)} + (v_{a_2}^{f_1}, \partial_\nu^{a_2} v_{a_2}^{f_2})_{L^2(\partial\Omega)} - (v_{a_1}^{f_1}, \partial_\nu^{a_2} v_{a_2}^{f_2})_{L^2(\partial\Omega)}\\
&\qquad = (\partial_\nu^{a_1} v_{a_1}^{f_1} - \partial_\nu^{a_2} v_{a_2}^{f_1}, v_{a_2}^{f_2})_{L^2(\partial\Omega)} + (v_{a_2}^{f_1} - v_{a_1}^{f_1}, \partial_\nu^{a_2} v_{a_2}^{f_2})_{L^2(\partial\Omega)}\\
&\qquad = \left( \partial_\nu^a \left( \int_0^\infty t^{1-2s} (\tilde{u}_{a_1}^{f_1}(\cdot,t) - \tilde{u}_{a_2}^{f_1}(\cdot,t)) dt \right), \left( \int_0^\infty t^{1-2s} \tilde{u}_{a_2}^{f_2}(\cdot,t) dt \right) \right)_{L^2(\partial\Omega)}\\
&\hspace{1.5cm} + \left( \left( \int_0^\infty t^{1-2s} (\tilde{u}_{a_2}^{f_1}(\cdot,t) - \tilde{u}_{a_1}^{f_1}(\cdot,t)) dt \right), \partial_\nu^a \left( \int_0^\infty t^{1-2s} \tilde{u}_{a_2}^{f_2}(\cdot,t) dt \right) \right)_{L^2(\partial\Omega)},\\
\end{aligned}
\end{equation}
where for the second equality we used integration by parts and that $\nabla'\cdot a_j \nabla' v_{a_j}^{f_i} = 0$ in $\Omega$, and for the last equality we used that $a_1 = a_2 = a$ on $\partial \Omega$. By Proposition \ref{prop:QUCP_1} it holds that
\begin{equation}\label{eq:proof_main_eq4}
\begin{aligned}
\left\| \partial_\nu^a \left( \int_0^\infty t^{1-2s} (\tilde{u}_{a_1}^{f_1}(\cdot,t) - \tilde{u}_{a_2}^{f_1}(\cdot,t)) dt \right) \right\|_{H^{-\frac{1}{2}}(\partial\Omega)} &\leq C \vert \log( \Vert \Lambda_s^{a_1} - \Lambda_s^{a_2} \Vert_{H^{-s}(W)} ) \vert^{-\beta} \Vert f_1 \Vert_{\widetilde{H}^s(W)},\\
\left\| \left( \int_0^\infty t^{1-2s} (\tilde{u}_{a_2}^{f_1}(\cdot,t) - \tilde{u}_{a_1}^{f_1}(\cdot,t)) dt \right) \right\|_{H^{\frac{1}{2}}(\partial\Omega)} &\leq C \vert \log( \Vert \Lambda_s^{a_1} - \Lambda_s^{a_2} \Vert_{H^{-s}(W)} ) \vert^{-\beta} \Vert f_1 \Vert_{\widetilde{H}^s(W)},
\end{aligned}
\end{equation}
and by \cite{CGRU23} (see Lemma \ref{lem:integral_uniform_estimate} for the uniform dependence of the constant $C$ on stated quantities)
\begin{equation}\label{eq:proof_main_eq5}
\begin{aligned}
\left\| \int_0^\infty t^{1-2s} \tilde{u}_{a_2}^{f_2}(\cdot,t) dt \right\|_{H^{\frac{1}{2}}(\partial\Omega)} &\leq C \Vert f_2 \Vert_{\widetilde{H}^s(W)},\\
\left\| \partial_\nu^a \left( \int_0^\infty t^{1-2s} \tilde{u}_{a_2}^{f_2}(\cdot,t) dt \right) \right\|_{H^{-\frac{1}{2}}(\partial\Omega)} &\leq C \Vert f_2 \Vert_{\widetilde{H}^s(W)}.
\end{aligned}
\end{equation}

Combining \eqref{eq:proof_main_eq2}, \eqref{eq:proof_main_eq3}, \eqref{eq:proof_main_eq4} and \eqref{eq:proof_main_eq5}, we deduce
\begin{align*}
&\big\vert \langle (\Lambda_{1,\Omega_1}^{a_1} - \Lambda_{1,\Omega_1}^{a_2}) g_1, g_2 \rangle_{H^{-\frac{1}{2}}(\partial\Omega_1), H^{\frac{1}{2}}(\partial\Omega_1)} \big\vert\\
&\hspace{3mm} \leq C \vert \log( \Vert \Lambda_s^{a_1} - \Lambda_s^{a_2} \Vert_{\text{Op}} ) \vert^{-\beta} \Vert f_1 \Vert_{\widetilde{H}^s(W)} \Vert f_2 \Vert_{\widetilde{H}^s(W)} + C \varepsilon \Vert g_1 \Vert_{H^{\frac{1}{2}}(\partial\Omega_1)} \Vert g_2 \Vert_{H^{\frac{1}{2}}(\partial\Omega_1)}\\
&\hspace{3mm} \leq C e^{2C_1\varepsilon^{-\mu}} \left\vert \log\left( \Vert \Lambda_s^{a_1} - \Lambda_s^{a_2} \Vert_{\text{Op}} \right) \right\vert^{-\beta} \Vert g_1 \Vert_{H^{\frac{1}{2}}(\partial\Omega_1)} \Vert g_2 \Vert_{H^{\frac{1}{2}}(\partial\Omega_1)} + C \varepsilon \Vert g_1 \Vert_{H^{\frac{1}{2}}(\partial\Omega_1)} \Vert g_2 \Vert_{H^{\frac{1}{2}}(\partial\Omega_1)}.
\end{align*}
In order to optimize the contribution on the right hand side in terms of $\Vert \Lambda_s^{a_1} - \Lambda_s^{a_2} \Vert_{\text{Op}}$, we choose $\varepsilon = \frac{1}{2C_1} \log( \vert \log( \Vert \Lambda_s^{a_1} - \Lambda_s^{a_2} \Vert_{\text{Op}} ) \vert^{\frac{\beta}{2}})^{-\frac{1}{\mu}}$ to get
\begin{align*}
\big\vert \langle (\Lambda_{1,\Omega_1}^{a_1} - \Lambda_{1,\Omega_1}^{a_2}) &g_1, g_2 \rangle_{H^{-\frac{1}{2}}(\partial\Omega_1), H^{\frac{1}{2}}(\partial\Omega_1)} \big\vert\\
&\quad \leq C \vert \log( \Vert \Lambda_s^{a_1} - \Lambda_s^{a_2} \Vert_{\text{Op}} ) \vert^{-\frac{\beta}{2}} \Vert g_1 \Vert_{H^{\frac{1}{2}}(\partial\Omega_1)} \Vert g_2 \Vert_{H^{\frac{1}{2}}(\partial\Omega_1)}\\
&\hspace{1.5cm} + C \log( \vert \log( \Vert \Lambda_s^{a_1} - \Lambda_s^{a_2} \Vert_{\text{Op}} ) \vert)^{-\frac{1}{\mu}} \Vert g_1 \Vert_{H^{\frac{1}{2}}(\partial\Omega_1)} \Vert g_2 \Vert_{H^{\frac{1}{2}}(\partial\Omega_1)}\\
&\quad \leq C \log( \vert \log( \Vert \Lambda_s^{a_1} - \Lambda_s^{a_2} \Vert_{\text{Op}} ) \vert)^{-\frac{1}{\mu}} \Vert g_1 \Vert_{H^{\frac{1}{2}}(\partial\Omega_1)} \Vert g_2 \Vert_{H^{\frac{1}{2}}(\partial\Omega_1)}.
\end{align*}
Taking the supremum in $\Vert g_1 \Vert_{H^{\frac{1}{2}}(\partial\Omega_1)} = 1$ and $\Vert g_2 \Vert_{H^{\frac{1}{2}}(\partial\Omega_1)} = 1$, yields
\begin{align*}
\Vert \Lambda_{1,\Omega_1}^{a_1} - \Lambda_{1,\Omega_1}^{a_2} \Vert_{H^{\frac{1}{2}}(\partial\Omega_1) \to H^{-\frac{1}{2}}(\partial\Omega_1)} \leq C \log( \vert \log( \Vert \Lambda_s^{a_1} - \Lambda_s^{a_2} \Vert_{\text{Op}} ) \vert)^{-\frac{1}{\mu}},
\end{align*}
which finishes the proof.
\end{proof}

Let us briefly comment on where our assumption (A4'), that the conductivity is known and all admissible metrics are identical near the boundary, comes into play. It is used when we reduce the Alessandrini-type identity on $\Omega$ to an Alessandrini-type identity on $\Omega'$ (see equation \eqref{eq:proof_main_eq1}). As a consequence, in \eqref{eq:appl_Cacc_loss}, we can apply Caccioppoli's inequality to 
\begin{align*}
v_{a_j}^{g_j} - \int_0^\infty t^{1-2s} \tilde{u}_{a_j}^{f_j} (\cdot,t) dt 
\end{align*}
 to estimate the $H^1$-norm on $\Omega'$ by the $L^2$-norm on $\Omega$. In particular, we stay inside the set $\Omega$ in which both contributions satisfy a homogeneous equation. Since the metric is already predetermined on $\R^n \setminus \Omega$ and since we are already considering a Dirichlet-to-Neumann map for a larger set, one could also try to work with the larger set $\Omega_1 \Supset \Omega$ in the anisotropic setting here. However, in this set $\Omega_1 \setminus \Omega$, the function $\int_0^\infty t^{1-2s} \tilde{u}_{a_j}^{f_j} (\cdot,t) dt$ does not solve the homogeneous elliptic equation. Indeed, in $\Omega_1 \setminus \Omega$ we have that $\int_0^\infty t^{1-2s} \tilde{u}_{a_j}^{f_j} (\cdot,t) dt$ satisfies 
\begin{align*} 
 -\nabla'\cdot a_j \nabla' \Big( \int_0^\infty t^{1-2s} \tilde{u}_{a_j}^{f_j} (\cdot,t) dt \Big) = (-\nabla'\cdot a_j \nabla')^s u_j^f
\end{align*} 
  (see Theorem 3 in \cite{CGRU23}). Thus, in this set Caccioppoli's inequality would always involve some norm of $(-\nabla'\cdot a_j \nabla')^s u_j^f$ on $\Omega_1 \setminus \Omega$, which can only be bounded in terms of $C \Vert f \Vert_{\widetilde{H}^s(W)}$ which does not suffice for our purposes.

\subsection{Isotropic setting}\label{sec:proof_main_results_isotropic}

Applying a Liouville reduction (on the level of the local equation), in the isotropic setting, we can avoid the additional structural assumption that $a$ is known near the boundary. The arguments are similar as in the previous section. We start by recalling some elementary properties of the Dirichlet-to-Neumann map specific to the isotropic setting.

Let $\Omega_1 \subset \R^n$ be open, bounded and Lipschitz. Let $a = \gamma \Id \in L^\infty(\R^n, \R^{n \times n}_{\text{sym}})$ be isotropic conductivities, where $\gamma \in C^2(\R^n)$ with $\gamma = 1$ on $\partial\Omega_1$. We recall the Liouville reduction. For $g \in H^{\frac{1}{2}}(\partial\Omega_1)$, if $v_a^g \in H^1(\Omega_1)$ is a weak solution to
\begin{equation*}
\begin{cases}
\begin{alignedat}{2}
- \nabla'\cdot a \nabla' v_a^g &= 0 \quad &&\text{in } \Omega_1,\\
v_a^g &= g \quad &&\text{on } \partial\Omega_1,
\end{alignedat}
\end{cases}
\end{equation*}
then $w_q^g := \gamma^{\frac{1}{2}} v_a^g \in H^1(\Omega_1)$ is a solution to the Schrödinger equation with $q := \gamma^{-\frac{1}{2}} \Delta(\gamma^{\frac{1}{2}})$
\begin{equation}\label{eq:Schrödinger_eq}
\begin{cases}
\begin{alignedat}{2}
(-\Delta + q) w_q^g &= 0 \quad &&\text{in } \Omega_1,\\
w_q^g &= g \quad &&\text{on } \partial\Omega_1.
\end{alignedat}
\end{cases}
\end{equation}
Here, we already used the assumption that $\gamma = 1$ on $\partial\Omega_1$. The Dirichlet-to-Neumann map $\Lambda_{q,\Omega_1}$ associated to the Schrödinger equation is defined by
\begin{align*}
\Lambda_{q,\Omega_1}: H^{\frac{1}{2}}(\partial\Omega_1) \to H^{-\frac{1}{2}}(\partial\Omega_1), \qquad g \mapsto \restr{\partial_\nu w_q^g}{\partial\Omega_1},
\end{align*}
where $w_q^g \in H^1(\Omega_1)$ is a weak solution to \eqref{eq:Schrödinger_eq} with boundary data $g$ and $\nu$ is the outward unit normal at $\partial\Omega_1$. In weak terms, for $g_1, g_2 \in H^{\frac{1}{2}}(\partial\Omega_1)$ the operator $\Lambda_{q,\Omega_1}$ is given by
\begin{align*}
(\Lambda_{q,\Omega_1} g_1, g_2)_{L^2(\partial\Omega_1)} := \int_{\Omega_1} \nabla' w_q^{g_1} \cdot \nabla' e^{g_2} + q w_q^{g_1} e^{g_2} dx,
\end{align*}
where $w_q^{g_1} \in H^1(\Omega_1)$ is a weak solution to \eqref{eq:Schrödinger_eq} with potential $q$ and Dirichlet boundary data $g_1$, and $e^{g_2} \in H^1(\Omega_1)$ is any extension of $g_2$. As in the previous subsection, we note that this weak definition is independent of the precise choice of the extension $e^{g_2}$. In this case we obtain the following Alessandrini-type identity
\begin{align}\label{eq:Alessandrini_Id_Schrödinger_eq}
( (\Lambda_{q_1,\Omega_1} - \Lambda_{q_2,\Omega_1})g_1, g_2 )_{L^2(\partial\Omega_1)} = (w_{q_1}^{g_1}, (q_1 - q_2) w_{q_2}^{g_2})_{L^2(\Omega_1)}.
\end{align}

With these observations in hand, we now prove the main quantitative reduction result for isotropic coefficients, Theorem \ref{thm:quantitative_reduction_isotropic}. The proof is very similar to the proof in the anisotropic case from the previous section. However, noting that the Alessandrini identity does not contain a gradient term of $w_{q_j}^{g_j}$, there is no need to apply Caccioppoli's inequality and thus the assumption that $a$ is known near the boundary can be removed (in contrast to our proof of the anisotropic setting).

\begin{proof}
By the result of Proposition \ref{prop:QRA}, there exist some constants $C_1,\mu > 0$ such that for any $v_{a_j}^{g_j}$ solving
\begin{equation*}
\begin{cases}
\begin{alignedat}{2}
-\nabla'\cdot a_j \nabla' v_{a_j}^{g_j} &= 0 \quad &&\text{in } \Omega_1,\\
v_{a_j}^{g_j} &= g_j \quad &&\text{on } \partial\Omega_1,
\end{alignedat}
\end{cases}
\end{equation*}
and any $\varepsilon \in (0, \varepsilon_0)$ there exists $f_j \in \widetilde{H}^s(W)$ such that
\begin{align*}
\Vert \restr{v_{a_j}^{g_j}}{\Omega} - \restr{\int_0^\infty t^{1-2s} \tilde{u}_{a_j}^{f_j}(\cdot,t) dt}{\Omega} \Vert_{L^2(\Omega)} \leq \varepsilon \Vert v_{a_j}^{g_j} \Vert_{H^1(\Omega_1)}, \qquad \Vert f_j \Vert_{\widetilde{H}^s(W)} \leq C_1 e^{C_1\varepsilon^{-\mu}} \Vert \restr{v_{a_j}^{g_j}}{\Omega} \Vert_{L^2(\Omega)},
\end{align*}
where $\tilde{u}_{a_j}^{f_j}$ is a weak solution to \eqref{eq:QRA_CSExtension} with metric $\tilde{a}_j$ and Dirichlet data $f_j$.

Recalling that $w_{q_j}^{g_j} = \gamma_j^{\frac{1}{2}} v_{a_j}^{g_j}$, we infer as a consequence of the previous estimates
\begin{equation}\label{eq:proof_main_iso_1}
\begin{aligned}
\Vert \restr{w_{q_j}^{g_j}}{\Omega} - &\restr{\gamma_j^{\frac{1}{2}} \int_0^\infty t^{1-2s} \tilde{u}_{a_j}^{f_j} (\cdot,t) dt}{\Omega} \Vert_{L^2(\Omega)} = \Vert \restr{\gamma_j^{\frac{1}{2}} v_{a_j}^{g_j}}{\Omega} - \restr{\gamma_j^{\frac{1}{2}} \int_0^\infty t^{1-2s} \tilde{u}_{a_j}^{f_j} (\cdot,t) dt}{\Omega} \Vert_{L^2(\Omega)}\\
&\leq \theta_1^{-\frac{1}{2}} \Vert \restr{v_{a_j}^{g_j}}{\Omega} - \restr{\int_0^\infty t^{1-2s} \tilde{u}_{a_j}^{f_j} (\cdot,t) dt}{\Omega} \Vert_{L^2(\Omega)}\\
&\leq \theta_1^{-\frac{1}{2}} \varepsilon \Vert v_{a_j}^{g_j} \Vert_{H^1(\Omega_1)} ,
\end{aligned}
\end{equation}
where $\theta_1 \in (0,1)$ is the ellipticity constant from assumption (A1) such that $0 < \theta_1 \leq \gamma \leq \theta_1^{-1}$.

To shorten the notation, we will write $v_{a_j}^{f_i} := \int_0^\infty t^{1-2s} \tilde{u}_{a_j}^{f_i} (\cdot,t) dt$. By the Alessandrini-type identity, \eqref{eq:Alessandrini_Id_Schrödinger_eq}, and using that $\gamma_1 = \gamma_2$, and thus $q_1 = q_2$ in $\Omega_1 \setminus \Omega$, we infer
\begin{align*}
\vert (\Lambda_{q_1,\Omega_1} - &\Lambda_{q_2,\Omega_1}) g_1, \ g_2 )_{L^2(\partial\Omega_1)} \vert = \vert (w_{q_1}^{g_1}, \ (q_1-q_2) w_{q_2}^{g_2})_{L^2(\Omega_1)} \vert = \vert (w_{q_1}^{g_1}, \ (q_1-q_2) w_{q_2}^{g_2})_{L^2(\Omega)} \vert\\
&\leq \vert (\gamma_1^{\frac{1}{2}} \int_0^\infty t^{1-2s} \tilde{u}_{a_1}^{f_1} (\cdot,t) dt, \ (q_1-q_2) \gamma_2^{\frac{1}{2}} \int_0^\infty t^{1-2s} \tilde{u}_{a_2}^{f_2} (\cdot,t) dt)_{L^2(\Omega)} \vert\\
&\hspace{1cm} + C \theta_1^{-\frac{1}{2}} \varepsilon \Vert v_{a_1}^{g_1} \Vert_{H^1(\Omega_{1})} \Vert w_{q_2}^{g_2} \Vert_{L^2(\Omega)} + C \theta_1^{-\frac{1}{2}} \varepsilon \Vert v_{a_2}^{g_2} \Vert_{H^1(\Omega_{1})} \Vert v_{a_1}^{f_1} \Vert_{L^2(\Omega)}\\
&\leq \vert (\gamma_1^{\frac{1}{2}} v_{a_1}^{f_1}, \ (q_1-q_2) \gamma_2^{\frac{1}{2}} v_{a_2}^{f_2} )_{L^2(\Omega)} \vert + C \theta_1^{-\frac{1}{2}} \varepsilon \Vert g_1 \Vert_{H^{\frac{1}{2}}(\partial\Omega_1)} \Vert g_2 \Vert_{H^{\frac{1}{2}}(\partial\Omega_1)}.
\end{align*}
For the third inequality we used \eqref{eq:proof_main_iso_1} and that $\Vert q_j \Vert_{L^\infty(\Omega)} \leq C$ by the definition of $q_j$ and the a-priori assumptions (A1) and (A4''). For the last inequality we have used that $\Vert v_{a_j}^{f_j} \Vert_{L^2(\Omega)} \leq (1+\varepsilon) \Vert v_{a_j}^{g_j} \Vert_{H^1(\Omega_1)}$ and the a-priori estimates $\Vert v_{a_j}^{g_j} \Vert_{H^1(\Omega_1)}, \Vert w_{q_j}^{g_j} \Vert_{H^1(\Omega_1)} \leq C \Vert g_j \Vert_{H^{\frac{1}{2}}(\partial\Omega_1)}$. For the estimate $\Vert w_{q_j}^{g_j} \Vert_{H^1(\Omega_1)} \leq C \Vert g_j \Vert_{H^{\frac{1}{2}}(\partial\Omega_1)}$ note that $\Vert w_{q_j}^{g_j} \Vert_{H^1(\Omega_1)} \leq \Vert \gamma_j^{\frac{1}{2}} \Vert_{C^1(\Omega_1)} \Vert v_{a_j}^{g_j} \Vert_{H^1(\Omega_1)}$ and apply assumption (A4''). We continue to estimate for the first term on the right hand side of the previous inequality
\begin{align*}
\vert ( \gamma_1^{\frac{1}{2}} v_{a_1}^{f_1}, \ (q_1-q_2) \gamma_2^{\frac{1}{2}} &v_{a_2}^{f_2})_{L^2(\Omega)} \vert = \vert (\gamma_1^{\frac{1}{2}} v_{a_1}^{f_1}, \ q_1 \gamma_2^{\frac{1}{2}} v_{a_2}^{f_2} )_{L^2(\Omega)} - (\gamma_2^{\frac{1}{2}} v_{a_2}^{f_1}, \ q_2 \gamma_2^{\frac{1}{2}} v_{a_2}^{f_2} )_{L^2(\Omega)}\\
&\hspace{30mm}+ (\gamma_2^{\frac{1}{2}} v_{a_2}^{f_1}, \ q_2 \gamma_2^{\frac{1}{2}} v_{a_2}^{f_2} )_{L^2(\Omega)} - ( \gamma_1^{\frac{1}{2}} v_{a_1}^{f_1}, \ q_2 \gamma_2^{\frac{1}{2}} v_{a_2}^{f_2} )_{L^2(\Omega)} \vert\\
& = \left| \Big[ ( \nabla'(\gamma_1^{\frac{1}{2}} v_{a_1}^{f_1}), \ \nabla'(\gamma_2^{\frac{1}{2}} v_{a_2}^{f_2}) )_{L^2(\Omega)} + (\gamma_1^{\frac{1}{2}} v_{a_1}^{f_1}, \ q_1 \gamma_2^{\frac{1}{2}} v_{a_2}^{f_2} )_{L^2(\Omega)} \Big] \right.\\
&\hspace{10mm} - \Big[ ( \nabla'(\gamma_2^{\frac{1}{2}} v_{a_2}^{f_1}), \ \nabla'(\gamma_2^{\frac{1}{2}} v_{a_2}^{f_2}) )_{L^2(\Omega)} + (\gamma_2^{\frac{1}{2}} v_{a_2}^{f_1}, \ q_2 \gamma_2^{\frac{1}{2}} v_{a_2}^{f_2} )_{L^2(\Omega)} \Big]\\
&\hspace{10mm} + \Big[ ( \nabla'(\gamma_2^{\frac{1}{2}} v_{a_2}^{f_1}), \ \nabla'(\gamma_2^{\frac{1}{2}} v_{a_2}^{f_2}) )_{L^2(\Omega)} + (\gamma_2^{\frac{1}{2}} v_{a_2}^{f_1}, \ q_2 \gamma_2^{\frac{1}{2}} v_{a_2}^{f_2} )_{L^2(\Omega)} \Big]\\
&\hspace{10mm} \left. - \Big[ ( \nabla'(\gamma_1^{\frac{1}{2}} v_{a_1}^{f_1}), \ \nabla'(\gamma_2^{\frac{1}{2}} v_{a_2}^{f_2}) )_{L^2(\Omega)} + (\gamma_1^{\frac{1}{2}} v_{a_1}^{f_1}, \ q_2 \gamma_2^{\frac{1}{2}} v_{a_2}^{f_2} )_{L^2(\Omega)} \Big] \right|\\
& = \vert (\partial_\nu (\gamma_1^{\frac{1}{2}} v_{a_1}^{f_1}), \ \gamma_2^{\frac{1}{2}} v_{a_2}^{f_2})_{L^2(\partial\Omega)} - (\partial_\nu (\gamma_2^{\frac{1}{2}} v_{a_2}^{f_1}), \ \gamma_2^{\frac{1}{2}} v_{a_2}^{f_2})_{L^2(\partial\Omega)}\\
&\hspace{10mm} + (\gamma_2^{\frac{1}{2}} v_{a_2}^{f_1}, \ \partial_\nu (\gamma_2^{\frac{1}{2}} v_{a_2}^{f_2}))_{L^2(\partial\Omega)} - (\gamma_1^{\frac{1}{2}} v_{a_1}^{f_1}, \ \partial_\nu (\gamma_2^{\frac{1}{2}} v_{a_2}^{f_2}))_{L^2(\partial\Omega)} \vert\\
&\ = \vert ( \partial_\nu (v_{a_1}^{f_1} - v_{a_2}^{f_1}), \ v_{a_2}^{f_2} )_{L^2(\partial\Omega)} + ( (v_{a_2}^{f_1} - v_{a_1}^{f_1}), \ \partial_\nu v_{a_2}^{f_2} )_{L^2(\partial\Omega)} \vert.
\end{align*}
Here, for the first two equalities we have just added zero-terms, and for the third equality we used that $\gamma_j^{\frac{1}{2}} v_{a_j}^{f_i}$ solves the Schrödinger equation, i.e. $(-\Delta + q_j) (\gamma_j^{\frac{1}{2}} v_{a_j}^{f_i}) = 0$ in $\Omega$ since $v_{a_j}^{f_i}$ solves the conductivity equation $-\nabla'\cdot \gamma_{j} \nabla' v_{a_j}^{f_i} = 0$ in $\Omega$. For the last equality, we used that $\gamma_j \in C^2(\R^n)$ and $\gamma_j \equiv 1$ in $\R^n\setminus\Omega$ to get that $\gamma_j = 1$, $\nabla' \gamma_j = 0$ at $\partial\Omega$ and thus $\partial_\nu (\gamma_j^{\frac{1}{2}} v_{a_j}^{f_i}) = \partial_\nu v_{a_j}^{f_i}$.

By the same arguments as in the proof of Theorem \ref{thm:quantitative_reduction_anisotropic} we then infer
\begin{align*}
\Vert \Lambda_{q_1,\Omega_1} - \Lambda_{q_2,\Omega_1} \Vert_{H^{\frac{1}{2}}(\partial\Omega_1) \to H^{-\frac{1}{2}}(\partial\Omega_1)} \leq C \log( \vert \log( \Vert \Lambda_s^{a_1} - \Lambda_s^{a_2} \Vert_{\text{Op}} ) \vert)^{-\frac{1}{\mu}}.
\end{align*}
Since $\gamma_j \equiv 1$ in $\Omega_1 \setminus \Omega$ we know (by the relation between $v_a^g$ and $w_q^g$ in the Liouville reduction) that $\Lambda_{q_j,\Omega_1} = \Lambda_{1,\Omega_1}^{a_j}$ and thus
\begin{align*}
\Vert \Lambda_{q_1,\Omega_1} - \Lambda_{q_2,\Omega_1} \Vert_{H^{\frac{1}{2}}(\partial\Omega_1) \to H^{-\frac{1}{2}}(\partial\Omega_1)} = \Vert \Lambda_{1,\Omega_1}^{a_1} - \Lambda_{1,\Omega_1}^{a_2} \Vert_{H^{\frac{1}{2}}(\partial\Omega_1) \to H^{-\frac{1}{2}}(\partial\Omega_1)}.
\end{align*}
Combining this with the previous inequality finishes the proof.
\end{proof}

Lastly, we conclude the main body of this article with the proof of the novel stability estimate for the fractional isotropic Calderón problem, Corollary \ref{cor:stability_fract_Calderon_isotropic}.

\begin{proof}[Proof of Corollary \ref{cor:stability_fract_Calderon_isotropic}]
Let $\Omega_1$ be open, bounded, Lipschitz such that $\Omega \Subset \Omega_1$. Let $\Lambda_{1,\Omega_1}^{a_j}$ be the local Dirichlet-to-Neumann map as in \eqref{eq:Local_DN_map} with respect to the set $\Omega_1$, $j\in\{1,2\}$. It is well known that under the given assumptions (see \cite{A88} or \cite{Uhlmann09}) it holds that
\begin{align*}
\Vert a_1 - a_2 \Vert_{L^\infty(\Omega_1)} \leq C \vert \log( \Vert \Lambda_1^{a_1} - \Lambda_1^{a_2} \Vert_{H^{\frac{1}{2}}(\partial\Omega_1) \to H^{-\frac{1}{2}}(\partial\Omega_1)} ) \vert^{-\sigma}
\end{align*}
for some $C>0$ and $\sigma>0$. Using that $a_1 = a_2$ in $\R^n \setminus \Omega$ and applying the result of Theorem \ref{thm:quantitative_reduction_isotropic} we infer
\begin{align*}
\Vert a_1 - a_2 \Vert_{L^\infty(\R^n)} &= \Vert a_1 - a_2 \Vert_{L^\infty(\Omega_1)} \leq C \vert \log( \Vert \Lambda_1^{a_1} - \Lambda_1^{a_2} \Vert_{H^{\frac{1}{2}}(\partial\Omega_1) \to H^{-\frac{1}{2}}(\partial\Omega_1)} ) \vert^{-\sigma}\\
&\leq C \left\vert \log\left( \log\left( \big\vert \log\big( \Vert \Lambda_s^{a_1} - \Lambda_s^{a_2} \Vert_{\widetilde{H}^s(W) \to H^{-s}(W)} \big) \big\vert \right) \right) \right\vert^{-\sigma},
\end{align*}
which finishes the proof.
\end{proof}

\section*{Acknowledgements}
Both authors gratefully acknowledge funding by the Deutsche Forschungsgemeinschaft (DFG, German Research Foundation) through the CRC 1720 -- 539309657 and  under Germany’s Excellence Strategy -- EXC-2047/1 -- 390685813.

\appendix
\section{Caccioppoli's inequality}

We provide a Caccioppoli-type inequality for a variable coefficient non-homogeneous weighted elliptic equation.

\begin{lem}\label{lem:Caccioppoli_inequality}
Let $\tilde{a} \in L^\infty(\R^{n+1}_+, \R^{(n+1) \times (n+1)}_{\text{sym}})$ be uniformly elliptic with ellipticity constant $\lambda$. Let $x_0 \in \R^{n+1}_+$ and $r>0$ such that $B_{2r}(x_0) \subset \R^{n+1}_+$. Let $\tilde{h} \in L^2(B_{2r}(x_0), x_{n+1}^{1-2s})$ and $\tilde{H} \in L^2(B_{2r}(x_0), x_{n+1}^{1-2s}; \R^{n+1})$. Assume that $\tilde{u} \in H^1(B_{2r}(x_0), x_{n+1}^{1-2s})$ is a weak solution to
\begin{align}\label{eq:CS_inhomogeneous_strong}
-\nabla \cdot x_{n+1}^{1-2s} \tilde{a} \nabla \tilde{u} = x_{n+1}^{1-2s} \tilde{h} - \nabla \cdot x_{n+1}^{1-2s} \tilde{H} \quad \text{in } B_{2r}(x_0).
\end{align}
Then there exists a constant $C>1$ depending only on $\lambda$ and $\Vert \tilde{a} \Vert_{L^\infty(B_{2r}(x_0))}$ such that
\begin{align*}
\Vert x_{n+1}^{\frac{1-2s}{2}} \nabla \tilde{u} \Vert_{L^2(B_r(x_0))}^2 \leq \frac{C}{r^2} \Vert x_{n+1}^{\frac{1-2s}{2}} \tilde{u} \Vert_{L^2(B_{2r}(x_0))}^2 + C \Big( \Vert x_{n+1}^{\frac{1-2s}{2}} \tilde{h} \Vert_{L^2(B_{2r}(x_0))}^2 + \Vert x_{n+1}^{\frac{1-2s}{2}} \tilde{H} \Vert_{L^2(B_{2r}(x_0))}^2 \Big).
\end{align*}
\end{lem}

\begin{proof}
In order to save notation, we will write $B_r := B_r(x_0)$. Since $\tilde{u}$ is a weak solution to \eqref{eq:CS_inhomogeneous_strong} it holds for any $\tilde{\varphi} \in H_0^1(B_{2r}, x_{n+1}^{1-2s})$ that
\begin{equation}\label{eq:CS_inhomogeneous_weak}
\begin{aligned}
\int_{B_{2r}} x_{n+1}^{1-2s} \nabla \tilde{\varphi} \cdot \tilde{a} \nabla \tilde{u} dx &= \int_{B_{2r}} x_{n+1}^{1-2s} \tilde{h} \tilde{\varphi} dx + \int_{B_{2r}} x_{n+1}^{1-2s} \tilde{H} \cdot \nabla \tilde{\varphi} dx.
\end{aligned}
\end{equation}
Let $\eta \in C_c^\infty(B_{2r})$ be a smooth, nonnegative, radially-symmetric (centered at $x_0$) cut-off function such that $\eta = 1$ in $\overline{B_r}$, $\supp(\eta) \subset \overline{B_{2r}}$ and $\vert \nabla \eta \vert \leq \frac{C}{r}$. We test equation \eqref{eq:CS_inhomogeneous_weak} against $\tilde{\varphi} = \eta^2 \tilde{u} \in H_0^1(B_{2r}, x_{n+1}^{1-2s})$ and in order to save notation, in the following we denote by $G$ the right hand side of \eqref{eq:CS_inhomogeneous_weak} with $\tilde{\varphi} = \eta^2 \tilde{u}$
\begin{equation}\label{eq:proof_Caccioppoli_1}
\begin{aligned}
G &:= \int_{B_{2r}} x_{n+1}^{1-2s} \tilde{h} \eta^2 \tilde{u} dx + \int_{B_{2r}} x_{n+1}^{1-2s} \tilde{H} \cdot \nabla (\eta^2 \tilde{u}) dx\\
&= \int_{B_{2r}} x_{n+1}^{1-2s} \nabla( \eta^2 \tilde{u} ) \cdot \tilde{a} \nabla\tilde{u} dx = \int_{B_{2r}} x_{n+1}^{1-2s} (2 \eta \nabla\eta \tilde{u} + \eta^2 \nabla\tilde{u} ) \cdot \tilde{a} \nabla\tilde{u} dx.
\end{aligned}
\end{equation}
Using the ellipticity of $\tilde{a}$, equation \eqref{eq:proof_Caccioppoli_1}, the Cauchy-Schwarz inequality and Young's inequality we infer
\begin{equation}\label{eq:proof_Caccioppoli_2}
\begin{aligned}
\lambda \Vert \eta \nabla\tilde{u} \Vert_{L^2(B_{2r}, x_{n+1}^{1-2s})}^2 &= \lambda \int_{B_{2r}} x_{n+1}^{1-2s} \eta^2 \vert \nabla\tilde{u} \vert^2 dx \leq \int_{B_{2r}} x_{n+1}^{1-2s} \eta^2 \nabla\tilde{u} \cdot \tilde{a} \nabla\tilde{u} dx\\
&= - \int_{B_{2r}} x_{n+1}^{1-2s} (2\eta \tilde{u} \nabla\eta  \cdot \tilde{a} \nabla\tilde{u} ) dx + G\\
&\leq 2 \Vert \tilde{a} \Vert_{L^\infty(B_{2r})} \Vert \eta \nabla\tilde{u} \Vert_{L^2(B_{2r},x_{n+1}^{1-2s})} \Vert \tilde{u} \nabla\eta \Vert_{L^2(B_{2r}, x_{n+1}^{1-2s})} + \vert G \vert\\
&\leq \frac{\lambda}{2} \Vert \eta \nabla\tilde{u} \Vert_{L^2(B_{2r}, x_{n+1}^{1-2s})}^2 + \frac{(2 \Vert \tilde{a} \Vert_{L^\infty(B_{2r})})^2}{2\lambda} \Vert \tilde{u} \nabla\eta \Vert_{L^2(B_{2r}, x_{n+1}^{1-2s})}^2 + \vert G \vert.
\end{aligned}
\end{equation}
We estimate $\vert G \vert$ by means of the Cauchy-Schwarz inequality and Young's inequality
\begin{equation}\label{eq:proof_Caccioppoli_3}
\begin{aligned}
\vert G \vert &\leq C \Big( \Vert \tilde{h} \Vert_{L^2(B_{2r}, x_{n+1}^{1-2s})}^2 + \Vert \tilde{u} \Vert_{L^2(B_{2r}, x_{n+1}^{1-2s})}^2 + \frac{1}{\lambda} \Vert \tilde{H} \Vert_{L^2(B_{2r}, x_{n+1}^{1-2s})}^2 + \frac{\lambda}{4} \Vert \eta \nabla\tilde{u} \Vert_{L^2(B_{2r}, x_{n+1}^{1-2s})}^2\\
&\hspace{15mm} + \Vert \tilde{H} \Vert_{L^2(B_{2r}, x_{n+1}^{1-2s})}^2 + \Vert \nabla(\eta^2) \tilde{u} \Vert_{L^2(B_{2r}, x_{n+1}^{1-2s})}^2 \Big).
\end{aligned}
\end{equation}
Putting \eqref{eq:proof_Caccioppoli_2} and \eqref{eq:proof_Caccioppoli_3} together, rearranging terms and recalling the conditions for $\eta$ then concludes the argument.
\end{proof}

\section{Auxiliary results for the quantitative Runge approximation}

In this section, we discuss several auxiliary results used  in the proof of the quantitative Runge approximation in Section \ref{sec:QRA_Proof}.  We emphasize that this is purely for the convenience of the reader and that we do not claim any novelty on these results.

We begin by noting that, as a consequence of the results for the adjoint problem, in particular, of Lemma \ref{lem:QRA_adjoint_problem}, we can derive the following uniform bound for the integrated quantity $\int_0^\infty t^{1-2s} \tilde{u}^f (\cdot,t) dt$.

\begin{lem}\label{lem:integral_uniform_estimate}
Let $n \geq 3$, let $\Omega \subset \R^n$ be open, non-empty, bounded and Lipschitz and $W \subset \Omega_e$ be open, non-empty, bounded and Lipschitz. Let $\theta_1 \in (0,1)$ and assume that $a \in L^\infty(\R^n, \R^{n \times n}_{\text{sym}})$ satisfies the assumptions (A1) with the given $\theta_1$, and (A3). Let $f \in \widetilde{H}^s(W)$ and let $\tilde{u}^f \in \dot{H}^1_{\text{loc}}(\R^{n+1}_+,x_{n+1}^{1-2s})$ be the unique solution to
\begin{equation}
\begin{cases}
\begin{alignedat}{2}
-\nabla\cdot x_{n+1}^{1-2s} \tilde{a} \nabla \tilde{u}^f &= 0 \quad &&\text{in } \R^{n+1}_+,\\
\lim_{x_{n+1}\to0} x_{n+1}^{1-2s} \partial_{n+1} \tilde{u}^f &= 0 \quad &&\text{on } \Omega \times \{0\},\\
\tilde{u}^f &= f \quad &&\text{on } \Omega_e \times \{0\}.
\end{alignedat}
\end{cases}
\end{equation}
Then, there exists a constant $C>0$ only depending on $n$, $s$, $\Omega$, $W$ and $\theta_1$ such that
\begin{align*}
\Vert \int_0^\infty t^{1-2s} \tilde{u}^f(\cdot,t) dt \Vert_{H^1(\Omega)} \leq C \Vert f \Vert_{\widetilde{H}^s(W)}.
\end{align*}
\end{lem}

\begin{proof}
Let us write
\begin{align*}
\Vert \int_0^\infty t^{1-2s} \tilde{u}^f(\cdot,t) dt \Vert_{H^1(\Omega)} = \sup_{\substack{\psi \in \widetilde{H}^{-1}(\Omega)\\ \Vert \psi \Vert_{\widetilde{H}^{-1}(\Omega)} = 1}} \langle \psi, \int_0^\infty t^{1-2s} \tilde{u}^f(\cdot,t) dt \rangle_{\widetilde{H}^{-1}(\Omega), H^1(\Omega)}.
\end{align*}
Let $\tilde{h}_\psi$ be as in Lemma \ref{lem:QRA_adjoint_problem} for $w = \psi$ and let $\tilde{u}^f_k \in H_c^1(\R^{n+1}_+, x_{n+1}^{1-2s})$ be defined as in the proof of Lemma \ref{lem:QRA_auxiliary}, i.e. $\tilde{u}^f_k(x) := \tilde{u}^f(x) \sigma_k(x') \eta_k(x_{n+1})$ for some cut-off functions $\eta_k$ and $\sigma_k$. By the weak equation for $\tilde{h}_\psi$ we find that (cf. equation (23) in the proof of Proposition 3.1 in \cite{CGRU23})
\begin{align*}
&\langle \psi, \int_0^\infty t^{1-2s} \tilde{u}^f(\cdot,t) dt \rangle_{\widetilde{H}^{-1}(\Omega), H^1(\Omega)}\\
&\hspace{20mm} = \lim_{k\to \infty} \left( - \int_W f \left( \lim_{t\to0} t^{1-2s} \partial_t \tilde{h}_\psi \right) dx' + \int_{\R^{n+1}_+} t^{1-2s} \nabla \tilde{h}_\psi \cdot \tilde{a} \nabla \tilde{u}^f_k dx \right)
\end{align*}
Using the weak equation for $\tilde{u}^f$ we observe (as in Step 1b of the proof of Proposition 3.1 in \cite{CGRU23}) that
\begin{align*}
\lim_{k\to\infty} \int_{\R^{n+1}_+} t^{1-2s} \nabla \tilde{h}_\psi \cdot \tilde{a} \nabla \tilde{u}^f_k dx = 0.
\end{align*} 
Let $W_1$ be open, bounded, Lipschitz such that $W \Subset W_1$. As a consequence of the above we infer
\begin{align*}
\langle \psi, \int_0^\infty t^{1-2s} &\tilde{u}^f(\cdot,t) dt \rangle_{\widetilde{H}^{-1}(\Omega), H^1(\Omega)} = - \int_W f \left( \lim_{t\to0} t^{1-2s} \partial_t \tilde{h}_\psi \right) dx'\\
&\leq \Vert f \Vert_{\widetilde{H}^s(W)} \Vert \lim_{t\to0} t^{1-2s} \partial_t \tilde{h}_\psi \Vert_{H^{-s}(W)} \leq C \Vert f \Vert_{\widetilde{H}^s(W)} \Vert \tilde{h}_\psi \Vert_{H^1(W_1 \times (0,1), x_{n+1}^{1-2s})}\\
&\leq C(n,s,\Omega,W,\theta_1) \Vert f \Vert_{\widetilde{H}^s(W)} \Vert \psi \Vert_{\widetilde{H}^{-1}(\Omega)},
\end{align*}
where for the last inequality we have applied the a-priori estimate from Lemma \ref{lem:QRA_adjoint_problem}. Taking the supremum over $\psi \in \widetilde{H}^{-1}(\Omega)$ with $\Vert \psi \Vert_{\widetilde{H}^{-1}(\Omega)} = 1$ finally yields
\begin{align*}
\Vert \int_0^\infty t^{1-2s} \tilde{u}^f(\cdot,t) dt \Vert_{H^1(\Omega)} \leq C(n,s,\Omega,W,\theta_1) \Vert f \Vert_{\widetilde{H}^s(W)}
\end{align*}
and the proof is finished.
\end{proof}

Next, we provide an argument for the support condition of the function $\beta_k$ in the proof of the quantitative Runge approximation, Proposition \ref{prop:QRA}, which was only discussed qualitatively in \cite{CGRU23}.

\begin{lem}\label{lem:supp_beta_k}
Let $\beta_k: \R \to \R$ be defined as in Step 2 of the proof of Proposition \ref{prop:QRA}. Then, there exist constants $C>0$ and $p_2>0$ such that $\supp(\beta_k) \subset (k, Ck^{p_2})$.
\end{lem}

\begin{proof}
We recall from \cite{CGRU23} that $\beta_k$ was defined as $\beta_k(t) := \gamma_{b_{k,s}}(t-k)$, where $\gamma_b \in C^\infty(\R)$ satisfies $\supp(\gamma_b) \subset [0,\frac{2-b}{1-b}]$, $0 \leq \gamma_b(t) \leq b$, $\gamma_b(t) = b$ for $t \in [1,\frac{1}{1-b}]$, $\vert \nabla^l \gamma_b \vert \leq C$ for $l \in \{0,1,2\}$, and $\int_0^\infty \gamma_b(t) dt = \frac{b}{1-b}$. Here, we always have $b\in(0,1)$.

We know that $\supp(\beta_k) \subset (k, R_{k,s}+1)$, where $R_{k,s} := k + \frac{1}{1-b_{k,s}}$. In particular, we seek to prove that $R_{k,s} = k + \frac{1}{1-b_{k,s}} \leq Ck^{p_2}$ for some $p_2>0$. The defining equation for the choice of $b_{k,s} \in (0,1)$ is
\begin{align*}
\int_0^\infty t^{1-2s} \beta_k(t) dt = 1.
\end{align*}
Using the definition of $\beta_k$, we need to choose $b = b_{k,s}$ such that
\begin{align*}
1 = \int_0^{\frac{2-b}{1-b}} (t+k)^{1-2s} \gamma_b(t) dt =: I_{b,k}.
\end{align*}
We make a case distinction.

\textit{Case 1: $s \in (0,\frac{1}{2}]$.} We estimate $I_{b,k}$ from below by $I_{b,k} \geq k^{1-2s} \frac{b}{1-b}$. Since the choice of $b_{k,s}$ needs to cancel out the largeness of (at least) $k^{1-2s}$ and since $\frac{b}{1-b}$ is increasing as $b \to 1$, (for large $k$) we can assure that $1-b_{k,s} \geq \frac{1}{2}$. Consequently, it indeed holds true that $R_{k,s} = k+\frac{1}{1-b_{k,s}} \leq Ck^{p_2}$ for some $p_2>0$.

\textit{Case 2: $s \in (\frac{1}{2},1)$.} We argue similarly. This time, we observe that $I_{b,k}$ is bounded from below by $I_{b,k} \geq (\frac{2-b}{1-b} + k)^{1-2s} \frac{b}{1-b}$. We seek to choose $b_{k,s}$ (which will now be close to $1$) such that it cancels out the smallness of $(\frac{2-b}{1-b} + k)^{1-2s}$. The choice $b = \frac{k}{k+1}$ yields that (for large $k$) it holds $(\frac{2-b}{1-b} + k)^{1-2s} \frac{b}{1-b} \geq 1$. In particular, $b_{k,s}$ can be chosen such that $b_{k,s} \leq \frac{k}{k+1}$ and thus $1-b_{k,s} \geq \frac{1}{k+1}$ (for large $k$). This again verifies that $R_{k,s} = k+\frac{1}{1-b_{k,s}} \leq Ck^{p_2}$ for some $p_2>0$ and the proof is finished.
\end{proof}

\section{The heat representation for the Poisson kernel}

In this section, we recall the heat representation for the Poisson kernel. As in the previous sections of the appendix, we do not claim any novelty at this point but refer to the article \cite{ST10} and subsequent work instead. As the regularity properties are not stated explicitly in the form of standard Sobolev spaces there, we opted to include the following statement for completeness although it is certainly well-known.

\begin{prop}[Heat kernel representation]
Let $\Omega \subset \R^n$ be open, non-empty, bounded and Lipschitz and let $W \subset \Omega_e$ be open, bounded, non-empty and Lipschitz. Let $\theta_1\in(0,1)$, $\theta_2>0$. Let $a \in L^\infty(\R^n, \R^{n \times n}_{\text{sym}})$ satisfy the assumptions (A1) and (A2) with the given $\theta_1$ and $\theta_2$, respectively. Let $f \in \widetilde{H}^{s}(W)$, $u \in H^s(\R^n)$ be the solution to \eqref{eq:Fractional_Calderon_equation} with Dirichlet data $f$, and let $\tilde{u}^f \in \dot{H}^1(\R^{n+1}_+, x_{n+1}^{1-2s})$ be the Caffarelli-Silvestre type extension of $u$ (see \eqref{eq:CS_intro} and \eqref{eq:CS}). 
Let $K_t(\cdot,\cdot)$ denote the heat kernel associated with the equation
\begin{equation*}
\begin{cases}
\begin{alignedat}{2}
(\p_t - \nabla' \cdot a \nabla') w & = 0 \quad &&\text{in } \R^n \times (0,\infty),\\
w(\cdot, 0) &= w_0 \quad &&\text{on } \R^n.
\end{alignedat}
\end{cases}
\end{equation*}
Then, the following representation holds: There exists a constant $c_s\neq 0$ such that for all $x' \in \R^n, y\in \R_+$
\begin{align}
\label{eq:representation}
\tilde{u}^f(x',y) = c_s y^{2s} \int\limits_{\R^n} \int\limits_{0}^{\infty} K_t(x', z') e^{-\frac{y^2}{4t}} \frac{dt}{t^{1+s}} u(z') dz'.
\end{align}
\end{prop}

\begin{proof}
We begin by outlining the strategy of the proof.

\emph{Strategy.}
We first assume that $u \in C_c^{\infty}(\R^n)$ and, in particular, $\supp(u)$ is compact.
The proof then follows from two steps. Similarly as in \cite[Lemma 4.26]{FGKRSU25} by a finite volume approximation, we obtain that for $x' \in \R^{n}\setminus \supp(u)$ and $y \in \R_+$ the desired representation holds. In a second step, we use that the results from \cite{ST10} imply that on $\R^{n+1}_+$
\begin{align*}
v(x',y):=c_s y^{2s} \int\limits_{\R^n} \int\limits_{0}^{\infty} K_t(x', z') e^{-\frac{y^2}{4t}} \frac{dt}{t^{1+s}} u(z') dz'
\end{align*}
solves the same equation (with the same boundary data) as the function $\tilde{u}^f(x',y)$. As the two functions coincide on the set $\big( \R^n \setminus \supp(u) \big) \times \R_+$ (which contains an open set), by unique continuation \cite{FF14,R15,Y17}, the two functions then agree globally, providing the desired representation formula. Finally, for $u \in H^s(\R^n)$ the argument follows by approximation and continuity.

As there are slight differences in the setting from \cite[Lemma 4.26]{FGKRSU25} and our present set-up (Dirichlet vs Neumann boundary data), we provide some further details on the first step of the argument. The argument for this first step in turn is split into several substeps.\\

\emph{Step 1: Representation formula in the complement of $\supp(u)$.}
We begin by a finite volume approximation argument, as here the representation is explicit in terms of suitable (Dirichlet) eigenfunctions in the tangential directions and an ODE in the vertical direction.

\emph{Step 1a: Finite volume approximation.}
Let $R>0$ be so large that $u= 0$ in $\R^n \setminus B_R'(0)$. We then consider the following auxiliary problem:
\begin{equation}\label{eq:finite_vol}
\begin{cases}
\begin{alignedat}{2}
-\nabla \cdot x_{n+1}^{1-2s} \tilde{a} \nabla \tilde{u}_R & = 0 \quad &&\text{in } B_R'(0) \times \R_+,\\
\tilde{u}_R & = 0 \quad &&\text{on } \partial B_R'(0) \times \R_+,\\
\tilde{u}_R & = u \quad &&\text{on }  B_R'(0) \times \{0\}.
\end{alignedat}
\end{cases}
\end{equation}
As above, $\tilde{a} \in L^\infty(\R^{n+1}_+, \R^{(n+1)\times (n+1)}_{\text{sym}})$ is given by
\begin{align*}
\tilde{a}(x') := \begin{pmatrix} a(x') & 0 \\ 0 & 1\end{pmatrix}.
\end{align*}
The representation of the function $\tilde{u}_R$ can be obtained by an eigenfunction expansion in the tangential eigenfunctions $\psi_k$, $k \in \N$, of 
\begin{equation*}
\begin{cases}
\begin{alignedat}{2}
-\nabla' \cdot a \nabla' \psi_k & = \lambda_k \psi_k \quad &&\text{in } B_R'(0),\\
\psi_k & = 0 \quad &&\text{on } \partial B_R'(0),
\end{alignedat}
\end{cases}
\end{equation*}
i.e., $\tilde{u}_R(x', x_{n+1})= \sum\limits_{k \in \N} \gamma_k(x_{n+1}) \psi_k(x')$,
with the coefficients $\gamma_k(x_{n+1})$ being determined as the solution of the remaining vertical ODE 
\begin{equation*}
\begin{cases}
\begin{alignedat}{2}
\p_{n+1} x_{n+1}^{1-2s} \p_{n+1} \gamma_k - x_{n+1}^{1-2s} \lambda_k \gamma_k & = 0 \quad &&\text{in } (0,\infty),\\
\gamma_k(0) & = \gamma_{k,0},
\end{alignedat}
\end{cases}
\end{equation*}
where $\gamma_{k,0}$ are the coefficients from the expansion of $u(x'):= \sum\limits_{k \in \N} \gamma_{k,0} \psi_k(x')$. Observing that the vertical equation can be brought into the form of a modified Bessel equation, 
this then gives rise to the representation
\begin{align*}
\tilde{u}_R(x',x_{n+1}) = c_{s} \sum\limits_{k=1}^{\infty} (u,\psi_k)_{L^2(B_R'(0))} (\sqrt{\lambda_k} x_{n+1})^s K_s(\sqrt{\lambda_k} x_{n+1}) \psi_k(x'),
\end{align*}
where $K_s$ denotes the modified Bessel function of the second kind and $c_s \neq 0$.

We next claim that for $x' \in B_R'(0), x_{n+1}>0$, the function $\tilde{u}_R(x', x_{n+1})$ has a representation of the form 
\begin{align}
\label{eq:heat_repr1}
\tilde{u}_R(x', x_{n+1})= c_s x_{n+1}^{2s} \int\limits_{\R^n} \int\limits_{0}^{\infty} K_t^R(x', z') e^{-\frac{y^2}{4t}} \frac{dt}{t^{1+s}} u(z') dz',
\end{align}
where $K_t^R(x', z')$ denotes the heat kernel associated with the operator $\nabla' \cdot a \nabla'$ on $B_R'(0)$ with Dirichlet boundary data. In order to observe the representation \eqref{eq:heat_repr1}, we expand the heat kernel into the eigenfunctions $\psi_k$, i.e., with convergence in $H^1(B_R'(0))$,
\begin{align*}
K_t^R(x', z') = \sum\limits_{k=1}^{\infty} e^{-\lambda_k t} \psi_k(x') \psi_k(z'), \quad x', z' \in \R^n,
\end{align*}  
to infer that
\begin{align*}
x_{n+1}^{2s} \int_{\R^n} &\int_{0}^{\infty} K_t^R(x', z') e^{-\frac{x_{n+1}^2}{4t}} \frac{dt}{t^{1+s}} u(z') dz'\\
& =  \sum\limits_{k=1}^{\infty} (u,\psi_k)_{L^2(B_R'(0))} x_{n+1}^{2s}\int_{0}^{\infty} e^{-\lambda_k t} e^{-\frac{x_{n+1}^2}{4t}} \frac{dt}{t^{1+s}} \psi_k(x')\\
& = 2^{1+s}  \sum\limits_{k=1}^{\infty} (u,\psi_k)_{L^2(B_R'(0))} (\sqrt{\lambda_k} x_{n+1})^s K_s(\sqrt{\lambda_k} x_{n+1}) \psi_k(x')\\
 & = 2^{1+s}c_s^{-1} \tilde{u}_R(x', x_{n+1}).
\end{align*}
Here we have used Fubini's theorem and the $L^2$ convergence of the series for $u$ together with the identity 
\begin{align*}
x_{n+1}^{2s}\int\limits_{0}^{\infty} e^{-\lambda_k t} e^{-\frac{x_{n+1}^2}{4t}} \frac{dt}{t^{1+s}} = 2^{1+s} (\sqrt{\lambda_k} x_{n+1})^s K_s(\sqrt{\lambda_k} x_{n+1}),
\end{align*}
which in turn is a consequence of the fact that 
\begin{align*}
K_s(z) = 2^{-1-s} z^{s} \int\limits_{0}^{\infty} e^{- t} e^{-\frac{z^2}{4t}} \frac{dt}{t^{1+s}} 
\end{align*}
and a change of variables (see, formula (10.32.10) in \cite{Olver10}).

Moreover, we complement the representation \eqref{eq:heat_repr1} by energy estimates. Testing the weak form of the equation \eqref{eq:finite_vol} with an extension of $\tilde{u}_R$ by zero across $\partial B_R'(0) \times \R_+$ implies that
\begin{align}
\label{eq:energy_bdd}
\|x_{n+1}^{\frac{1-2s}{2}} \nabla \tilde{u}_R\|_{L^2(B_R'(0) \times (0,\infty))} \leq C \|u\|_{\dot{H}^s(\R^n)}.
\end{align}

\emph{Step 1b: Limit.} In the second step of the finite volume approximation, we pass to the limit $R \rightarrow \infty$ in the representation \eqref{eq:heat_repr1} using the energy bound \eqref{eq:energy_bdd}. Indeed, by the regularity of $f \in \widetilde{H}^s(W)$, we have that $u \in \dot{H}^s(\R^n)$. As a result of \eqref{eq:energy_bdd}, $\tilde{u}_R$ satisfies uniform bounds in $\dot{H}^{1}(\R^{n+1}_+, x_{n+1}^{1-2s})$. Then, by weak convergence, $\tilde{u}_R \rightarrow \tilde{u}$ along some subsequence and the weak equation passes to its limit:
\begin{align*}
\int\limits_{\R^{n+1}_+} x_{n+1}^{1-2s} \tilde{a} \nabla \tilde{u} \cdot \nabla \tilde{\varphi} dx = \lim\limits_{R_k \rightarrow \infty} 
\int\limits_{\R^{n+1}_+} x_{n+1}^{1-2s} \tilde{a} \nabla \tilde{u}_{R_k} \cdot \nabla \tilde{\varphi} dx = 0,
\end{align*}
for $\tilde{\varphi} \in \dot{H}^1(\R^{n+1}_+, x_{n+1}^{1-2s})$. Moreover, as  $\tilde{u}_{R_k}|_{B_{R_k}(0) \times \{0\}} = u$, it also follows that $\tilde{u}|_{\R^n \times \{0\}} = u$.

It remains to argue that the representation \eqref{eq:representation} holds true. This follows from the maximum principle for the heat equation which implies that $K_t^R(x',z') \rightarrow K_t(x', z')$ uniformly for $x', z' \in \R^n$ with $x'\neq z'$. Hence, the desired representation on $\big( \R^n \setminus \supp(u) \big) \times \R_+$ follows from dominated convergence.
\end{proof}

\bibliographystyle{alpha}
\bibliography{bibliography}

\end{document}